\documentclass[10pt, a4paper, notitlepage]{article}
\pdfoutput=1

\usepackage{biblatex}
\usepackage{bbm}

\usepackage{etoolbox}
\usepackage{tikz}
\usetikzlibrary{calc}
\usetikzlibrary{cd}
\usetikzlibrary{decorations.markings}
\usetikzlibrary{decorations.pathreplacing}
\usetikzlibrary{decorations.pathmorphing}
\usetikzlibrary{decorations.text}
\usetikzlibrary{arrows.meta}
\usetikzlibrary{arrows}
\usetikzlibrary{positioning}

\usepackage{geometry} 
\usepackage{tocloft} 
\usepackage{fancyhdr} 
\usepackage{longtable}
\usepackage{subcaption}
\usepackage{enumitem}
\usepackage{amssymb}
\usepackage{amsmath}
\usepackage{stackrel} 
\usepackage{uniinput}
\usepackage{amsthm}
\usepackage{stmaryrd}
\usepackage{mathtools}
\usepackage{fouridx}
\usepackage{xparse}
\usepackage{bm}
\usepackage{aliascnt}
\usepackage{import}

\usepackage{hyperref}
\theoremstyle{definition}
\newtheorem{theorem}{Theorem}
\let\emph\textbf

\newaliascnt{remark}{theorem}
\newaliascnt{definition}{theorem}
\newaliascnt{proposition}{theorem}
\newaliascnt{lemma}{theorem}
\newaliascnt{corollary}{theorem}
\newaliascnt{example}{theorem}
\newaliascnt{convention}{theorem}
\newaliascnt{question}{theorem}
\newaliascnt{stageplan}{theorem}
\newaliascnt{todo}{theorem}

\newtheorem{remark}[remark]{Remark}
\newtheorem{definition}[definition]{Definition}
\newtheorem{proposition}[proposition]{Proposition}
\newtheorem{lemma}[lemma]{Lemma}

\newtheorem{example}[example]{Example}

\newtheorem{question}[question]{Question}
\newtheorem{stageplan}[stageplan]{Stage Plan}

\aliascntresetthe{remark}
\aliascntresetthe{definition}
\aliascntresetthe{proposition}
\aliascntresetthe{lemma}
\aliascntresetthe{corollary}
\aliascntresetthe{example}
\aliascntresetthe{convention}
\aliascntresetthe{question}
\aliascntresetthe{stageplan}
\aliascntresetthe{todo}

\numberwithin{equation}{section}
\numberwithin{figure}{section}
\numberwithin{table}{section}
\makeatletter\let\c@table\c@figure\makeatother

\numberwithin{theorem}{section}
\numberwithin{remark}{section}
\numberwithin{definition}{section}
\numberwithin{proposition}{section}
\numberwithin{lemma}{section}
\numberwithin{corollary}{section}
\numberwithin{example}{section}
\numberwithin{convention}{section}
\numberwithin{question}{section}
\numberwithin{stageplan}{section}
\numberwithin{todo}{section}

\newcommand{\tinymath}[1]{\scalebox{0.7}{#1}}

\newcommand{\Id}{\operatorname{Id}}
\newcommand{\id}{\mathrm{id}}

\newcommand{\Hom}{\operatorname{Hom}}

\newcommand{\D}{\operatorname{D}}

\newcommand{\pmat}[1]{\begin{pmatrix} #1 \end{pmatrix}}
\newcommand{\vspan}{\operatorname{span}}

\newcommand{\isoto}{\xrightarrow{\sim}}

\def\Im{\operatorname{Im}}

\newcommand{\running}{~|~}

\newcommand{\Stab}{\operatorname{Stab}}

\newcommand{\Proj}{\operatorname{Proj}}

\newcommand{\vdim}{\operatorname{dim}}

\newcommand{\Rep}{\operatorname{Rep}}
\newcommand{\qQ}{\overline{Q}}

\newcommand{\GL}{\operatorname{GL}}
\newcommand{\SL}{\operatorname{SL}}

\newcommand{\Vect}{\mathsf{Vec}}
\newcommand{\open}{\circ}
\newcommand{\Sym}{\operatorname{Sym}}
\newcommand{\closure}[1]{\overline{#1}}
\newcommand{\trace}{\operatorname{tr}}
\newcommand{\BD}{\operatorname{BD}}
\newcommand{\Qbar}{\overline{Q}}
\newcommand{\surjects}{\twoheadrightarrow}
\newcommand{\CG}{\mathrm{CG}}

\newcommand{\Coh}{\operatorname{Coh}}

\title{Resolution of singularities via Tannaka duality}
\author{Jasper van de Kreeke}
\date{}

\begin{document}

\maketitle

\begin{abstract}
Resolving finite quotient singularities is a classical problem in algebraic geometry. Traditional methods of Geometric Invariant Theory (GIT) translate the singularity into a quiver representation space and take the GIT quotient with respect to a generic stability parameter. While this approach easily produces smooth resolutions, it fails to produce any stacky resolutions, as quiver representation spaces lack finite stabilizers.

This paper provides an alternative framework which produces both smooth and stacky resolutions. Our framework is based on a trick of Abdelgadir and Segal, which deploys Tannaka duality to describe the points of the classifying stack of a finite group in terms of algebraic data. Abdelgadir and Segal successfully pursue this strategy and obtain smooth and stacky resolutions in the Kleinian $ D_4 $ case. We generalize this strategy to all Kleinian singularities and obtain a series of varieties which we refer to as “Clebsch-Gordan varieties”. We provide tools to work with these Clebsch-Gordan varieties, analyze their stable loci with respect to different stability parameters, and study the Kleinian $ A_n $ and $ D_n $ cases in detail.
\end{abstract}

\tableofcontents

\section{Introduction}
Resolving singularities is one of the core tasks in algebraic geometry, and multiple strategies exist to accompalish it. Finite quotient singularities $ X = V \sslash Γ $ are a particular class of singularities and a popular strategy to resolve them \cite{cb-kleinian, king-quivers, bellamy-schedler} consists of turning $ X $ into a quiver variety:
\begin{center}
\begin{tikzpicture}
\path (0, 0) node[align=center] (A) {$ U_0, …, U_n ∈ \Rep(Γ) $ \\ \emph{Simple representations}};
\path (5, 0) node[align=center] (B) {$ φ_i: U_i ¤ V → \bigoplus_k U_k^{⊕ c_{ik}} $ \\ \emph{Linear maps}};
\path (10, 0) node[align=center] (C) {$ X ≅ \Rep(Π_Q, α) \sslash \GL_α $ \\ \emph{Quiver variety}};
\path ($ (A.east)!0.5!(B.west) $) node {\Large $ \rightsquigarrow $};
\path ($ (B.east)!0.5!(C.west) $) node {\Large $ \rightsquigarrow $};
\end{tikzpicture}
\end{center}
Adding a stability parameter $ θ ∈ ℤ^α $ produces a map $ \Rep(Π_Q, α) \sslash_θ \GL_α \surjects \Rep(Π_Q, α) \sslash \GL_α $ which typically is a resolution of singularities. However, all these classical quotients are preprojective varieties and one never obtains the stacky resolution $ [V / Γ] $. In the present paper, we present an alternative strategy to reformulate the finite quotient as a GIT quotient:
\begin{center}
\begin{tikzpicture}
\path (0, 0) node[align=center] (A) {$ U_0, …, U_n ∈ \Rep(Γ) $ \\ \emph{Simple representations}};
\path (5, 0) node[align=center] (B) {$ φ_{ij}: U_i ¤ U_j → \bigoplus_k U_k^{⊕ c_{ijk}} $ \\ \emph{Bilinear maps}};
\path (10, 0) node[align=center] (C) {$ X ≅ (\CG_Γ × V) \sslash \GL $ \\ \emph{Clebsch-Gordan variety}};
\path ($ (A.east)!0.5!(B.west) $) node {\Large $ \rightsquigarrow $};
\path ($ (B.east)!0.5!(C.west) $) node {\Large $ \rightsquigarrow $};
\end{tikzpicture}
\end{center}
The strategy is to define the points of the variety $ \CG_Γ $ to consist of bilinear maps $ φ_{ij}: U_i ¤ U_j → \bigoplus_k U_k^{⊕ c_{ijk}} $. Then $ \CG_Γ $ comes naturally with an action of the general linear group $ \GL = \GL(U_1) × … × \GL(U_n) $, but in contrast to the quiver representation space the action is quadratic on the domain of each $ φ_{ij} $. In case of the Kleinian singularity $ ℂ^2 \sslash Γ $, we show that the variety $ \CG_Γ × ℂ^2 $ indeed produces both the smooth resolution $ \widetilde{ℂ^2 \sslash Γ} $ and the stacky resolution $ [ℂ^2 / Γ] $ as GIT quotients. This paper is a cumulation of a long chain of developments:
\begin{center}
\renewcommand{\arraystretch}{1.2}
\begin{tabular}{ccccc}
\emph{Year} & \emph{Key contributions} & \emph{Approach} & \emph{Type of resolution} & \emph{Type of group} \\\hline
1993 & Folklore \cite{Fulton-toric} & Toric & Smooth & $ A_n $ \\
1994 & King \cite{king-quivers} & Quiver & Smooth & $ A_n $, $ D_n $, $ E_{6, 7, 8} $ \\
2001 & Bridgeland-King-Reid \cite{Bridgeland-King-Reid} & $ Γ $-clusters & Smooth & Finite $ Γ ⊂ \SL_3 (ℂ) $ \\
2004 & Van den Bergh \cite{vandenBergh-flops} & Smash & NCCR & Any finite group \\
2024 & Abdelgadir-Segal \cite{Abdelgadir-Segal} & Tannakian & Smooth and stacky & $ A_n $, $ D_4 $ \\
\emph{2025} & \emph{This paper} & \emph{Tannakian} & \emph{Smooth and stacky} & \emph{Any finite group}
\end{tabular}
\end{center}
In the remainder of this introduction, we elaborate on some of these developments and explain the road to the construction presented in this paper.

\subsection*{Smooth resolutions via quiver varieties}
Kleinian singularities are the affine quotients $ ℂ^2 \sslash Γ $ where $ Γ $ is a finite subgroup of $ \SL_2 (ℂ) $. They enjoy an ADE classification, so that we can speak of the Kleinian $ A_n $, $ D_n $ and $ E_{6/7/8} $ singularities. Their coordinate rings $ ℂ[X, Y]^Γ $ are generated by three variables with a single relation. In fact, the Kleinian singularities are precisely the simple isolated hypersurface singularities in three-dimensional complex space. Kleinian singularities can be resolved and it is well-known that they have a minimal resolution. If $ Γ $ is the $ A_n $, $ D_n $ or $ E_{n=6, 7, 8} $ group, it takes precisely $ n $-many individual blow-ups to reach the minimal resolution.

Every Kleinian singularity can be written as a quiver variety $ \Rep(Π_Q, α) \sslash \GL_α $, where $ \Qbar $ is the Kleinian double quiver and $ α $ is the minimal imaginary dimension vector. The classical road to this was established by Crawley-Boevey and Holland \cite{Crawley-Boevey-Holland-deformed-preprojective} and proceeds via the skew-group ring $ ℂ[X, Y] \rtimes Γ $ and a calculation that the skew-group ring is Morita equivalent to the preprojective algebra $ Π_Q $ of the Kleinian quiver $ Q $. Describing the singularity as a quiver variety opens up the possibility to introduce stability parameters $ θ ∈ ℤ^{Q_0} $ in the aim of providing a resolution of singularities. The theory in this direction was heavily influenced by King \cite{king-quivers} and in the Kleinian case, when $ θ ∈ ℤ^{Q_0} $ is a generic stability parameter which pairs to zero with $ α $, the variety $ \Rep(Π_Q, α) \sslash_θ \GL_α $ is indeed the minimal resolution of $ \Rep(Π_Q, α) \sslash \GL_α $.

\subsection*{The quest for stacky resolutions via GIT}
Noncommutative geometry suggests that we also study resolutions of categories on a categorical level \cite{Kuznetsov}. In case of finite quotient singularities $ V \sslash Γ $, the general theme is that one should study the stack $ [V / Γ] $. This stack is also known as the orbifold resolution or stacky resolution of $ ℂ^2 \sslash Γ $. The relevance of this stack is that the derived category $ \D\Coh[V / Γ] $ is supposed to be a categorical resolution of $ V \sslash Γ $. This appears to be known in many cases, at least in case of the Kleinian singularities $ ℂ^2 \sslash Γ $ as $ [ℂ^2 / Γ] $ is a DM stack resolution of $ ℂ^2 \sslash Γ $ and hence $ \D\Coh[ℂ^2 / Γ] ≅ \D\Coh\widetilde{ℂ^2 \sslash Γ} $, see e.g.~the discussion of \cite[Conjecture 1.2]{vandenBergh-NCresolutions}. Categorical and stacky technicalities are however beyond the scope of this paper.

The natural question arises how $ [ℂ^2 / Γ] $ can be obtained as a GIT quotient of the form $ [X^θ / \GL] $. It is impossible to use the quiver approach $ X = \Rep(Π_Q, α) $ for this purpose. Indeed, $ θ $-semistable representations are always stabilized by the diagonal $ ℂ^* ⊂ \GL $, and even after the $ ℂ^* $ factor has been removed from the gauge group the stabilizer never becomes finite. Therefore a representation space of a quiver will never yield the stacky resolution as GIT quotient. We are therefore led to search for alternative varieties $ X $ with $ \GL $-action.

\subsection*{The Tannakian strategy of Abdelgadir and Segal}
Abdelgadir and Segal \cite{Abdelgadir-Segal} paved the way for constructing a variety $ X $ with $ \GL $-action for a given finite group $ Γ $. Their work consists of two theoretical parts and a precise execution in the Kleinian $ A_n $ and $ D_4 $ cases. We shall now list the two theoretical parts and comment on them in more detail.
\begin{enumerate}
\item An approach for turning a finite group $ Γ $ into an affine variety $ Z $ with a $ \GL $-action and an open subset $ Z^{\open} ⊂ Z $ such that $ \GL $ acts on $ Z^{\open} $ with stabilizer $ Γ $.
\item The ansatz $ X = Z × ℂ^2 $. They equip the product with a $ \GL $-action and propose to find stability parameters $ θ_1 $ and $ θ_2 $ such that $ X^{θ_1} = Z^{\open} × ℂ^2 $ and $ X^{θ_2} ≅ \Rep(Π_Q, α)^{θ_2} $. In consequence we have
\begin{equation*}
[X^{θ_1} / \GL] ≅ [ℂ^2 / Γ] \quad \text{and} \quad [X^{θ_2} / \GL] ≅ \widetilde{ℂ^2 \sslash Γ}.
\end{equation*}
\end{enumerate}
Let us explain the first part in more detail. The starting point is the observation that Lurie's Tannaka duality for geometric stacks \cite{Lurie-tannaka} guarantees an equivalence between $ [*/Γ] $ and the stack of tensor-product preserving functors from $ \Coh[*/Γ] = \Rep(Γ) $ to $ \Vect $. Abdelgadir and Segal focus on those functors which on objects are merely the forgetful functor. They further propose to analyze these functors by studying the category $ \Rep(Γ) $ in terms of the simple representations $ U_0, …, U_n $ and the tensor and wedge relations among them. Through such presentation, the set of functors can be expressed as an algebraic variety $ Z $. For instance, if there is a relation of the form $ U_i ∧ U_j ≅ U_k $ in $ \Rep(Γ) $, then the datum of a point $ z ∈ Z $ should include a linear map $ U_i ∧ U_j → U_k $. There is an obvious action by the group $ \GL = \prod_{k = 1, …, n} \GL(U_k) $ on $ Z $ by gauging all linear maps on their domain and codomain. The problem is that the variety $ Z $ obtained this way has far too many orbits. Unfortunately, in the approach of Abdelgadir and Segal there is no immediate way to reduce the size of $ Z $ in general and their work is therefore limited to the more tractable $ A_n $ and $ D_4 $ case.

\subsection*{The solution for the $ A_n $ and $ D_4 $ type}
Abdelgadir and Segal successfully carry out the construction of a suitable variety $ Z $ and stability parameters $ θ_1 $ and $ θ_2 $ in the case of the Kleinian $ A_n $ and $ D_4 $ groups. Their implementation is however limited to these cases, as their size reduction procedure for $ Z $ does not generalize to other finite groups. To illustrate this limitation, we shall treat here the $ A_n $ case.

For the $ A_n $ case, we have the simple representations $ U_0, U_1, …, U_n $ where the generator $ σ ∈ Γ = C_{n+1} $ acts by $ e^{2πij/(n+1)} $ on $ U_j $. Among others, the simple representations satisfy the relations $ U_1 ¤ U_i ≅ U_{i+1} $ for $ 1 ≤ i ≤ n $. Therefore the datum of a point in $ B ∈ Z $ shall be given by $ n $ scalar numbers $ B_i $, standing for linear maps $ B_i: U_1 ¤ U_i → U_{i+1} $. This choice $ Z = ℂ^n $ and $ Z^{\open} = (ℂ^*)^n $ successfully implements the first part of the proposal of Abdelgadir and Segal in the $ A_n $ case.

To visualize the limitation of this approach, note that many relations among the simple representations have been discarded in the construction of $ Z $. Indeed, we have the general relation $ U_i ¤ U_j ≅ U_{i+j} $. If we were to include these relations into the definition of $ Z $, we would end up with a $ (n+1)^2 $-dimensional variety $ Z $ which is by far too large to consist of one single orbit. In conclusion, there is good reason in discarding these additional relations, but the choice of relations to be discarded is rather arbitrary.

In the $ D_4 $ case, the limitations become even more visible. There is one higher-dimensional irreducible representation and therefore many non-scalar relations arise. We are thus facing the impending question which part of those non-scalar relations should be discarded and which should be preserved. Abdelgadir and Segal succeed in making a valid selection, but the selection is necessarily fragmented and discards many scalar and non-scalar relations. It is even necessary to enforce additional coherence conditions among the non-discarded relations. Nevertheless, the work of Abdelgadir and Segal clearly paves the way for a general construction which we shall now present.

\subsection*{The general solution for all finite groups}
The aim of this paper is to break through the limitation to the $ D_4 $ case and construct a variety $ Z $ for any finite group $ Γ $. Our key insight is that it is not necessary to reduce the amount of relations included in the construction of $ Z $. Rather, we keep all tensor relations $ U_i ¤ U_j ≅ \bigoplus_{k = 0, …, n} U_k^{⊕ c_{ijk}} $, where $ c_{ijk} $ are the Clebsch-Gordan coefficients. Instead of discarding relations, the trick is to define the variety as the closure of one single orbit. This surprisingly simple definition produces a variety with $ \GL $-action for every finite group $ Γ $.

To mark this level of generality, we will call this variety the \emph{Clebsch-Gordan variety} $ \CG_Γ $ of $ Γ $, based on a terminology poll held at a conference in Paderborn. The remainder of the paper is devoted to constructing $ \CG_Γ $ explicitly in case $ Γ $ is the Kleinian group of $ A_n $ or $ D_n $ type and checking in detail that $ \CG_Γ × ℂ^2 $ produces both the stacky and smooth resolution of $ ℂ^2 \sslash Γ $ as GIT quotients. As a bonus, we prove the surprising fact that even in the $ D_4 $ case our variety $ \CG_Γ $ is similar but not naturally isomorphic to the variety constructed by Abdelgadir and Segal.

The paper is structured as follows. In \autoref{sec:prelim}, we recall several preliminaries. In \autoref{sec:tarig}, we recall the approach of Abdelgadir and Segal and their solution in the $ D_4 $ case. In \autoref{sec:strategy}, we define the Clebsch-Gordan variety, develop its general mechanism to produce resolutions for Kleinian singularities and state the main results. In \autoref{sec:caseA}, we explicitly work out the solution for the Kleinian $ A_n $ case and provide additional material on the classification of semiinvariant functions on $ \CG_Γ × ℂ^2 $. In \autoref{sec:caseD}, we explicitly work out the solution for the Kleinian $ D_n $ case. We include calculations of the symmetry and coherence relations and the comparison with the variety of Abdelgadir and Segal.

\subsection*{Connection to the physics literature}
Abdelgadir and Segal write that “It is perhaps surprising, given the extensive literature on the McKay
correspondence, that such a construction has not appeared before.” It was pointed out to the author by Mina Aganagic that theoretical physicists have worked with resolutions of singularities as well. At least two bits of our construction have appeared in work of Lawrence-Nekrasov-Vafa \cite{Lawrence-Nekrasov-Vafa} before. The first point concerns the use of bilinear maps $ U_i ¤ U_j → \bigoplus_k U_k^{⊕ c_{ij}} $. In physics the datum of such a map is known as field content of bifundamental matter type. The second point concerns our comparison map $ R: \CG_Γ × ℂ^2 → \Rep(Π_Q, α) $, which is mentioned by Lawrence-Nekrasov-Vafa as well. We hope the present paper enriches the mathematical toolbox and provides insipiration to physicists.

\subsection*{Acknowledgements}
The ideas for the paper were born during a visit at the University of Glasgow, where Tarig Abdelgadir was presenting his work. The author would like to thank Gwyn Bellamy for the invitation and the University of Glasgow for hospitality. The author was partially supported by NWO Rubicon grant 019.232EN.029.

\section{Preliminaries}
\label{sec:prelim}
In this section, we recall several preliminaries and fix pieces of notation. We treat GIT quotients, quivers, Kleinian singularities and their resolutions. For further reading, we recommend \cite{bellamy-schedler}. All algebraic varieties in this paper are defined over the field of complex numbers.

\paragraph*{GIT quotients} If $ X $ is an affine variety and $ G $ is a algebraic reductive group acting on $ X $, then the affine quotient $ X \sslash G $ is the affine variety whose coordinate ring is $ ℂ[X]^G $. A stability parameter for $ G $ is a group character $ θ: G → ℂ^* $. If $ k ∈ ℕ $, then a $ θ^k $-semiinvariant on $ X $ is a function $ f ∈ ℂ[X] $ such that $ f(gx) = θ^n (g) f(x) $. A point $ x ∈ X $ is $ θ $-semistable if there exists a $ θ^k $-semiinvariant $ f $ for some $ k ∈ ℕ $ such that $ f(x) ≠ 0 $. We write $ X^θ $ for the subset of $ θ $-semistable points of $ X $. The space of $ θ^k $-semiinvariants is denoted $ ℂ[X]_{θ^k} $. The term GIT quotient may refer to either of two concepts. The GIT quotient $ X \sslash_θ G $ is the quasiprojective variety with graded coordinate ring $ \bigoplus_{k ∈ ℕ} ℂ[X]_{θ^n} $. The GIT quotient $ [X^θ / G] $ is the stacky quotient of the $ θ $-semistable locus by $ G $.

The Hilbert-Mumford criterion for deciding whether a point $ x ∈ X $ is $ θ $-semistable works as follows. A one-parameter subgroup of $ G $ is an algebraic group homomorphisms $ g: ℂ^* → G $, which we commonly denote by $ g(t) $. The composition $ θ ∘ g $ with a stability parameter $ θ $ is thus of the form $ (θ ∘ g)(t) = t^α $ for some $ α ∈ ℤ $. The one-parameter subgroup pairs positively with $ θ $ if $ α > 0 $. The Hilbert-Mumford criterion states that an element $ x ∈ X $ is not $ θ $-semistable if and only if there exists a one-parameter subgroup $ g(t) $ which pairs positively with $ θ $ and for which $ \lim_{t → 0} g(t)^{-1} x $ exists.

\paragraph*{Quivers} A quiver $ Q $ is a finite directed graph. The vertex and arrow sets are denoted $ Q_0 $ and $ Q_1 $, respectively. We denote the head and tail of an arrow $ a ∈ Q_1 $ by $ h(a) $ and $ t(a) $. The path algebra $ ℂQ $ is the associative unital algebra with basis given by the set of paths in $ Q $ and product given by concatenation. The double quiver $ \qQ $ of $ Q $ is the quiver obtained from $ Q $ by adjoining an arrow $ a^*: h(a) → t(a) $ for every arrow $ a: t(a) → h(a) $. The preprojective algebra $ Π_Q $ of $ Q $ is defined as
\begin{equation*}
Π_Q = \frac{ℂ\qQ}{\big(\sum_{t(a) = v} a a^* - a^* a\big)_{v ∈ Q_0}}.
\end{equation*}
A dimension vector is an element $ α ∈ ℤ^{Q_0} $. The representation space $ \Rep(Π_Q, α) $ is the affine variety given by all representations of $ Π_Q $ of dimension $ α $. The gauge group $ \GL_α $ is the product of general linear groups given by
\begin{equation*}
\GL_α = \prod_{v ∈ Q_0} \GL_{α_v} (ℂ).
\end{equation*}
The group $ \GL_α $ acts on $ \Rep(Π_Q, α) $ by
\begin{equation*}
(g_v).ρ = (g_{h(a)} (ρ_a)_{a ∈ Q_1} g_{t(a)}^{-1})_{a ∈ \qQ_1}.
\end{equation*}
The quiver variety of $ Q $ is the affine quotient $ \Rep(Π_Q, α) \sslash \GL_α $. When $ θ ∈ ℤ^{Q_0} $, we define a group character $ θ: \GL_α → ℂ^* $ by $ θ(g) = \prod_{v ∈ Q_0} \det(g_v) $, using the same letter $ θ $ by abuse of notation. A representation $ ρ ∈ \Rep(Π_Q, α) $ is $ θ $-semistable if we have $ θ · α = 0 $ and for every nontrivial strict subrepresentation $ η ⊂ ρ $ we have $ θ · \dim η ≥ 0 $.

\paragraph*{Kleinian singularities} The finite subgroups $ Γ ⊂ \SL_2 (ℂ) $ are known as the Kleinian groups. They are classified by $ A_n $, $ D_n $ and $ E_{n=6, 7, 8} $ types. The affine quotients $ ℂ^2 \sslash Γ $ are known as Kleinian singularities. It turns out that every Kleinian singularity can be written as a quiver variety. The classical road to this result is to regard the skew-group ring $ ℂ[X, Y] \rtimes Γ $ whose center is equal to $ ℂ[X, Y]^Γ $. By the Artin-Wedderburn theorem, $ ℂΓ $ is isomorphic to a product of matrix rings. Similar to the fact that any matrix ring is Morita-equivalent to its coefficient ring, an idempotent reduction yields a Morita-equivalent subring $ e(ℂ[X, Y] \rtimes Γ)e $ of the skew-group ring. It is a classical result of Crawley-Boevey and Holland \cite{Crawley-Boevey-Holland-deformed-preprojective} that this subring is isomorphic to the preprojective algebra $ Π_Q $ of a quiver $ Q $. There is a classical theory of roots for quivers \cite{kac} and one denotes by $ α $ the minimal imaginary root of $ Q $. The set $ \Rep(Π_Q, α) $ of all $ α $-dimensional representations of $ Π_Q $, or in other words all representations of the double quiver $ \qQ $ which satisfy the preprojective relations, form an affine variety over the complex numbers. The product $ \GL_α $ of general linear groups acts on $ \Rep(Π_Q, α) $ and it turns out that the affine quotient $ \Rep(Π_Q, α) \sslash \GL_α $ is isomorphic to the Kleinian singularity $ ℂ^2 \sslash Γ $. This way, we can write every Kleinian singularity as a quiver variety. The data attached to the Kleinian $ A_n $ and $ D_n $ singularities is depicted in \autoref{fig:prelim-kleinian}.

\begin{figure}
\centering
\begin{subfigure}{0.99\linewidth}
\centering
\begin{tikzpicture}[xscale=0.9, yscale=0.8]
\begin{scope}[shift={(-10, 0)}]
\path (0, 0) node {$ Γ = \left⟨\pmat{e^{2πi/(n+1)} & 0 \\ 0 & e^{-2πi/(n+1)}}\right⟩ $};
\path (5, 0) node {$ xy - z^{n+1} = 0 $};
\end{scope}
\path (90:2) node (A) {1};
\path (150:2) node (B) {1};
\path (210:2) node (C) {1};
\path (270:2) node (D) {$ \dots $};
\path (330:2) node (E) {1};
\path (30:2) node (F) {1};
\path[draw, ->] ($ (A.west) + (up:0.1) $) to node[midway, above] {$ A_1 $} ($ (B.east) + (up:0.1) $);
\path[draw, <-] ($ (A.west) + (down:0.1) $) to node[midway, below] {$ A_1^* $} ($ (B.east) + (down:0.1) $);
\path[draw, ->] ($ (B.south) + (left:0.1) $) to node[midway, left] {$ A_2 $} ($ (C.north) + (left:0.1) $);
\path[draw, <-] ($ (B.south) + (right:0.1) $) to node[midway, right] {$ A_2^* $} ($ (C.north) + (right:0.1) $);
\path[draw, <-] ($ (C.east) + (up:0.1) $) to node[midway, below, shift={(0, -0.1)}] {$ A_3 $} ($ (D.west) + (up:0.1) $);
\path[draw, ->] ($ (C.east) + (down:0.1) $) to node[midway, above, shift={(0, 0.1)}] {$ A_3^* $} ($ (D.west) + (down:0.1) $);
\path[draw, <-] ($ (D.east) + (up:0.1) $) to node[pos=0.4, above, shift={(-0.2, 0)}] {$ A_{n-1}^* $} ($ (E.west) + (up:0.1) $);
\path[draw, ->] ($ (D.east) + (down:0.1) $) to node[midway, below, shift={(0.1, 0)}] {$ A_{n-1} $} ($ (E.west) + (down:0.1) $);
\path[draw, ->] ($ (E.north) + (right:0.1) $) to node[midway, right] {$ A_n $} ($ (F.south) + (right:0.1) $);
\path[draw, <-] ($ (E.north) + (left:0.1) $) to node[midway, left] {$ A_n^* $}($ (F.south) + (left:0.1) $);
\path[draw, ->] ($ (F.west) + (up:0.1) $) to node[midway, above, shift={(0.1, 0)}] {$ A_{n+1} $} ($ (A.east) + (up:0.1) $);
\path[draw, <-] ($ (F.west) + (down:0.1) $) to node[midway, below] {$ A_{n+1}^* $} ($ (A.east) + (down:0.1) $);
\end{tikzpicture}
\caption{The $ A_n $ singularity}
\label{fig:prelim-A}
\end{subfigure}

\vspace{0.5cm}

\begin{subfigure}{0.99\linewidth}
\centering
\begin{tikzpicture}[xscale=0.8, yscale=0.5]
\begin{scope}[shift={(3, 3)}]
\path (0, 0) node {$ Γ = \left⟨\pmat{e^{πi/(n-2)} & 0 \\ 0 & e^{-πi/(n-2)}}, \pmat{0 & i \\ i & 0}\right⟩ $};
\path (8, 0) node {$ x^2 + y^2 z + z^{n-1} = 0 $};
\end{scope}\path (-2, 2) node (E1) {1};
\path (-2, -2) node (E2) {1};
\path (0, 0) node (O1) {2};
\path (4, 0) node (I2) {2};
\path (8, 0) node (dots) {$ \dots $};
\path (12, 0) node (On-3) {2};
\path (14, 2) node (E3) {1};
\path (14, -2) node (E4) {1};
\path[draw, ->] (E1.south east) to node[midway, above] {$ A_1 $} (O1.north west);
\path[draw, ->] ($ (O1.north west) + (down:0.2) $) to node[midway, below] {$ A_1^* $} ($ (E1.south east) + (down:0.2) $);
\path[draw, ->] (E2.north east) to node[midway, above] {$ A_2 $} (O1.south west);
\path[draw, ->] ($ (O1.south west) + (down:0.2) $) to node[midway, below] {$ A_2^* $} ($ (E2.north east) + (down:0.2) $);
\path[draw, ->] ($ (O1.east) + (up:0.1) $) to node[midway, above] {$ A_3 $} ($ (I2.west) + (up:0.1) $);
\path[draw, ->] ($ (I2.west) + (down:0.1) $) to node[midway, below] {$ A_3^* $} ($ (O1.east) + (down:0.1) $);
\path[draw, ->] ($ (I2.east) + (up:0.1) $) to node[midway, above] {$ A_4 $} ($ (dots.west) + (up:0.1) $);
\path[draw, ->] ($ (dots.west) + (down:0.1) $) to node[midway, below] {$ A_4^* $} ($ (I2.east) + (down:0.1) $);
\path[draw, ->] ($ (dots.east) + (up:0.1) $) to node[midway, above] {$ A_{n-2} $} ($ (On-3.west) + (up:0.1) $);
\path[draw, ->] ($ (On-3.west) + (down:0.1) $) to node[midway, below] {$ A_{n-2}^* $} ($ (dots.east) + (down:0.1) $);
\path[draw, ->] ($ (On-3.north east) $) to node[midway, above] {$ A_{n-1} $} (E3.south west);
\path[draw, ->] ($ (E3.south west) + (down:0.2) $) to node[midway, below] {$ A_{n-1}^* $} ($ (On-3.north east) + (down:0.2) $);
\path[draw, ->] (On-3.south east) to node[midway, above] {$ A_n $} (E4.north west);
\path[draw, ->] ($ (E4.north west) + (down:0.2) $) to node[midway, below] {$ A_n^* $} ($ (On-3.south east) + (down:0.2) $);
\end{tikzpicture}
\caption{The $ D_n $ singularity}
\label{fig:prelim-D}
\end{subfigure}
%
\caption{This figure depicts the $ A_n $ and $ D_n $ singularities. We list the Kleinian group $ Γ $, the equivalent description as hypersurface in $ ℂ^3 $ and the Kleinian quiver setting $ (Q, α) $. Note that the Kleinian quiver with index $ n $ has $ (n+1) $-many vertices.}
\label{fig:prelim-kleinian}
\end{figure}

\paragraph*{Resolutions} Let $ (Q, α) $ be a Kleinian quiver setting and let $ θ ∈ ℤ^{Q_0} $ be a stability parameter. Then there is a natural map $ \Rep(Π_Q, α) \sslash_θ \GL_α → \Rep(Π_Q, α) \sslash \GL_α $. If $ θ ∈ ℤ^{Q_0} $ is generic and $ θ · α = 0 $, then this map is the minimal resolution of the Kleinian singularity. For this classical fact and material on other quiver varieties and their resolutions we refer to \cite{bellamy-schedler}.

The points of the variety $ \Rep(Π_Q, α) \sslash_θ \GL_α $ are in one-to-one correspondence with orbits of semistable points (or polystable points, in case of general quiver varieties). This way, the variety is equal to the stack $ [\Rep(Π_Q, α)^θ / G] $. The diagonal subgroup $ ℂ^* ⊂ \GL_α $ stabilizes the entire representation space. The reader can easily convince themself that the stabilizer of $ θ $-semistable representations is precisely $ ℂ^* $.

The special fiber $ π^{-1} (0) $ of the resolution $ π: \Rep(Π_Q, α) \sslash_θ \GL_α \surjects \Rep(Π_Q, α) \sslash \GL_α $ consists of $ n $-many transversally intersecting projective lines. In fact, the graph whose vertices are the projective lines and whose edges are the intersections is precisely the Dynkin diagram associated with the Kleinian ADE type. The special fiber $ π^{-1} (0) $ consists of the orbits of $ θ $-semistable representations which lie in the nullcone. By definition, the nullcone is the set of representations $ ρ $ for which there exists a 1-parameter subgroup $ g(t) ⊂ \GL_α $ such that $ g(t) ρ → 0 $ as $ t → 0 $. For specific values of $ θ ∈ ℤ^{Q_0} $, one can provide explicit representatives for the points in the special fiber. The specific case where $ θ_v < 0 $ for the special vertex and $ θ_v > 0 $ for all other vertices has been investigated by Crawley-Boevey \cite{cb-kleinian}. The statement is that the $ n $-many 1-parameter families of representations are distinguished by their socle. For every vertex $ v ∈ Q_0 $, apart from the special vertex, there is a 1-parameter family of representations in $ π^{-1} (0) $ whose socle equals the simple representation $ S_v $ at vertex $ v $. Explicit representatives of these 1-parameter families have been calculated for instance in \cite{vdKreeke-Namikawa}.

\section{The Tannakian approach}
\label{sec:tarig}
In this section, we recapitulate the Tannakian approach of Abdelgadir and Segal. In \autoref{sec:tarig-tannaka}, we recall the approach to produce a variety $ Z $ together with a $ \GL $-action via Tannaka duality. In \autoref{sec:tarig-construction}, we recall the implementation of this construction in the $ A_n $ and $ D_4 $ case. In \autoref{sec:tarig-implementing}, we recall the ansatz $ X = Z × ℂ^2 $ to produce both stacky and smooth resolutions. We finish by explaining why the implementation chosen by Abdelgadir and Segal does not easily extend beyond the $ A_n $ and $ D_4 $ cases.
\begin{center}
\begin{tikzpicture}
\path (0, 0) node[align=center, above] (A) {$ \mathcal{X} = [*/Γ] $ \\ \emph{Classifying stack}};
\path (5, 0) node[align=center, above] (B) {$ \mathcal{X} = \{F: \Rep(Γ) \overset{¤}{→} \Vect\} $ \\ \emph{Tannakian description}};
\path (10.5, 0) node[align=center, above] (C) {$ Z = \{(β, A, B) \running \text{Relations}\} $ \\ \emph{Algebraic variety}};
\path[draw, decoration={snake, amplitude=0.1em}, decorate, -{To[scale=1.5]}] (A) to ($ (B.west) + (left:0.3) $);
\path[draw, decoration={snake, amplitude=0.1em}, decorate, -{To[scale=1.5]}] ($ (B.south east) + (0.2, 0.5) $) to ($ (C.south west) + (-0.3, 0.5) $);
\end{tikzpicture}
\end{center}

\subsection{Tannakian approach of Abdelgadir and Segal}
\label{sec:tarig-tannaka}
In this section, we describe the approach of Abdelgadir and Segal to produce a variety $ Z $ with a $ \GL $-action from a finite group $ Γ $. The starting point is the work of Lurie \cite{Lurie-tannaka} which establishes that a geometric stack $ \mathcal{X} $ is equivalent to the stack of tensor-preserving functors from the category of coherent sheaves $ \Coh(\mathcal{X}) $ to the category of vector spaces:
\begin{equation*}
\mathcal{X} ≅ \{F: \Coh(\mathcal{X}) \xrightarrow{¤} \Vect\}.
\end{equation*}
For $ \mathcal{X} = [*/Γ] $ the statement specializes to $ [*/Γ] ≅ \{F: \Rep(Γ) \xrightarrow{¤} \Vect\} $. Abdelgadir and Segal propose to convert the data encoded by the functor $ F $ into algebraic data. To simplify the problem, they regard only those functors $ F $ which on objects are the forgetful map. To further understand the functors, they propose to write the category $ \Rep(Γ) $ in terms of generators and tensor relations. Implementing this concept precisely is not entirely straight-forward due to the coherence issues between the desirable isomorphisms. Abdelgadir and Segal succeed in implementing this approach in the case where $ Γ $ is the Kleinian $ A_n $ or $ D_4 $ group. In what follows, we summarize their implementation.

\subsection{Construction of the variety $ Z $}
\label{sec:tarig-construction}
In this section, we recall the construction of the variety $ Z $ of Abdelgadir and Segal in the $ A_n $ and $ D_4 $ cases. This section is for illustrative purposes only, since the materials will not be used in the remainder of the paper and we will provide our alternative construction in the generality of arbitrary finite groups $ Γ $ in \autoref{sec:strategy}.

We start with the $ A_n $ case. The Kleinian group $ Γ = C_{n+1} $ has the simple representations $ U_0, U_1, …, U_n $. The generator $ σ ∈ Γ = C_{n+1} $ acts by $ e^{2πij/(n+1)} $ on $ U_j $. Among others, the simple representations satisfy the relations $ U_1 ¤ U_i ≅ U_{i+1} $ for $ 1 ≤ i ≤ n $. Therefore the datum of a point in $ B ∈ Z $ shall be given by $ n $ scalar numbers $ B_i $, standing for linear maps $ B_i: U_1 ¤ U_i → U_{i+1} $. The gauge action of an element $ g = (g_1, …, g_n) ∈ \GL = (ℂ^*)^n $ is given by
\begin{equation*}
(gB)_i = g_{i+1} B_i g_1^{-1} g_i^{-1}, \quad 1 ≤ i ≤ n.
\end{equation*}
Note that we use the notation $ U_{n+1} = U_0 $ and $ g_0 = g_{n+1} = 1 $. The open locus $ Z^{\open} ⊂ Z $ is given by those points $ B $ such that $ B_i ≠ 0 $ for all $ 1 ≤ i ≤ n $. It is an instructive exercise to verify that $ Z^{\open} $ consists of one single $ \GL $-orbit. Therefore this choice of variety $ Z $ successfully implements the first part of the proposal of Abdelgadir and Segal in the $ A_n $ case.

We now explain the $ D_4 $ case. The Kleinian group is the quaternion group $ Γ = Q_8 = \{±1, ±i, ±j, ±k\} $. Abdelgadir and Segal label the irreducible representations by $ ℂ, L_1, L_2, L_3, V $ which are of dimensions $ 1, 1, 1, 1, 2 $. The gauge group is $ G = (ℂ^*)^3 × \GL_2 (ℂ) $. Abdelgadir and Segal choose the following three relations among the simple representations:
\begin{align*}
(V ∧ V) ¤ (V ∧ V) &≅ L_1 ¤ L_2 ¤ L_3 \\
L_i ¤ L_i &≅ V ∧ V, \quad i = 1, 2, 3, \\
\Sym^2 V &≅ L_1 ⊕ L_2 ⊕ L_3.
\end{align*}
All other relations are discarded. The construction therefore revolves around tuples $ (β, α_1, α_2, α_3, B) $ where
\begin{align*}
β &∈ \Hom((V ∧ V) ¤ (V ∧ V), L_1 ¤ L_2 ¤ L_3), \\
α_i &∈ \Hom(L_i^2, V ∧ V), \\
B &∈ \Hom(\Sym^2 V, L_1 ⊕ L_2 ⊕ L_3).
\end{align*}
The letter $ A $ is a compact way of denoting the matrix
\begin{equation*}
A = \pmat{α_1 & 0 & 0 \\ 0 & α_2 & 0 \\ 0 & 0 & α_3}.
\end{equation*}
There are several interrelations among the three relations, so that it is necessary to enforce coherence conditions among $ β $, $ A $ and $ B $. In total, the variety $ Z $ is defined as all tuples $ (β, A, B) $ which satisfy the following three conditions:
\begin{itemize}
\item[(E1)] $ B^* A B = α_1 α_2 α_3 β^2 J $,
\item[(E2)] $ B J^{-1} B^* A = α_1 α_2 α_3 β^2 \Id_V $,
\item[(E3)] $ ∧^2 B = β AB J^{-1} $.
\end{itemize}
Here $ J $ denotes the canonical isomorphism $ J: \Sym^2 V \isoto \Sym^2 V^* ¤ (V ∧ V)^{¤2} $. The left-hand side on the third row must in fact be interpreted through yet another canonical isomorphism. We elaborate on these conditions and on a hands-on version with fixed bases in \autoref{sec:caseD-comparison}.

The action of $ G $ on $ Z $ is given naturally by left-multiplication on the codomain of $ β $, $ A $ and $ B $ and by inverse-left-multiplication on their codomains. The open subset $ Z^{\open} ⊂ Z $ is defined to consist of those points for which $ \det B ≠ 0 $. Abdelgadir and Segal prove that $ G $ indeed acts transitively on $ Z^{\open} $ with stabilizer isomorphic to $ Γ $.

\subsection{Implementation of the ansatz $ X = Z × ℂ^2 $}
\label{sec:tarig-implementing}
In this section, we show how Abdelgadir and Segal utilize the affine variety $ Z $ to produce both the smooth and the stacky resolution of the Kleinian singularity in the $ D_4 $ case. The ansatz is to put $ X = Z × ℂ^2 $. Recall that the gauge group $ G = \GL(L_1) × \GL(L_2) × \GL(L_3) × \GL_2 (V) $ acts naturally on $ Z $. Moreover, the space $ ℂ^2 $ is naturally a representation of the Kleinian group $ Γ ⊂ \SL_2 (ℂ) $ and as such it is isomorphic to $ V $. Through this identification, the product $ X = Z × ℂ^2 $ obtains a natural $ G $-action.

The choice of stability parameters is $ θ_1 = (-1, …, -1) ∈ ℤ^n $ and $ θ_2 = (+1, …, +1) ∈ ℤ^n $. Abdelgadir and Segal prove that $ (Z × ℂ^2)^{θ_1} = Z^{\open} × ℂ^2 $. The easy direction of this equality is the inclusion $ Z^{\open} × ℂ^2 ⊂  (Z × ℂ^2)^{θ_1} $. To verify the inclusion, it suffices to note that the function $ f((β, A, B), x) = α_1^2 α_2^2 α_3^2 β^2 \det(B) $ is a $ θ_1 $-semiinvariant on $ Z × ℂ^2 $ that does not vanish on $ Z^{\open} $. In consequence, we have $ [X^{θ_1} / G] ≅ [ℂ^2 / Γ] $.

Abdelgadir and Segal prove moreover that $ [(Z × ℂ^2)^{θ_2} / G] ≅ \widetilde{ℂ^2 \sslash Γ} $. Their stategy is to construct a comparison map $ R: Z × ℂ^2 → \Rep(Π_Q, α) $, where $ (Q, α) $ is the Kleinian $ D_4 $ quiver. The map $ R $ is constructed explicitly by combining the datum of the maps $ β $, $ A $, $ B $ with the vector $ x ∈ ℂ^2 $. The representation $ R((β, A, B), x) $ is depicted in \autoref{fig:abdelgadir-quiverrep}.

\begin{figure}
\centering
\begin{tikzpicture}[scale=1.5]
\path (0, 0) node (C) {$ ℂ $};
\path (2, -2) node (U2) {$ L_1 $};
\path (4, 0) node (U3) {$ L_2 $};
\path (2, 2) node (U4) {$ L_3 $};
\path (2, 0) node (U5) {$ V $};
\path[draw, ->] ($ (C.east) + (up:0.1) $) to node[midway, above] {$ x $} ($ (U5.west) + (up:0.1) $);
\path[draw, <-] ($ (C.east) + (down:0.1) $) to node[midway, below] {$ (α_1 α_2 α_3 β^2)(x, -) $} ($ (U5.west) + (down:0.1) $);
\path[draw, ->] ($ (U2.north) + (left:0.1) $) to node[midway, left] {$ α_1 (B_1^x)^{∨} $} ($ (U5.south) + (left:0.1) $);
\path[draw, <-] ($ (U2.north) + (right:0.1) $) to node[midway, right] {$ B_1^x $} ($ (U5.south) + (right:0.1) $);
\path[draw, ->] ($ (U3.west) + (down:0.1) $) to node[midway, below] {$ α_2 (B_2^x)^{∨} $} ($ (U5.east) + (down:0.1) $);
\path[draw, <-] ($ (U3.west) + (up:0.1) $) to node[midway, above] {$ B_2^x $} ($ (U5.east) + (up:0.1) $);
\path[draw, ->] ($ (U4.south) + (right:0.1) $) to node[midway, right] {$ α_3 (B_3^x)^{∨} $} ($ (U5.north) + (right:0.1) $);
\path[draw, <-] ($ (U4.south) + (left:0.1) $) to node[midway, left] {$ B_3^x $} ($ (U5.north) + (left:0.1) $);
\end{tikzpicture}
\caption{This figure depicts the quiver representation $ R((β, A, B), x) $. The meaning of all symbols is elaborated in \autoref{sec:caseD-comparison}.}
\label{fig:abdelgadir-quiverrep}
\end{figure}
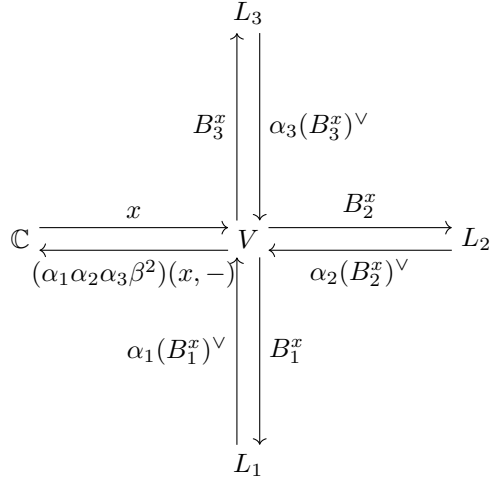

Generalizing this implementation to all ADE cases is a nontrivial task. The $ A_n $ case is much simpler than the $ D_4 $ case due to the scalar nature of the relations and the abelianness of the gauge group, and can be dealt with in a similar manner. The ad-hoc nature in the $ D_4 $ case leaves it unclear how to generalize the implementation of Abdelgadir and Segal. The most pressing challenge for the general $ D_n $ and $ E_{6, 7, 8} $ cases is the choice of tensor, symmetric and wedge relations among the simple representations $ U_0, …, U_n ∈ \Rep(Γ) $.

Abdelgadir and Segal express their wish that the mathematical community replace their “hands-on proof with a more abstract and elegant one, which might then work in greater generality.” In the present article, we respond to this appeal by continuing the Tannakian approach of Abdelgadir and Segal in a different direction that generalizes to all Kleinian singularities.

\section{Resolution of finite quotient singularities}
\label{sec:strategy}
In this section, we define Clebsch-Gordan varieties and show how to use them to resolve Kleinian singularities.
\begin{center}
\begin{tikzpicture}
\path (0, 0) node[align=center] (A) {$ X = ℂ^2 \sslash Γ $ \\ \emph{Kleinian}};
\path (4, 0) node[align=center] (B) {$ \GL \curvearrowright \CG_Γ $, $ θ_1 $, $ θ_2 $ \\ \emph{Clebsch-Gordan}};
\path (10, 1) node[align=center] (C) {$ [(\CG_Γ × ℂ^2)^{θ_1} / \GL] ≅ [ℂ^2 / Γ] $ \\ \emph{Stacky resolution}};
\path (10, -1) node[align=center] (D) {$ [(\CG_Γ × ℂ^2)^{θ_2} / \GL] ≅ \widetilde{ℂ^2 \sslash Γ} $ \\ \emph{Smooth resolution}};
\path[draw, decoration={snake, amplitude=0.1em}, decorate, -{To[scale=2]}] (A) to ($ (B.west) + (left:0.2) $);
\path[draw, decoration={snake, amplitude=0.1em}, decorate, -{To[scale=2]}] (B) to (C);
\path[draw, decoration={snake, amplitude=0.1em}, decorate, -{To[scale=2]}] (B) to (D);
\end{tikzpicture}
\end{center}
In \autoref{sec:strategy-cgdata} and \ref{sec:strategy-variety} we define the Clebsch-Gordan variety $ \CG_Γ $ together with the action of $ \GL $ for an arbitrary finite group $ Γ $. In \autoref{sec:strategy-finquot} we explain how to use $ \CG_Γ $ as candidate for resolutions of finite quotient singularities. In \autoref{sec:strategy-kleinian} we specialize to the case of Kleinian groups $ Γ ⊂ \SL_2 (ℂ) $. In \autoref{th:strategy-mainresult}, we formulate the main result. In \autoref{sec:strategy-proof}, we prove the main result following the six stages formulated in \autoref{stageplan}. Most parts of the stages are proven simultaneously for all ADE types and the others are proven for the $ A_n $ and $ D_n $ types in \autoref{sec:caseA} and \ref{sec:caseD}.

\subsection{Clebsch-Gordan data}
\label{sec:strategy-cgdata}
In this section, we define the notion of a Clebsch-Gordan datum. Simply speaking, a Clebsch-Gordan datum consists of bilinear maps $ U_i ¤ U_j → \oplus_k U_k^{⊕ c_{ijk}} $. The starting point is the Tannakian approach recalled in \autoref{sec:tarig}. While Abdelgadir and Segal convert a Tannakian functor $ F: \Rep(Γ) → \Vect $ into a limited amount of algebraic data by truncating the information contained in the functor, we instead convert the functor into the full amount of algebraic data and preserve the information contained in the functor.

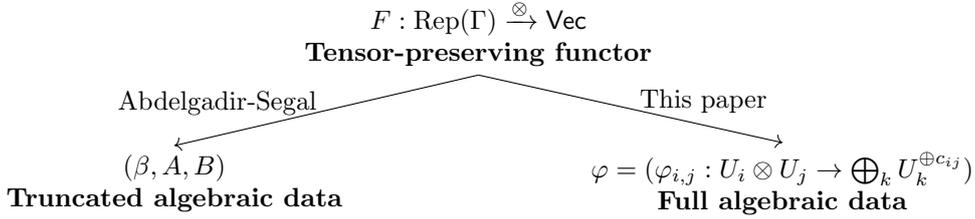
\begin{center}
\begin{tikzpicture}
\path (0, 0) node[align=center] (A) {$ F: \Rep(Γ) \xrightarrow{¤} \Vect $ \\ \emph{Tensor-preserving functor}};
\path (-4, -2) node[align=center] (B) {$ (β, A, B) $ \\ \emph{Truncated algebraic data}};
\path (4, -2) node[align=center] (C) {$ φ = (φ_{i,j}: U_i ¤ U_j → \bigoplus_k U_k^{⊕ c_{ij}}) $ \\ \emph{Full algebraic data}};
\path[draw, ->] (A.south) to node[midway, left, shift={(0, 0.1)}] {Abdelgadir-Segal} (B.north);
\path[draw, ->] (A.south) to node[midway, right, shift={(0, 0.1)}] {This paper} (C.north);
\end{tikzpicture}
\end{center}

We start with a finite group $ Γ $. Let $ U_0, …, U_n $ be its simple representations, with $ U_0 $ denoting the trivial representation. We regard only functors $ F $ which send $ U_i $ to $ U_i $ as a vector space. It is our aim to convert the abstract property that $ F $ preserves tensor products into an algebraic property. We start by the observation that inside the category $ \Rep(Γ) $ we have the isomorphisms
\begin{equation*}
U_i ¤ U_j ≅ \bigoplus_{k = 0, …, n} U_k^{⊕ c_{ijk}}.
\end{equation*}
Here $ c_{ijk} $ are the Clebsch-Gordan coefficients. We convert the property that $ F $ preserves tensor products into the condition that there be an isomorphism $ φ_{i,j}: F(U_i) ¤ F(U_j) \isoto F(U_i ¤ U_j) $. Note that the domain and codomain of the maps $ φ_{i,j} $ are simply complex vector spaces of known dimension. We have thus turned the functor data of $ F $ into purely algebraic data. Since $ F(U_i) $ is simply the representation $ U_i $ regarded as a vector space, we shall drop the letter $ F $. We fix this terminology and notation as follows:

\begin{definition}
Let $ Γ $ be a finite group. Let $ U_0, …, U_n $ be its simple representations, with $ U_0 $ denoting the trivial representation. Let $ c_{ijk} ∈ ℕ $ be the Clebsch-Gordan coefficients. A \emph{Clebsch-Gordan datum} $ φ = (φ_{i,j})_{0 ≤ i, j ≤ n} $ is the datum of bilinear maps
\begin{equation*}
φ_{i,j}: U_i ¤ U_j → \bigoplus_k U_k^{⊕ c_{ijk}}.
\end{equation*}
The \emph{gauge group} associated with $ Γ $ is defined as follows:
\begin{equation*}
\GL = \GL(U_1) × … × \GL(U_n).
\end{equation*}
\end{definition}

\begin{remark}
The factor $ \GL(U_0) = ℂ^* $ does not appear in our definition of the gauge group. The group $ \GL $ acts on a Clebsch-Gordan datum $ φ $ by “change of basis”. More precisely, a group element acts by left composition on the codomain and inverse composition on the domain:
\begin{equation*}
(gφ)_{ij} = \big(\bigoplus_{k = 0, …, n} g_k^{⊕ c_{ijk}}\big) ∘ φ_{i,j} ∘ (g_i^{-1} ¤ g_j^{-1}).
\end{equation*}
The maps $ φ_{i,j} $ in a Clebsch-Gordan datum need not be isomorphisms and need not satisfy any coherence relations. Any Clebsch-Gordan datum coming from a functor $ F $ will however satisfy certain coherence relations. We shall exploit this property later on.
\end{remark}

\subsection{The Clebsch-Gordan variety}
\label{sec:strategy-variety}
In this section, we define the Clebsch-Gordan variety $ \CG_Γ $ for general finite groups $ Γ $. While Abdelgadir and Segal define $ Z $ as the collection of all truncated data $ (β, A, B) $ that satisfy coherence relations, we instead define $ \CG_Γ $ as closure of one single $ \GL $-orbit:
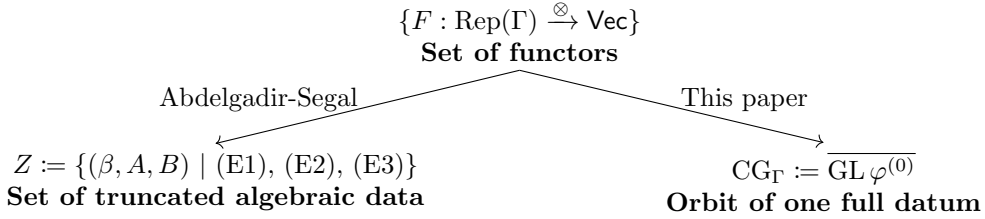
\begin{center}
\begin{tikzpicture}
\path (0, 0) node[align=center] (A) {$ \{F: \Rep(Γ) \xrightarrow{¤} \Vect\} $ \\ \emph{Set of functors}};
\path (-4, -2) node[align=center] (B) {$ Z ≔ \{(β, A, B) \running \text{(E1), (E2), (E3)}\} $ \\ \emph{Set of truncated algebraic data}};
\path (4, -2) node[align=center] (C) {$ \CG_Γ ≔ \overline{\GL φ^{(0)}} $ \\ \emph{Orbit of one full datum}};
\path[draw, ->] (A.south) to node[midway, left, shift={(up:0.1)}] {Abdelgadir-Segal} (B.north);
\path[draw, ->] (A.south) to node[midway, right, shift={(up:0.1)}] {This paper} (C.north);
\end{tikzpicture}
\end{center}

\begin{definition}
Let $ Γ $ be a finite group. Let $ U_0, …, U_n $ be its simple representations, with $ U_0 $ denoting the trivial representation. A Clebsch-Gordan datum $ φ = (φ_{i,j}) $ of $ Γ $ is \emph{regular} if each component $ φ_{i,j} $ is an isomorphism of $ Γ $-representations and each component $ φ_{0, i} $ and $ φ_{i, 0} $ is the identity. The \emph{Clebsch-Gordan variety} $ \CG_Γ $ of $ Γ $ is defined as $ \CG_Γ = \overline{\GL φ^{(0)}} $ where $ φ^{(0)} $ is any choice of regular Clebsch-Gordan datum. Its open dense subset $ \CG_Γ^{\open} ⊂ \CG_Γ $ is defined as the orbit $ \GL φ^{(0)} $.
\end{definition}

\begin{remark}
For every finite group $ Γ $ there exists a regular Clebsch-Gordan datum $ φ^{(0)} $ and therefore a notion of Clebsch-Gordan variety $ \CG_Γ $. It seems likely but not a priori clear that the Clebsch-Gordan variety is also independent of the choice of $ φ^{(0)} $. In the remainder of the paper, we work out the Clebsch-Gordan variety in the specific case where $ Γ $ is the $ A_n $ or $ D_n $ Kleinian group. In these cases, we choose the components $ φ^{(0)}_{i, j} $ to be specific isomorphisms that are easy to handle.
\end{remark}

The maps $ φ_{i,j} $ contained in a Clebsch-Gordan datum $ φ $ are not by definition required to satisfy coherence relations. However, we shall explain now that they do in fact satisfy coherence relations. The precise shape of these relations depends on the choice of $ φ^{(0)} $.

\begin{lemma}
\label{th:construction-CGvariety-coherence}
Let $ U_i, U_j, U_k $ be a choice of simple $ Γ $-representations. Then there exists a unique linear map $ γ_{i, j, k} $ which renders the following diagram commutative for all $ φ ∈ \CG_Γ $:
\begin{center}
\begin{tikzcd}
U_i ¤ U_j ¤ U_k \arrow[dd, "φ_{i,j} ¤ \id"] \arrow[rrr, "\id ¤ φ_{j,k}"] &&& U_i ¤ \bigoplus_l U_l^{⊕ c_{jkl}} \arrow[d, "\bigoplus_l φ_{i, l}^{⊕ c_{jkl}}"] \\
&&& \bigoplus_l \bigoplus_m U_m^{⊕ c_{jkl} c_{ilm}} \\
\bigoplus_{l'} U_{l'}^{⊕ c_{ijl'}} ¤ U_k \arrow[rr, "\bigoplus_{l'} φ_{l', k}^{⊕ c_{ijl'}}"] && \bigoplus_{l'} \bigoplus_{m'} U_{m'}^{⊕ c_{ijl'} c_{l'km'}} \arrow[ur, dashed, "∃! ~ γ_{i, j, k}", "\sim"'] &
\end{tikzcd}
\end{center}
\end{lemma}

\begin{proof}
We start by picking any element $ φ ∈ \GL φ^{(0)} $ and regarding the tensor product of vector spaces $ U_i ¤ U_j ¤ U_k $. There are two different ways of applying the bilinear maps contained in $ φ $ here: either first on the $ U_i ¤ U_j $ component or first on the $ U_j ¤ U_k $ component. After applying on either side, there is only one possible way to apply the bilinear maps to the remaining terms. These two ways are captured in the diagram. The left/bottom path of the diagram is the trilinear map arising from first applying $ φ $ on the left tensor component. The top/right path of this diagram is the trilinear map arising from first applying $ φ $ on the right tensor component. The core observation is that both paths are isomorphisms of vector spaces since this is true for $ φ = φ^{(0)} $ and therefore for any $ φ ∈ \GL φ^{(0)} $. Therefore there exists a unique isomorphism $ γ_{i, j, k}^{φ} $ which renders the diagram commutative.

Let us now explain that the isomorphism $ γ_{i, j, k}^{φ} $ does not depend on $ φ $. Indeed, by assumption the bilinear maps $ φ^{(0)}_{i, j} $ are isomorphisms of $ Γ $-representations. Therefore $ γ_{i, j, k}^{φ^{(0)}} $ is an isomorphism of $ Γ $-representations as well. In particular, it is a direct sum of components which only run between identical simple representations:
\begin{equation*}
γ_{i, j, k}^{φ^{(0)}} = \bigoplus_m γ_m, \quad γ_m = \pmat{γ_{m, 11} \Id_{U_m} & … & … \\ … & … & … \\ … & … & γ_{m, d_m d_m} \Id_{U_m}}.
\end{equation*}
Here $ d_m $ denotes the total number of appearances of $ U_m $ in the domain (equivalently codomain) of $ γ_{i, j, k}^{φ^{(0)}} $, and the entries $ γ_{m, st} $ are scalars. For an element $ φ = gφ^{(0)} $ with $ g ∈ \GL $, the map $ γ_{i, j, k}^{φ} $ is simply the conjugate of $ γ_{i, j, k}^{φ^{(0)}} $ in every individual matrix entry. Since the matrix entries $ γ_{m, st} \Id_{U_m} $ are just scalar multiples of the identity, we conclude that $ γ_{i, j, k}^{φ} = γ_{i, j, k}^{φ^{(0)}} $. In conclusion, the map $ γ_{i, j, k}^{φ} $ is independent of $ φ $ as long as $ φ ∈ \GL φ^{(0)} $, and we shall simply denote it by $ γ_{i, j, k} $. We obtain a single coherence relation which holds equally for all $ φ ∈ \GL φ^{(0)} $ and by passing to limits, we conclude that the coherence relation in fact holds for every $ φ ∈ \CG_Γ $. This finishes the proof.
\end{proof}

\begin{remark}
We shall memorize the coherence relation in the following shortcut notation: For $ u ∈ U_i $ and $ v ∈ U_j $ we have
\begin{equation*}
φ(- ¤ φ(u ¤ v)) = γ_{•, i, j} (φ(φ(- ¤ u) ¤ v)).
\end{equation*}
\end{remark}

\begin{remark}
Similar to the coherence relations, for every pair of simples $ U_i, U_j $ and $ φ ∈ \CG_Γ $ we have symmetry relations
\begin{equation*}
φ_{i, j} = γ_{i, j} ∘ φ_{j, i} ∘ σ.
\end{equation*}
Here $ σ: U_i ¤ U_j → U_j ¤ U_i $ denotes the flip, and the linear map $ γ_{i, j} $ is independent of $ φ $.
\end{remark}

\subsection{Resolution of finite quotient singularities}
\label{sec:strategy-finquot}
In this section, we explain how to use Clebsch-Gordan varieties for resolutions of finite quotient singularities. Let $ V $ be a representation of $ Γ $ and regard the quotient singularity $ V \sslash Γ $. Decomposing $ V $ into simple representations as $ V = U_0^{d_0} ⊕ … ⊕ U_n^{d_n} $, we obtain a natural action of the gauge group $ \GL $ on $ V $. The gauge group also acts on the Clebsch-Gordan variety $ \CG_Γ $ and we obtain an action on the product variety:
\begin{equation*}
\GL \curvearrowright X = \CG_Γ × V.
\end{equation*}
Together with different choices of stability parameters $ θ $ for $ \GL $, this GIT package serves as candidate to provide both smooth and stacky resolutions of $ V \sslash Γ $. In particular, if $ ℂ^2 \sslash Γ $ is a Kleinian singularity, we regard $ V = ℂ^2 $ as natural $ Γ $-representation and define $ X = \CG_Γ × ℂ^2 $.

\subsection{Resolution of Kleinian singularities}
\label{sec:strategy-kleinian}
In this section, we study the case of Kleinian singularities $ ℂ^2 \sslash Γ $. The main result reads as follows.

\begin{theorem}
\label{th:strategy-mainresult}
Let $ Γ $ be any Kleinian group, and let $ θ_1 $ and $ θ_2 $ be a certain choice of stability parameters. Then we have
\begin{align*}
[(\CG_Γ × ℂ^2)^{θ_1} / \GL] ≅ [ℂ^2 / Γ] \quad \text{and} \quad [(\CG_Γ × ℂ^2)^{θ_2} / \GL] &≅ \widetilde{ℂ^2 \sslash Γ}.
\end{align*}
\end{theorem}

Even though our variety $ \CG_Γ $ is different from the variety $ Z $ of Abdelgadir and Segal, we can adapt the six-stage framework of proof. For instance, we succeed in defining a comparison map $ R: \CG_Γ × ℂ^2 → \Rep(Π_Q, α) $ for all Kleinian singularities. Another important tool in our construction is the projection map $ π: ℂ[\CG_Γ × ℂ^2] → ℂ[X, Y] $ along the inclusion $ \{φ^{(0)}\} × ℂ^2 ⊂ \CG_Γ × ℂ^2 $.

\begin{stageplan}
\label{stageplan}
The proof will proceed in the following six stages:
\begin{enumerate}
\item Construct the varieties $ \CG_Γ^{\open} $, $ \CG_Γ $ and the parameters $ θ_1 $, $ θ_2 $.
\item Construct the $ \GL $-equivariant map $ R: \CG_Γ × ℂ^2 → \Rep(Π_Q, α) $.
\item Verify $ ℂ[\Rep(Π_Q, α)]^{\GL} \overunderset{R^*}{\sim}{→} ℂ[\CG_Γ × ℂ^2]^{\GL} \overunderset{π}{\sim}{→} ℂ[X, Y]^{\Stab_{\GL} (φ^{(0)})} $.
\item Verify $ (\Stab_{\GL} (φ^{(0)}) \curvearrowright ℂ^2) ≅ (Γ \curvearrowright ℂ^2) $.
\item Verify $ \CG_Γ^{θ_1} = \CG_Γ^{\open} $ and $ (\CG_Γ × ℂ^2)^{θ_1} = \CG_Γ^{\open} × ℂ^2 $.
\item Verify $ R: (\CG_Γ × ℂ^2)^{θ_2} \isoto \Rep(Π_Q, α)^{θ_2} $.
\end{enumerate}
\end{stageplan}
The first, second, and part of the third, fourth, fifth and sixth stage can be proved for all ADE types simultaneously. The remaining parts are detailed separately for the $ A_n $ and $ D_n $ case in \autoref{sec:caseA} and \ref{sec:caseD}. Even though these remaining parts are only checked in the $ A_n $ and $ D_n $ cases, we claim the main result for all Kleinian singularities.

\subsection{Proof in the Kleinian case}
\label{sec:strategy-proof}
In this section, we prove the six stages of \autoref{stageplan}. Several technical parts are not proven here, but are deferred to appendices where they are checked in the $ A_n $ and $ D_n $ case. The reader finds an enumeration of these technical statements at the beginning of \autoref{sec:caseA}. Rather than providing intuitive insight, the aim of this section is to provide general proofs that work for all Kleinian groups. We highly recommend that the interested reader consult \autoref{sec:caseA} and \ref{sec:caseD} where all materials are illustrated in a more explicit fashion.

\paragraph*{Stage 1: Construction of the Clebsch-Gordan variety}
The Clebsch-Gordan variety $ \CG_Γ $ is defined for any finite group and therefore also for the Kleinian group $ Γ $. We choose the stability parameters $ θ_1 $ and $ θ_2 $ as follows:
\begin{equation*}
θ_1 = (-\dim U_i)_{i = 0, …, n}, \quad θ_2 = (+1, …, +1).
\end{equation*}

\paragraph*{Stage 2: Construction of the map $ R $}
It is our task to construct the $ \GL $-equivariant map $ R: \CG_Γ × ℂ^2 → \Rep(Π_Q, α) $. Here $ (Q, α) $ is the Kleinian quiver setting and $ \Rep(Π_Q, α) $ denotes the affine variety whose points are the representations of $ (Q, α) $ which satisfy the preprojective relations. Close inspection of the construction of the map $ R $ defined by Abdelgadir and Segal provides us with the cue that $ R(φ, x) $ should be defined by inserting the element $ x ∈ ℂ^2 $ into one slot of the Clebsch-Gordan datum $ φ $. Recall that the number of arrows from $ i $ to $ j $ in the Kleinian double quiver is equal to the dimension $ \dim_{Γ} (U_i ¤ ℂ^2, U_j) $. Therefore for every $ i = 0, …, n $, combining the Clebsch-Gordan datum $ φ $ and the element $ x $ we obtain a map $ φ_{i,•} (- ¤ x): U_i → \bigoplus_k U_k^{|a: i → j|} $. This procedure works for all Kleinian groups. It is implemented in detail in \autoref{sec:caseA-2} and \ref{sec:caseD-2}. We shall collect the definition as follows:

\begin{definition}
The map $ R: \CG_Γ × ℂ^2 → \Rep(Π_Q, α) $ is defined as
\begin{equation*}
R(φ, x) = φ(- ¤ x) ∈ \Rep(Π_Q, α).
\end{equation*}
\end{definition}

\begin{proposition}
The map $ R $ is well-defined and $ \GL $-equivariant.
\end{proposition}

\begin{proof}
Well-definedness comes down to checking that the representation $ R(φ, x) $ satisfies the preprojective relations. We verify this in \autoref{th:caseA-2-preproj} and \ref{th:caseD-2-preproj}. It is however easy to see that the map $ R $ is $ \GL $-equivariant: Let $ U_i $ be an irreducible representation and $ u ∈ U_i $. Then we have
\begin{equation*}
(R(gφ, gx))(u) = (gφ)(u ¤ gx) = g(φ(g^{-1} u ¤ x)) = g(R(φ, x)(g^{-1} u)) = (gR(φ, x))(u).
\end{equation*}
This finishes the proof.
\end{proof}

\paragraph*{Stage 3: Verification that $ R^* $ is an isomorphism of invariants}
It is our task to examine the pullback map $ R^*: ℂ[\Rep(Π_Q, α)] → ℂ[\CG_Γ × ℂ^2] $. We denote by $ \Stab_{\GL} (φ^{(0)}) ⊂ \GL $ the stabilizer group of $ φ^{(0)} $ under the $ \GL $-action. Recall that the group $ \GL $ acts on $ ℂ^2 $. Therefore also the stabilizer group acts on $ ℂ^2 $ and we have a notion of invariant ring $ ℂ[X, Y]^{\Stab_{\GL} (φ^{(0)})} $. A substantial trick in our investigations is the projection map $ π: ℂ[\CG_Γ × ℂ^2] → ℂ[X, Y] $ along the inclusion $ \{φ^{(0)}\} × ℂ^2 ⊂ \CG_Γ × ℂ^2 $. The notation $ π(f) $ applies regardless of whether $ f $ is an invariant function, a semiinvariant function, or otherwise. We are now ready to prove stage 3.

\begin{proposition}
\label{th:strategy-Rstar}
The map $ R^* $ provides an isomorphism between the $ ℂ[\Rep(Π_Q, α)]^{\GL} $ and $ ℂ[\CG_Γ × ℂ^2]^{\GL} $. The map $ π $ provides an isomorphism between $ ℂ[\CG_Γ × ℂ^2]^{\GL} $ and $ ℂ[X, Y]^{\Stab_{\GL} (φ^{(0)})} $. We thus have the commutative diagram
\begin{center}
\begin{tikzcd}
ℂ[\Rep(Π_Q, α)]^{\GL} \arrow[rd, "π ∘ R^*"', "\sim"] \arrow[rr, "R^*", "\sim"'] && ℂ[\CG_Γ × ℂ^2]^{\GL} \arrow[ld, "π", "\sim"'] \\
& ℂ[X, Y]^{\Stab_{\GL} (φ^{(0)})}. &
\end{tikzcd}
\end{center}
\end{proposition}

\begin{proof}
We divide the proof into several steps. The first step is to prove that $ R^* $ sends $ \GL $-invariants to $ \GL $-invariants. The second step is to prove that $ π $ sends $ \GL $-invariants to $ \Stab_{\GL} (φ^{(0)}) $-invariants. The third part is to prove that $ π $ is injective on $ \GL $-invariants. The fourth step is to prove that $ R^* ∘ π: ℂ[\Rep(Π_Q, α)]^{\GL} → ℂ[X, Y]^{\Stab_{\GL} (φ^{(0)})} $ is an isomorphism. Finally, we draw the conclusion that both factors in this composition are isomorphisms.

For the first step, we observe that $ R^* $ sends $ \GL $-invariant functions to $ \GL $-invariant functions since $ R $ itself is $ \GL $-equivariant. For the second step, we check that if $ f ∈ ℂ[\CG_Γ × ℂ^2] $ is $ \GL $-invariant, then $ π(f) $ is $ \Stab_{\GL} (φ^{(0)}) $-invariant. Indeed, pick $ x ∈ ℂ^2 $ and $ σ ∈ \Stab_{\GL} (φ^{(0)}) $. Then we have
\begin{equation*}
π(f)(σx) = f(φ^{(0)}, σx) = f(σ φ^{(0)}, σx) = f(φ^{(0)}, x) = π(f)(x).
\end{equation*}
For the third step, we observe that a $ \GL $-invariant function is already determined by its value on $ \{φ^{(0)}\} × ℂ^2 $. Indeed, let $ f, f' $ be two $ \GL $-invariant functions on $ \CG_Γ × ℂ^2 $ with $ π(f) = π(f') $, then
\begin{equation*}
f(gφ^{(0)}, x) = π(f)(g^{-1}x) = π(f')(g^{-1}x) = f'(gφ^{(0)}, x).
\end{equation*}
For the fourth step, we invoke the checks done in the $ A_n $ and $ D_n $ cases in \autoref{th:caseA-Rstarinv} and \ref{th:caseD-Rstarinv}. Finally, we conclude that $ π: ℂ[\CG_Γ × ℂ^2]^{\GL} → ℂ[X, Y]^{\Stab_{\GL} (φ^{(0)})} $ is both injective and surjective, thus the same holds for $ R^* $. This finishes the proof.
\end{proof}

\paragraph*{Stage 4: Identification of the stabilizer of $ φ^{(0)} $}
It is our task to prove that the stabilizer of $ φ^{(0)} $ under the $ \GL $-action is isomorphic to $ Γ $ and that $ ℂ^2 $ with the action of the stabilizer is isomorphic to $ ℂ^2 $ with the natural action of $ Γ ⊂ \SL_2 (ℂ) $. Tannaka theory predicts that $ \Stab_{\GL} (φ^{(0)}) ≅ Γ $ under the assumption that we have correctly turned Tannakian functors into algebraic data. We explain this stage in more detail in \autoref{sec:caseA-4} and \ref{sec:caseD-4}.

\paragraph*{Stage 5: Identification of the $ θ_1 $-semistable locus}
It is our task to identify the $ θ_1 $-semistable locus on $ \CG_Γ × ℂ^2 $. In fact, the reason behind our specific choice of $ θ_1 $ is that there is a very useful $ |Γ|θ_1 $-semiinvariant on $ \CG_Γ $, given essentially by multiplying up the determinants of all $ φ $ entries:

\begin{lemma}
\label{th:strategy-proof-theta1SICG}
The following function $ f_0: \CG_Γ → ℂ $ is a $ |Γ| θ_1 $-semiinvariant. We have $ f_0 (φ^{(0)}) ≠ 0 $. Moreover for $ k ≥ 1 $, any $ k|Γ| θ_1 $-semiinvariant is a scalar multiple of $ f_0^k $.
\begin{equation*}
f_0 (φ) = \prod_{i, j = 0}^n \det(φ_{i, j})^{\vdim U_i \vdim U_j}.
\end{equation*}
\end{lemma}

\begin{proof}
We shall prove all three parts of the claim after each other. We start with the first part. Its proof comes down to counting the number of appearances of $ \det(g_i) $ for $ i = 0, …, n $ in the expanded version of the term $ f_0 (gφ) $. We shall present a simple trick to count this number. We start by regarding the regular representation $ ℂΓ $. Its character has the property that $ χ_{ℂΓ} (1) = |Γ| $ and $ χ_{ℂΓ} (g) = 0 $ for $ g ≠ 1 $, and it decomposes as $ ℂΓ ≅ \bigoplus_i U_i^{\vdim U_i} $. The tensor product's character $ χ_{ℂΓ ¤ ℂΓ} $ has the property that $ χ_{ℂΓ ¤ ℂΓ} (1) = |Γ|^2 $ and $ χ_{ℂΓ ¤ ℂΓ} (g) = 0 $ for $ g ≠ 1 $, thus decomposes as $ ℂΓ ¤ ℂΓ ≅ (ℂΓ)^{⊕ |Γ|} $. The value $ f_0 (φ) $ is simply the determinant of the map $ ℂΓ ¤ ℂΓ → (ℂΓ)^{⊕ |Γ|} $ given by applying summing up all components of $ φ $ with multiplicities. Recall the rule $ \det(A ¤ B) = \det(A)^b \det(B)^a $ for square matrices $ A, B $ of dimension $ a×a $ and $ b×b $ respectively. We thus determine the total weight of $ \det(g_i) $ in expanding $ f_0 (gφ) $ to be
\begin{equation*}
|Γ| \vdim U_i - 2 \sum_{j ≠ i} \vdim U_i \vdim U_j^2 - 2 \vdim U_i^3 = - |Γ| \vdim U_i = |Γ| (θ_1)_i.
\end{equation*}
This shows that $ f_0 $ is a $ |Γ| θ_1 $-semiinvariant as desired.

The second part of the claim easy. All components of $ φ^{(0)} $ are linear isomorphisms and therefore $ f_0 (φ^{(0)}) ≠ 0 $ and thus $ f_0 (gφ^{(0)}) = θ_1 (g)^{|Γ|} f_0 (φ^{(0)}) ≠ 0 $ for any $ g ∈ \GL $.

For the third part of the claim, let $ f $ be another nonzero $ k|Γ|θ_1 $-semiinvariant. Then it necessarily takes nonzero value on $ φ^{(0)} $ and we conclude that $ f = λ f_0 $ holds on $ \GL φ^{(0)} $, where $ λ ≔ f(φ^{(0)}) / f_0 (φ^{(0)}) $. By passing to limits we conclude that $ f = λ f_0 $ on all of $ \CG_Γ $. This shows that $ f_0 $ is the only $ k|Γ|θ_1 $-semiinvariant on $ \CG_Γ $ up to scalar multiplication, finishing the proof.
\end{proof}

\begin{lemma}
\label{th:strategy-proof-theta1SI}
Any $ k|Γ|θ_1 $-semiinvariant on $ \CG_Γ × ℂ^2 $ with $ k ≥ 1 $ is the product of $ f_0^k $ and an invariant function.
\end{lemma}

\begin{proof}
We start by writing $ f_0: \CG_Γ × ℂ^2 → ℂ $ for the $ |Γ|θ_1 $-semiinvariant obtained from $ f_0: \CG_Γ → ℂ $ by simply extending it independent of the $ ℂ^2 $ factor. Now let $ f: \CG_Γ × ℂ^2 → ℂ $ be any $ k|Γ|θ_1 $-semiinvariant with $ k ≥ 1 $. Then for $ σ ∈ \Stab_{\GL} (φ^{(0)}) $ and $ x ∈ ℂ^2 $ we deduce
\begin{equation*}
π(f) (σx) = f(φ^{(0)}, σx) = θ_1 (σ)^{|Γ|} f(φ^{(0)}, x) = π(f) (x).
\end{equation*}
We have used that $ θ_1 $ is a group character on $ Γ $ and hence $ θ_1 (σ)^{|Γ|} = 1 $. We conclude that $ π(f) $ is $ \Stab_{\GL} (φ^{(0)}) $-invariant. By \autoref{th:strategy-Rstar}, there exists a $ \GL $-invariant function $ g: \Rep(Π_Q, α) → ℂ $ such that $ π(R^* g) = π(f) $. Define $ λ = 1 / f_0 (φ^{(0)}, x) $, noting that this number is independent on $ x $. We claim $ f = λ^k f_0^k R^* g $. Indeed, we have $ f(φ^{(0)}, x) = π(f) (x) = (R^* g)(φ^{(0)}, x) $, thus $ f = λ^k f_0^k R^* g $ holds on $ φ^{(0)} × ℂ^2 $. Since both sides are $ k|Γ|θ_1 $-semiinvariants, they agree on all of $ \GL φ^{(0)} × ℂ^2 $ and therefore on all of $ \CG_Γ × ℂ^2 $. We conclude that any $ k|Γ|θ_1 $-semiinvariant on $ \CG_Γ × ℂ^2 $ can be written as the product of $ f_0^k $ and a $ \GL $-invariant function.
\end{proof}

\begin{proposition}
We have $ \CG_Γ^{θ_1} = \GL φ^{(0)} $ and $ (\CG_Γ × ℂ^2)^{θ_1} = \GL φ^{(0)} × ℂ^2 $.
\end{proposition}

\begin{proof}
For the inclusions $ \GL φ^{(0)} ⊂ \CG_Γ^{θ_1} $ and $ \GL φ^{(0)} × ℂ^2 ⊂ (\CG_Γ × ℂ^2)^{θ_1} $, it suffices to recall that $ f_0 (φ^{(0)}) ≠ 0 $. We treat both reverse inclusions simultaneously. Let $ φ ∈ \CG_Γ^{θ_1} $ or $ (φ, x) ∈ (\CG_Γ × ℂ^2)^{θ_1} $, then there is a $ k|Γ|θ_1 $-semiinvariant for some $ k ≥ 1 $ which does not vanish on $ φ $ or $ (φ, x) $. By \autoref{th:strategy-proof-theta1SICG} or \ref{th:strategy-proof-theta1SI} it follows that $ f_0 (φ) ≠ 0 $. It is checked in \autoref{th:caseA-cgtheta1} and \ref{th:caseD-cgtheta1} that this implies $ φ ∈ \GL φ^{(0)} $. This finishes the proof.
\end{proof}

\paragraph*{Stage 6: Identification of the $ θ_2 $-semistable locus}
It is our task to identify the $ θ_2 $-semistable locus of $ \CG_Γ × ℂ^2 $. More precisely, we need to prove that $ R $ is an isomorphism when restricted to the sets of $ θ_2 $-semistable points. We note that the proof does not depend on the precise value of $ θ_2 $, only on the positivity of all entries.

\begin{lemma}
If $ (φ, x) ∈ \CG_Γ × ℂ^2 $ is $ θ_2 $-semistable, then $ R(φ, x) ∈ \Rep(Π_Q, α) $ is $ θ_2 $-semistable.
\end{lemma}

\begin{proof}
For the $ A_n $ and $ D_4 $ cases, this is checked in \autoref{th:caseA-theta2-semistable} and \ref{th:caseD-theta2-semistable}.
\end{proof}

We shall now prove that $ R $ actually reaches all $ θ_2 $-semistable representations. Recall that by definition the preimage of a $ θ_2 $-semistable is automatically $ θ_2 $-semistable. It is therefore entirely natural to restrict the domain of $ R $ to $ θ_2 $-semistables as well.

\begin{lemma}
The map $ R: (\CG_Γ × ℂ^2)^{θ_2} → \Rep(Π_Q, α)^{θ_2} $ is surjective.
\end{lemma}

\begin{proof}
The strategy is to regard two types of representations. We start by regarding the natural projection map $ \Rep(Π_Q, α) \sslash_{θ_2} \GL → \Rep(Π_Q, α) \sslash \GL $. The quotient $ \Rep(Π_Q, α) \sslash \GL $ is simply the Kleinian singularity and it consists of the singular zero point and the nonsingular points. The points in $ \Rep(Π_Q, α) \sslash_{θ_2} \GL $ which project to the zero point are known as the nullcone. The nullcone is known to consist \cite{cb-kleinian} of $ n $-many one-parameter families. We prove in \autoref{th:caseA-theta2surj} and \ref{th:caseD-theta2surj} that the representations in these $ n $-many one-parameter families lie in the image of $ R $. The remainder of the present proof is dedicated to those $ θ_2 $-semistable representations which do not lie in the nullcone.

Let $ ρ ∈ \Rep(Π_Q, α)^{θ_2} $ be a representation which does not does not lie in the nullcone. In other words, one of the three $ \GL $-invariants is nonzero on $ ρ $. Then since $ R $ induces an isomorphism $ (\CG_Γ × ℂ^2) \sslash \GL → \Rep(Π_Q, α) \sslash \GL $, we can pick an element $ p ∈ (\CG_Γ × ℂ^2) \sslash \GL $ with the same invariants. Since the map $ (\CG_Γ × ℂ^2)^{θ_2} → (\CG_Γ × ℂ^2) \sslash \GL $ is surjective, we can lift $ p $ to a $ θ_2 $-semistable element $ (φ, x) ∈ (\CG_Γ × ℂ^2)^{θ_2} $ with the same invariants. Then $ R(φ, x) ∈ \Rep(Π_Q, α)^{θ_2} $ has the same invariants as $ ρ $, and therefore both project to the same element of $ \Rep(Π_Q, α) \sslash \GL $. Since $ ρ $ does not lie in the nullcone, the projected element is nonzero, and since $ \Rep(Π_Q, α)^{θ_2} \sslash \GL → \Rep(Π_Q, α) \sslash \GL $ is bijective away from the zero point, we conclude that $ R(φ, x) $ and $ ρ $ are the same point in $ \Rep(Π_Q, α)^{θ_2} \sslash \GL $ and therefore only differ by gauge. This shows that $ ρ $ lies in the image of $ R $ and finishes the proof.
\end{proof}

\begin{lemma}
\label{th:construction-theta2}
The map $ R: (\CG_Γ × ℂ^2)^{θ_2} → \Rep(Π_Q, α)^{θ_2} $ is injective.
\end{lemma}

\begin{proof}
Let $ (φ, x) ∈ (\CG_Γ × ℂ^2)^{θ_2} $ and regard the simple representations $ U_0, U_1, …, U_n $ of $ Γ $. It is our goal to show that $ (φ, x) $ can be reconstructed from $ R(φ, x) $. We start with the observation that $ x $ is easily reconstructed from $ R(φ, x) $, since $ φ(1 ¤ x) = x $ for the unit element $ 1 ∈ U_0 $ in the trivial representation. Next, we shall explain how to reconstruct $ φ $.

Let us call a vector $ u ∈ U_i $ “recognized” if $ φ(- ¤ u) $ can be reconstructed from $ R(φ, x) $. By definition, the unit $ 1 ∈ U_0 $ is recognized, as the map $ φ(- ¤ 1) $ is the identity. Another key observation is that every individual summand of $ x $ in terms of the decomposition into simples $ ℂ^2 = ⊕ U_i^{k_i} $ is recognized. In the $ A_n $ case, the vector $ x ∈ U_1 ⊕ U_{n-1} $ has two components, and both are recognized since we can read off $ φ(- ¤ x_1) $ and $ φ(- ¤ x_n) $ directly from the representation $ R(φ, x) $. In the $ D_n $ and $ E_n $ cases, the vector $ x $ lies in one simple representation and by definition $ φ(- ¤ x) $ can be reconstructed from $ R(φ, x) $, therefore $ x $ is recognized as well.

We claim that if $ u ∈ U_i $ and $ v ∈ U_j $ are recognized, then also all individual components of $ φ(u ¤ v) $ are recognized. Indeed, we have the equality $ φ(- ¤ φ(u ¤ v)) = γ_{•, i, j} (φ(φ(- ¤ u) ¤ v)) $. Since $ u $ and $ v $ are recognized, the right-hand side can be reconstructed from $ R(φ, x) $. Both left-hand and right-hand side are, behind the scenes, maps from $ U = U_0 ⊕ U_1 ⊕ … ⊕ U_n $ to a direct sum of some $ U_k $'s. If $ φ(u ¤ v) $ consists of one single $ U_k $ component, then we are done. Otherwise, we shall argue that we can still read off $ φ(- ¤ w) $ for every component $ w $ of $ φ(u ¤ v) $. In fact, this is trivial since by definition the equality holds in a higher-dimensional space which keeps direct sum inputs separated from each other. This way, we conclude that $ φ(- ¤ w) $ is recognized for every component $ w $ of $ φ(u ¤ v) $.

Finally, let us prove that for $ i = 0, …, n $ every vector $ u ∈ U_i $ is recognized. We start by recalling that the representation $ R(φ, x) $ is generated by $ 1 ∈ U_0 $. Thus the space $ U_i $ is spanned by the transports of the vector $ 1 ∈ U_0 $ along some set $ P $ of paths in $ \Qbar $. By the earlier parts of this proof, we conclude inductively that any vector appearing in such a transport is recognized. Finally we have obtained a generating set for the vector space $ U_i $ that consists of recognized vectors. We conclude that any vector $ u ∈ U_i $ is recognized as well. This finishes the proof.
\end{proof}

\appendix

\section{The $ A_n $ case}
\label{sec:caseA}
In this section, we treat the $ A_n $ case in detail. We construct the variety $ \CG_Γ $, the stability parameters $ θ_1 $, $ θ_2 $ and the map $ R: \CG_Γ × ℂ^2 → \Rep(Π_Q, α) $, where $ (Q, α) $ is the Kleinian $ A_n $ quiver. We prove the remaining technical parts of \autoref{stageplan}. This concerns precisely the following statements:
\begin{itemize}
\item The representation $ R(φ, x) $ satisfies the preprojective relations.
\item The map $ π ∘ R^*: ℂ[\Rep(Π_Q, α)]^{\GL} → ℂ[X, Y]^{\Stab_{\GL} (φ^{(0)})} $ is an isomorphism.
\item We have $ (\Stab_{\GL} (φ^{(0)}) \curvearrowright ℂ^2) ≅ (Γ \curvearrowright ℂ^2) $.
\item If $ f_0 (φ) ≠ 0 $, then $ φ ∈ \GL φ^{(0)} $.
\item The map $ R $ sends $ θ_2 $-semistables to $ θ_2 $-semistables.
\item The $ θ_2 $-semistables in the nullcone lies in the image of $ R $.
\end{itemize}

\subsection{Stage 1: Construction of the Clebsch-Gordan variety}
\label{sec:caseA-1}
The Kleinian group of $ A_n $ is $ Γ = C_{n+1} $, the cyclic group of order $ n+1 $. The group lies embedded into $ \SL_2 (ℂ) $ as generated by a specific matrix $ σ ∈ \SL_2 (ℂ) $:
\begin{equation*}
Γ = C_{n+1} ≅ ⟨\pmat{e^{2πi/(n+1)} & 0 \\ 0 & e^{-2πi/(n+1)}}⟩ ⊂ \SL_2 (ℂ).
\end{equation*}
We enumerate the irreducible representations of $ Γ $ by $ U_0, U_1, …, U_n $, where $ U_0 $ is the trivial representation and the generator $ σ $ acts on $ U_k $ as $ e^{2πik/(n+1)} $:
\begin{center}
\begin{tabular}{lc}
Representation & Action of $ σ $ \\\hline
$ U_k $ & $ e^{2πik/(n+1)} $
\end{tabular}
\end{center}
The $ Γ $-representation $ ℂ^2 $ itself is isomorphic to $ U_1 ⊕ U_n $. The gauge group is the product of $ n $ general multiplicative groups:
\begin{equation*}
\GL = \underset{U_1}{ℂ^*} × … × \underset{U_n}{ℂ^*}, \\
\end{equation*}
A Clebsch-Gordan datum $ φ $ for this group consists of scalars $ φ_{i, j} ∈ ℂ $ for $ 0 ≤ i, j ≤ n $, representing maps $ U_i ¤ U_j → U_{i+j} $:
\begin{equation*}
φ = (φ_{i, j})_{0 ≤ i, j ≤ n}, \quad φ_{i, j}: U_i ¤ U_j → U_{i+j}.
\end{equation*}
Any sums of indices are calculated modulo $ n+1 $. The gauge group $ \GL $ acts on $ \CG_Γ $ by $ (gφ)_{ij} = g_{i+j} φ_{i, j} g_i^{-1} g_j^{-1} $. The specific datum $ φ^{(0)} $ is given by the following choice:
\begin{equation*}
φ^{(0)}_{i, j} = 1, \quad \text{for all } i, j.
\end{equation*}
We define $ \CG_Γ ≔ \closure{\GL φ^{(0)}} $. It is an affine variety of geometric dimension $ n+1 $. Any $ φ ∈ \CG_Γ $ satisfies the coherence and symmetry relations
\begin{equation*}
\begin{aligned}
φ_{i+j, k} φ_{i, j} &= φ_{i, j+k} φ_{j, k}, \quad && 0 ≤ i, j, k ≤ n, \\
φ_{i, j} &= φ_{j, i}, \quad && 0 ≤ i, j ≤ n.
\end{aligned}
\end{equation*}
Indeed, the element $ φ = φ^{(0)} $ satisfies these relations and in consequence any $ φ ∈ \CG_Γ $ satisfies these relations as well. We make the following choice of stability parameters:
\begin{align*}
& θ_1 = (-1, …, -1) ∈ ℤ^n, \\
& θ_2 = (+1, …, +1) ∈ ℤ^n.
\end{align*}

\begin{example}
Let us examine the Clebsch-Gordan variety for the case of $ n = 1, 2 $. Then a point in $ \GL φ^{(0)} $ is simply determined by the scalar $ φ_{1, 1} $. The orbit of the Clebsch-Gordan datum $ φ^{(0)} $ associated with the value $ φ_{1, 1} ∈ ℂ $ is isomorphic to $ ℂ^* ⊂ ℂ $ and the orbit closure $ \CG_Γ $ is simply $ ℂ $. In case $ n = 2 $, a Clebsch-Gordan datum is given by scalars $ φ_{1, 1}, φ_{1, 2}, φ_{2, 2} ∈ ℂ $ satisfying the condition $ φ_{2, 2} φ_{1, 1} = φ_{1, 2} $, therefore $ φ_{1, 1} $ and $ φ_{2, 2} $ suffice to determine a datum. In this complex plane, the orbit of $ φ^{(0)} $ is $ ℂ^* × ℂ^* ⊂ ℂ × ℂ $ and the orbit closure is the entire complex plane. The situation is more complicated for $ n ≥ 3 $.
\end{example}

\subsection{Stage 2: Construction of the map $ R $}
\label{sec:caseA-2}
We have depicted the Kleinian double quiver $ (\bar{Q}, α) $ with arrows $ A_i, A_i^* $ in \autoref{fig:prelim-A}. For $ (φ, x) ∈ \CG_Γ × ℂ^2 $ we define the representation $ R(φ, x) ∈ \Rep(Π_Q, α) $ following the recipe $ R(φ, x) = φ(- ¤ x) $. More precisely, write $ x = (x_1, x_n) $ and then define the representation by sending $ A_i ↦ φ_{i-1, 1} x_1 $ and $ A_i^* ↦ φ_{i, n} x_n $:
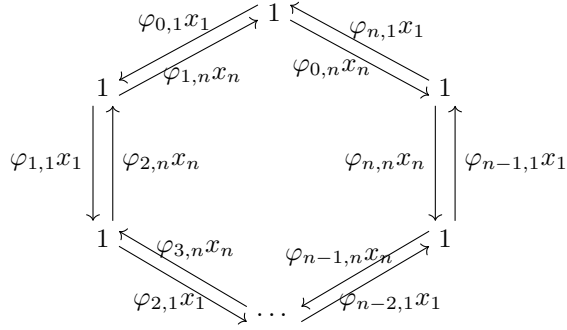
\begin{center}
\begin{tikzpicture}[xscale=1.3, yscale=1]
\path (90:2) node (A) {1};
\path (150:2) node (B) {1};
\path (210:2) node (C) {1};
\path (270:2) node (D) {$ \dots $};
\path (330:2) node (E) {1};
\path (30:2) node (F) {1};
\path[draw, ->] ($ (A.west) + (up:0.1) $) to node[pos=0.6, above, shift={(0, 0.1)}] {$ φ_{0, 1} x_1 $} ($ (B.east) + (up:0.1) $);
\path[draw, <-] ($ (A.west) + (down:0.1) $) to node[pos=0.4, below, shift={(0, -0.1)}] {$ φ_{1, n} x_n $} ($ (B.east) + (down:0.1) $);
\path[draw, ->] ($ (B.south) + (left:0.1) $) to node[midway, left] {$ φ_{1, 1} x_1 $} ($ (C.north) + (left:0.1) $);
\path[draw, <-] ($ (B.south) + (right:0.1) $) to node[midway, right] {$ φ_{2, n} x_n $} ($ (C.north) + (right:0.1) $);
\path[draw, <-] ($ (C.east) + (up:0.1) $) to node[pos=0.6, above, shift={(0, 0.1)}]  {$ φ_{3, n} x_n $} ($ (D.west) + (up:0.1) $);
\path[draw, ->] ($ (C.east) + (down:0.1) $) to node[pos=0.4, below, shift={(0, -0.1)}] {$ φ_{2, 1} x_1 $}  ($ (D.west) + (down:0.1) $);
\path[draw, <-] ($ (D.east) + (up:0.1) $) to node[pos=0.3, above, shift={(0, 0.1)}] {$ φ_{n-1, n} x_n $} ($ (E.west) + (up:0.1) $);
\path[draw, ->] ($ (D.east) + (down:0.1) $) to node[pos=0.7, below, shift={(0, -0.2)}] {$ φ_{n-2, 1} x_1 $} ($ (E.west) + (down:0.1) $);
\path[draw, ->] ($ (E.north) + (right:0.1) $) to node[midway, right] {$ φ_{n-1, 1} x_1 $} ($ (F.south) + (right:0.1) $);
\path[draw, <-] ($ (E.north) + (left:0.1) $) to node[midway, left] {$ φ_{n, n} x_n $}($ (F.south) + (left:0.1) $);
\path[draw, ->] ($ (F.west) + (up:0.1) $) to node[pos=0.3, above, shift={(0, 0.1)}] {$ φ_{n, 1} x_1 $} ($ (A.east) + (up:0.1) $);
\path[draw, <-] ($ (F.west) + (down:0.1) $) to node[pos=0.7, below, shift={(0, -0.1)}] {$ φ_{0, n} x_n $} ($ (A.east) + (down:0.1) $);
\end{tikzpicture}
\end{center}
In particular, the specific representation $ R(φ^{(0)}, x) $ takes the following shape:
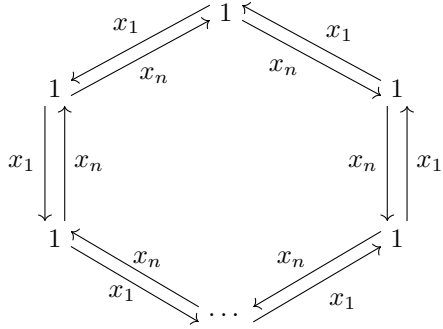
\begin{center}
\begin{tikzpicture}[xscale=1.3, yscale=1]
\path (90:2) node (A) {1};
\path (150:2) node (B) {1};
\path (210:2) node (C) {1};
\path (270:2) node (D) {$ \dots $};
\path (330:2) node (E) {1};
\path (30:2) node (F) {1};
\path[draw, ->] ($ (A.west) + (up:0.1) $) to node[pos=0.6, above, shift={(0, 0.1)}] {$ x_1 $} ($ (B.east) + (up:0.1) $);
\path[draw, <-] ($ (A.west) + (down:0.1) $) to node[pos=0.4, below, shift={(0, -0.1)}] {$ x_n $} ($ (B.east) + (down:0.1) $);
\path[draw, ->] ($ (B.south) + (left:0.1) $) to node[midway, left] {$ x_1 $} ($ (C.north) + (left:0.1) $);
\path[draw, <-] ($ (B.south) + (right:0.1) $) to node[midway, right] {$ x_n $} ($ (C.north) + (right:0.1) $);
\path[draw, <-] ($ (C.east) + (up:0.1) $) to node[pos=0.6, above, shift={(0, 0)}] {$ x_n $} ($ (D.west) + (up:0.1) $);
\path[draw, ->] ($ (C.east) + (down:0.1) $) to node[pos=0.4, below, shift={(0, 0)}] {$ x_1 $} ($ (D.west) + (down:0.1) $);
\path[draw, <-] ($ (D.east) + (up:0.1) $) to node[pos=0.3, above, shift={(0, 0.1)}] {$ x_n $} ($ (E.west) + (up:0.1) $);
\path[draw, ->] ($ (D.east) + (down:0.1) $) to node[pos=0.7, below, shift={(0, -0.2)}] {$ x_1 $} ($ (E.west) + (down:0.1) $);
\path[draw, ->] ($ (E.north) + (right:0.1) $) to node[midway, right] {$ x_1 $} ($ (F.south) + (right:0.1) $);
\path[draw, <-] ($ (E.north) + (left:0.1) $) to node[midway, left] {$ x_n $}($ (F.south) + (left:0.1) $);
\path[draw, ->] ($ (F.west) + (up:0.1) $) to node[pos=0.3, above, shift={(0, 0.1)}] {$ x_1 $} ($ (A.east) + (up:0.1) $);
\path[draw, <-] ($ (F.west) + (down:0.1) $) to node[pos=0.7, below, shift={(0, -0.1)}] {$ x_n $} ($ (A.east) + (down:0.1) $);
\end{tikzpicture}
\end{center}

\begin{lemma}
\label{th:caseA-2-preproj}
For any $ (φ, x) ∈ \CG_Γ × ℂ^2 $, the representation $ R(φ, x) $ satisfies the preprojective conditions.
\end{lemma}

\begin{proof}
We start with the observation that $ R(φ, x) $ is in every case a quiver representation and $ R $ defines a $ \GL $-equivariant map $ R: \CG_Γ × ℂ^2 → \Rep(\Qbar, α) $ to the representations of the Kleinian double quiver. Next, we observe due to the choice $ φ^{(0)}_{i, j} = 1 $, the specific representation $ R(φ^{(0)}, x) $ satisfies the preprojective relations $ A_i A_i^* - A_{i+1}^* A_{i+1}^* = 0 $. This implies that $ R(gφ^{(0)}, g(g^{-1} x)) = g R(φ^{(0)}, g^{-1} x) $ also satisfies the preprojective relations for any $ g ∈ \GL $ and $ x ∈ ℂ^2 $. By a standard limit argument, the preprojective conditions then also hold for any $ (φ, x) ∈ \CG_Γ × ℂ^2 $. We conclude that $ R $ becomes a $ \GL $-equivariant map $ R: \CG_Γ × ℂ^2 → \Rep(Π_Q, α) $. This finishes the proof.
\end{proof}

\subsection{Stage 3: Verification that $ R^* $ is an isomorphism of invariants}
\label{sec:caseA-3}
It is our task to prove the following lemma.

\begin{lemma}
\label{th:caseA-Rstarinv}
The map $ π ∘ R^*: ℂ[\Rep(Π_Q, α)]^{\GL} → ℂ[\CG_Γ × ℂ^2]^{\GL} → ℂ[X, Y]^{\Stab_{\GL} (φ^{(0)})} $ is an isomorphism.
\end{lemma}

\begin{proof}
The trick is to regard the three generating elements of the domain and codomain. In the codomain we pick the three elements
\begin{equation*}
A' = X^{n+1}, \quad B' = Y^{n+1}, \quad C' = XY.
\end{equation*}
In the domain we pick the three elements
\begin{equation*}
A = A_{n+1} … A_1, \quad B = A_1^* … A_{n+1}^*, \quad C = A_1 A_1^*.
\end{equation*}
Evidently, we have $ π(R^* (A)) = A' $ and $ π(R^* (B)) = B' $ and $ π(R^* (C)) = C' $. In both rings, the only relation satisfied by the three generators is $ A' B' - C'^{n+1} = 0 $. This finishes the proof.
\end{proof}

\subsection{Stage 4: Identification of the stabilizer of $ φ^{(0)} $}
\label{sec:caseA-4}
We shall examine the stabilizer of $ φ^{(0)} $ under the $ \GL $-action. Let $ g = (g_1, …, g_n) ∈ \GL $ and assume $ gφ^{(0)} = φ^{(0)} $. Then $ g_{i+j} = g_i g_j $ for any $ i, j $. We immediately see that $ g_1^{n+1} = 1 $ and $ g_i = g_1^i $. We immediately conclude that the stabilizer group of $ φ^{(0)} $ under the $ \GL $-action is the cyclic group of order $ n+1 $ generated by the element $ σ ∈ \GL $ given by $ σ_j = e^{2πij/(n+1)} $:
\begin{align*}
& \Stab_{\GL} (φ^{(0)}) = ⟨σ⟩ ⊂ \GL , \\
& σ = (e^{2πij/(n+1)})_j.
\end{align*}
The element $ σ $ acts on $ ℂ^2 = U_1 ⊕ U_n $ by $ σ(x_1, x_n) = (e^{2πi/(n+1)}, e^{-2πi/(n+1)}) $. This provides an identification
\begin{align*}
\Stab_{\GL} (φ^{(0)}) &\isoto Γ, \\
σ &↦ \pmat{e^{2πi/(n+1)} & 0 \\ 0 & e^{-2πi/(n+1)}}.
\end{align*}
Under this identification, the action of $ \Stab_{\GL} (φ^{(0)}) $ on $ ℂ^2 $ agrees with the action of $ Γ $ on $ ℂ^2 $. This finishes stage 4.

\subsection{Stage 5: Identification of the $ θ_1 $-semistable locus}
\label{sec:caseA-5}
It is our task to prove \autoref{th:caseA-cgtheta1}. We start by formulating the $ (n+1)θ_1 $-semiinvariant $ f_0: \CG_Γ → ℂ $ explicitly:
\begin{equation*}
f_0 (φ) = \prod_{i, j = 0}^n φ_{i, j}.
\end{equation*}

\begin{lemma}
\label{th:caseA-cgtheta1}
Let $ φ ∈ \CG_Γ $. If $ f_0 (φ) ≠ 0 $, then $ φ ∈ \GL φ^{(0)} $.
\end{lemma}

\begin{proof}
By assumption we have $ φ_{i, j} ≠ 0 $ for all $ 0 ≤ i, j ≤ n $. By gauging $ φ $, we can achieve $ φ_{1, i} = 1 $ for all $ i $. Since $ φ $ lies in $ \CG_Γ $, it satisfies the coherence and symmetry relations, and in particular $ φ_{i+1, j} φ_{1, i} = φ_{1, i+j} φ_{i, j} $ for all $ 0 ≤ i, j ≤ n $. We inductively deduce that $ φ_{i, j} = 1 $ for all $ 0 ≤ i, j ≤ n $ and conclude that $ φ $ lies in the orbit $ \GL φ^{(0)} $. This finishes the proof.
\end{proof}

\subsection{Stage 6: Identification of the $ θ_2 $-semistable locus}
\label{sec:caseA-6}
We are now ready to investigate the $ θ_2 $-semistable locus of $ \CG_Γ × ℂ^2 $. Our strategy is to utilize the $ \GL $-equivariant map $ R: \CG_Γ × ℂ^2 → \Rep(Π_Q, α) $. The standard argument of semiinvariants shows that if $ R(φ, x) $ is $ θ_2 $-semistable, then $ (φ, x) $ is $ θ_2 $-semistable as well. As we shall prove, the converse statement is also true.

\begin{figure}
\centering
\begin{subfigure}{0.99\linewidth}
\centering
\begin{minipage}{0.4\linewidth}
\begin{tikzpicture}
\path[fill] (90:2) circle[radius=0.05];
\foreach \i in {0, 2, 4, 5, 6, 8, 9, 11, 13, 14} {\path[draw, bend right, ->] ($ (22.5*\i + 3:2.1) $) to ($ (22.5*\i + 19.5:2.1) $);};
\foreach \i in {} {\path[draw, bend left, <-] ($ (22.5*\i + 3:2) $) to ($ (22.5*\i + 19.5:2) $);};
\path ($ (22.5*0:2.4) $) node {12};
\path ($ (22.5*1:2.4) $) node {12};
\path ($ (22.5*2:2.4) $) node {14};
\path ($ (22.5*3:2.4) $) node {14};
\path ($ (22.5*5:2.4) $) node {0};
\path ($ (22.5*6:2.4) $) node {0};
\path ($ (22.5*7:2.4) $) node {0};
\path ($ (22.5*8:2.4) $) node {4};
\path ($ (22.5*9:2.4) $) node {4};
\path ($ (22.5*10:2.4) $) node {4};
\path ($ (22.5*11:2.4) $) node {7};
\path ($ (22.5*12:2.4) $) node {7};
\path ($ (22.5*13:2.4) $) node {9};
\path ($ (22.5*14:2.4) $) node {9};
\path ($ (22.5*15:2.4) $) node {9};
\end{tikzpicture}
\end{minipage}
%
\begin{minipage}{0.4\linewidth}
\begin{tikzpicture}
\path[fill] (90:2) circle[radius=0.05];
\foreach \i in {0, 4, 5, 6, 8, 9, 11, 13, 14,  2, 3} {\path[draw, bend right, ->] ($ (22.5*\i + 3:2.1) $) to ($ (22.5*\i + 19.5:2.1) $);};
\foreach \i in {0, 4, 5, 6, 7, 8, 9, 10, 11, 12, 13, 14, 15} {\path[draw, bend left, <-] ($ (22.5*\i + 3:2) $) to ($ (22.5*\i + 19.5:2) $);};
\path ($ (22.5*0:2.4) $) node {12};
\path ($ (22.5*1:2.4) $) node {12};
\path ($ (22.5*2:2.4) $) node {14};
\path ($ (22.5*3:2.4) $) node {14};
\path ($ (22.5*5:2.4) $) node {0};
\path ($ (22.5*6:2.4) $) node {0};
\path ($ (22.5*7:2.4) $) node {0};
\path ($ (22.5*8:2.4) $) node {4};
\path ($ (22.5*9:2.4) $) node {4};
\path ($ (22.5*10:2.4) $) node {4};
\path ($ (22.5*11:2.4) $) node {7};
\path ($ (22.5*12:2.4) $) node {7};
\path ($ (22.5*13:2.4) $) node {9};
\path ($ (22.5*14:2.4) $) node {9};
\path ($ (22.5*15:2.4) $) node {9};
\end{tikzpicture}
\end{minipage}
\caption{This figure depicts two sample representations and their exponents $ α_i $ in the $ x_n = 0 $ case.}
\label{fig:caseA-exponents-rule-1}
\end{subfigure}
%
\begin{subfigure}{0.99\linewidth}
\centering
\begin{minipage}{0.4\linewidth}
\begin{tikzpicture}
\path[fill] (90:2) circle[radius=0.05];
\foreach \i in {0, 4, 5, 6, 8, 9, 11, 13, 14} {\path[draw, bend right, ->] ($ (22.5*\i + 3:2.1) $) to ($ (22.5*\i + 19.5:2.1) $);};
\foreach \i in {2, 3} {\path[draw, bend left, <-] ($ (22.5*\i + 3:2) $) to ($ (22.5*\i + 19.5:2) $);};
\path ($ (22.5*0:2.4) $) node {12};
\path ($ (22.5*1:2.4) $) node {12};
\path ($ (22.5*2:2.4) $) node {0};
\path ($ (22.5*3:2.4) $) node {0};
\path ($ (22.5*5:2.4) $) node {0};
\path ($ (22.5*6:2.4) $) node {0};
\path ($ (22.5*7:2.4) $) node {0};
\path ($ (22.5*8:2.4) $) node {4};
\path ($ (22.5*9:2.4) $) node {4};
\path ($ (22.5*10:2.4) $) node {4};
\path ($ (22.5*11:2.4) $) node {7};
\path ($ (22.5*12:2.4) $) node {7};
\path ($ (22.5*13:2.4) $) node {9};
\path ($ (22.5*14:2.4) $) node {9};
\path ($ (22.5*15:2.4) $) node {9};
\end{tikzpicture}
\end{minipage}
%
\begin{minipage}{0.4\linewidth}
\begin{tikzpicture}
\path[fill] (90:2) circle[radius=0.05];
\foreach \i in {0, 4, 5, 6, 8, 9, 11, 13, 14,  1, 2, 3} {\path[draw, bend right, ->] ($ (22.5*\i + 3:2.1) $) to ($ (22.5*\i + 19.5:2.1) $);};
\foreach \i in {2, 3,   0, 4, 5, 6, 7, 8, 9, 10, 11, 12, 13, 14, 15} {\path[draw, bend left, <-] ($ (22.5*\i + 3:2) $) to ($ (22.5*\i + 19.5:2) $);};
\path ($ (22.5*0:2.4) $) node {12};
\path ($ (22.5*1:2.4) $) node {12};
\path ($ (22.5*2:2.4) $) node {0};
\path ($ (22.5*3:2.4) $) node {0};
\path ($ (22.5*5:2.4) $) node {0};
\path ($ (22.5*6:2.4) $) node {0};
\path ($ (22.5*7:2.4) $) node {0};
\path ($ (22.5*8:2.4) $) node {4};
\path ($ (22.5*9:2.4) $) node {4};
\path ($ (22.5*10:2.4) $) node {4};
\path ($ (22.5*11:2.4) $) node {7};
\path ($ (22.5*12:2.4) $) node {7};
\path ($ (22.5*13:2.4) $) node {9};
\path ($ (22.5*14:2.4) $) node {9};
\path ($ (22.5*15:2.4) $) node {9};
\end{tikzpicture}
\end{minipage}
\caption{This figure depicts two sample representations and their exponents $ α_i $ in the $ x_1, x_n ≠ 0 $ case.}
\label{fig:caseA-exponents-rule-2}
\end{subfigure}
\caption{This figure visually depicts the rule for determining the exponents $ α_i $ used in the proof of \autoref{th:caseA-theta2-semistable}. Part (\subref{fig:caseA-exponents-rule-1}) concerns the case $ x_n = 0 $ and part (\subref{fig:caseA-exponents-rule-2}) concerns the case $ x_1, x_n ≠ 0 $. Each of the four graphics qualitatively depicts a representation of the $ A_n $ quiver with $ n = 15 $. The solid arrows are meant to depict a nonzero arrow value, and all missing arrows are meant to depict the zero arrow value. The special vertex of the quiver is highlighted and positioned in the top of the images. The number at the vertex $ i $ is the exponent $ α_i $ which we associate with that vertex. The representations on the left have more zero arrows than the representations on the right, but both representations in each row have the same exponents $ α_i $.}
\label{fig:caseA-exponents-rule}
\end{figure}
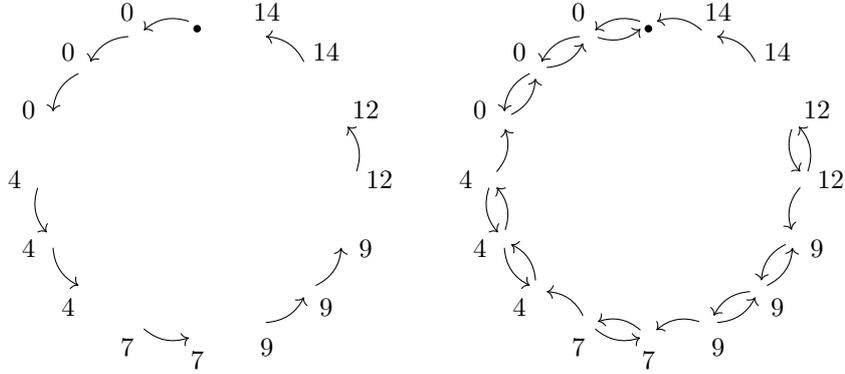
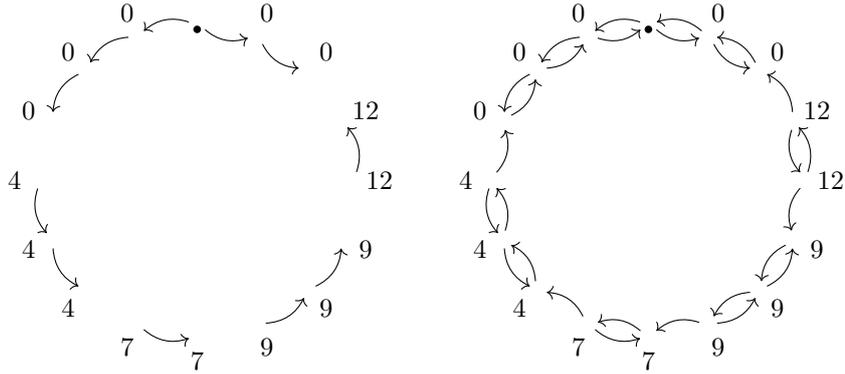

\begin{lemma}
\label{th:caseA-theta2-semistable}
Let $ (φ, x) ∈ \CG_Γ × ℂ^2 $ be $ θ_2 $-semistable. Then $ R(φ, x) ∈ \Rep(Π_Q, α) $ is $ θ_2 $-semistable.
\end{lemma}

\begin{proof}
Let us first sketch the structure of the proof. The basic setup is to prove the statement by contraposition. Therefore we shall assume that $ R(φ, x) $ is not $ θ_2 $-semistable, and shall deduce that $ (φ, x) $ is not $ θ_2 $-semistable. We prove that $ (φ, x) $ is not $ θ_2 $-semistable by applying Hilbert-Mumford criterion. This entails that we construct an explicit one-parameter subgroup $ g(t) = (t^{α_1}, …, t^{α_n}) ⊂ \GL $ which pairs positively with $ θ_2 $ and for which $ \lim_{t → 0} g(t)^{-1} (φ, x) $ exists within $ \CG_Γ × ℂ^2 $. Since $ \CG_Γ $ is a closed subset of an affine space, it suffices to simply show that $ g(t)^{-1} (φ, x) $ converges within the affine space. Our specific construction of the exponents $ α_1, …, α_n $ depends on several factors, and we shall proceed by case distinction. We write $ x = (x_1, x_n) ∈ U_1 ⊕ U_n = ℂ^2 $.

\begin{itemize}
\item Consider the case that $ x_n = 0 $. We define integers
\begin{align*}
α_1 &= 0, \\
α_i &= \max \{j \running 1 ≤ j ≤ i-1 \text{ and } φ_{j, 1} = 0\} + 1, \quad i = 2, …, n.
\end{align*}
If the set over which the maximum is taken is empty, then we have $ φ_{j, 1} ≠ 0 $ for all $ 1 ≤ j ≤ i-1 $ and we define $ α_i = 1 $. For convenience purposes, we may write $ α_0 = 0 $. To restate the definition in other words, the value $ α_i $ is simply the number smaller than $ i $ such that $ φ_{α_i - 1, 1} = 0 $ and $ φ_{α_i, 1}, …, φ_{i-1, 1} ≠ 0 $. Let us list a few important properties of the sequence $ (α_i) $:
\begin{align*}
& 1 ≤ α_i ≤ i, \\
& α_i ≤ α_{i+1}, \\
& α_i ≤ i-1 ~\Longrightarrow~ α_i = α_{i-1}.
\end{align*}
We choose the one-parameter subgroup as $ g(t) = (t^{α_1}, …, t^{α_n}) $. Since $ R(φ, x) $ is not $ θ_2 $-semistable, at least one of the values $ φ_{1, 1}, …, φ_{1, n-1} $ vanishes and therefore we have $ 1 ≤ α_n $. Therefore the group $ g(t) $ pairs positively with $ θ_2 $. The limit $ \lim_{t → 0} g(t)^{-1} x $ exists, since we assumed $ x_n = 0 $ and we have $ α_1 = 1 $. We now claim that $ \lim_{t → 0} g(t)^{-1} φ $ exists.

In an inductive fashion ranging over the pairs $ (i, j) $ with $ 1 ≤ i, j ≤ n $ we shall prove that $ φ_{i, j} ≠ 0 $ implies $ α_{i+j} ≤ α_i + α_j $. Recall that in the notation $ α_{i+j} $ the index is taken modulo $ n+1 $, and within this convention $ α_0 $ shall be interpreted as zero.

Regard the case $ i+j ≥ n+1 $. Then $ i+j-(n+1) ≤ i $ and thus $ α_{i+j} ≤ α_i $.

Regard the case $ i+j < n $ and $ α_i = i $ and $ α_j = j $. Then $ α_{i+j} ≤ i+j = α_i + α_j $, which finishes the case.

Regard the case $ i+j < n $ and $ α_{i+j} = i+j $. The coherence relations yield $ φ_{i, j} φ_{1, j-1} = φ_{i, j-1} φ_{1, i+j-1} $. Since $ φ_{i, j} ≠ 0 $ by assumption and $ φ_{1, i+j-1} = 0 $ by assumption that $ α_{i+j} = i+j $, we conclude $ φ_{1, j-1} = 0 $. This proves $ α_j = j $. Similarly we conclude $ α_i = i $. Finally, we have $ α_{i+j} = α_i + α_j $, which finishes the case.

Regard the case $ i+j < n $ and $ α_j < j $ and $ α_{i+j} < i+j $. We conlude $ i ≥ 2 $. By induction we can assume that we have already proven the statement for the pair $ (i-1, j) $. Note that $ φ_{1, j-1} φ_{i, j} = φ_{i, j-1} φ_{1, i+j-1} $. Since $ φ_{1, j-1} ≠ 0 $ and $ φ_{i, j} ≠ 0 $ and $ φ_{1, i+j-1} ≠ 0 $, we deduce $ φ_{i, j-1} ≠ 0 $. By induction we can assume that we have already proven the statement for the pair $ (i, j-1) $. Therefore we have $ α_{i+j-1} ≤ α_i + α_{j-1} $. Since $ α_i ≤ i-1 $ and $ α_{i+j} ≤ i+j-1 $, we deduce that $ α_{i+j-1} = α_{i+j} $ and $ α_{j-1} = α_j $, which finishes the case.

Regard the case $ i+j < n $ and $ α_i < i $ and $ α_{i+j} < i+j $. This case is dealt with in a manner analogous to the case $ α_j ≤ j-1 $.

We have now exhausted all cases and conclude that $ φ_{i, j} ≠ 0 $ implies $ α_{i+j} ≤ α_i + α_j $.

\item Consider the case that $ x_1 = 0 $. This case is dealt with in an analogous fashion.

\item Consider the case that $ x_1 ≠ 0 $ and $ x_n ≠ 0 $. We shall define the numbers $ α_i $ and pick $ g(t) = (t^{α_1}, …, t^{α_n}) $. For $ i = 1, …, n $ define $ α_i $ as
\begin{equation*}
α_i = \begin{cases} 0, & \quad \text{if } φ_{1, 1}, …, φ_{i-1, 1} ≠ 0, \\
0, & \quad \text{if } φ_{i+1, n}, …, φ_{n, n} ≠ 0, \\
\max \{j \running 1 ≤ j ≤ i-1 \text{ and } φ_{j, 1} = 0\} + 1, & \quad \text{else.} \end{cases}
\end{equation*}
We shall say that $ α_i $ is defined by the first, second or third alternative, depending on which of the three rules is used. Let us note that in the third alternative, the set over which the maximum is taken is indeed nonempty since the condition of the first alternative does not hold. We observe that all $ α_i $ are non-negative and at least one of them is nonzero. Indeed, since $ R(φ, x) $ is not $ θ_2 $-semistable and $ x_1, x_n ≠ 0 $, there exists an index $ 1 ≤ j ≤ n $ such that one of $ φ_{1, 1}, …, φ_{j-1, 1} $ is zero or one of $ φ_{n, n}, …, φ_{j+1, n} $ is zero. Finally, note that $ 0 ≤ α_i ≤ i $ for all $ i = 1, …, n $. We shall now prove that for any pair $ (i, j) $ with $ 1 ≤ i, j ≤ n $ we have that $ φ_{i, j} ≠ 0 $ implies $ α_{i+j} ≤ α_i + α_j $.

Regard the case that $ i+j ≤ n $ and $ α_i $ or $ α_j $ falls under the second alternative. Then we want to show $ α_{i+j} = 0 $. Indeed, we have $ i+j+1 ≥ i+1 $ and $ i+j+1 ≥ j+1 $, thus $ φ_{n, n}, …, φ_{i+j+1, n} ≠ 0 $, thus $ α_{i+j} = 0 $. This finishes the case.

Regard the case that $ i+j ≤ n $ and $ α_j < j $, and both $ α_i $ and $ α_j $ individually fall under under the first or third alternatives. We show by induction over $ 1 ≤ t ≤ j - α_j $ that $ φ_{i+j-t, 1} ≠ 0 $ and $ φ_{i, j-t} ≠ 0 $. The trick is to use the coherence relation $ φ_{i+j-t, 1} φ_{i, j-t} = φ_{j-t, 1} φ_{i, j-(t-1)} $. In case $ t = 1 $, we have $ φ_{i, j} ≠ 0 $ and by definition of $ α_j $ and $ α_j < j $ we have $ φ_{j-1, 1} ≠ 0 $, which proves the base case $ t = 1 $. For any other $ t $, by definition of $ α_j $ and $ t ≤ j - α_j $ we have $ φ_{j-t, 1} ≠ 0 $ and by induction hypothesis we have $ φ_{i, j-(t-1)} ≠ 0 $. Applying the coherence relation finishes the induction step.

We conclude $ φ_{i+j-1, 1}, …, φ_{i+j-(j-α_j), 1} = φ_{i+α_j, 1} ≠ 0 $, and also $ φ_{i, α_j} ≠ 0 $. We want to prove $ α_{i+j} ≤ α_i + α_j $. In case $ α_i = i $, we obtain $ φ_{i+j-1, 1}, …, φ_{α_i+α_j, 1} ≠ 0 $ and thus $ α_{i+j} ≤ α_i + α_j $ as desired. In case $ α_i < i $, then since $ α_i $ and $ α_j $ fall under the first or third alternative, we apply what we have just proven to $ φ_{α_j, i} ≠ 0 $. We then obtain $ φ_{i+α_j-1, 1}, …, φ_{i+α_j-(i-α_i), 1} = φ_{α_i+α_j, 1} ≠ 0 $ and conclude $ α_{i+j} ≤ α_i + α_j $ as desired. This finishes the cases.

Regard the case that $ i+j ≤ n $ and $ α_i < i $, and both $ α_i $ and $ α_j $ individually fall under under the first or third alternatives. Under the operation of swapping $ i $ and $ j $, this case is identical to the earlier case.

Regard the case that $ i+j ≤ n $ and $ α_i = i $ and $ α_j = j $. Then $ α_{i+j} ≤ i+j = α_i + α_j $ and we are done.

Regard the case that $ i+j > n+1 $, and that $ α_i $ or $ α_j $ falls under the first or third alternative. Then $ i+j-(n+1) ≤ i $ and $ i+j-(n+1) ≤ j $. If $ α_i $ falls under the first or third alternative, we conclude $ α_{i+j} ≤ α_i ≤ α_i + α_j $. If $ α_j $ falls under the first or third alternative, we conclude $ α_{i+j} ≤ α_j ≤ α_i + α_j $. This proves the case.

Regard the case that $ i+j > n+1 $, and that $ α_i $ and $ α_j $ fall under the second alternative. We have $ φ_{n, n}, …, φ_{i+1, n} ≠ 0 $ and $ φ_{n, n}, …, φ_{j+1, n} ≠ 0 $. We shall prove $ α_{i+j} = 0 $ by showing $ φ_{n, n}, …, φ_{i+j-(n+1)+1, n} ≠ 0 $. We shall assume that $ i ≥ j $.

Let us prove by induction that for all $ i ≤ s ≤ n $ we have $ φ_{s, j} ≠ 0 $ and $ φ_{s+j+1, n} ≠ 0 $. The trick is to apply for any $ i ≤ s ≤ n $ the coherence relation to $ U_n ¤ U_{s+1} ¤ U_j $ and obtain $ φ_{n, s+1} φ_{s, j} = φ_{s+1, j} φ_{s+j+1, n} $. Now in the base case $ s = i $, we have $ φ_{n, i+1} ≠ 0 $ and $ φ_{i, j} ≠ 0 $ by assumption, thus $ φ_{s, j} ≠ 0 $ and $ φ_{s+j+1, n} ≠ 0 $. As induction step, we are allowed to assume $ φ_{s, j} ≠ 0 $ and $ φ_{s+j+1, n} ≠ 0 $. Together with $ φ_{s+1, n} ≠ 0 $ we conclude $ φ_{s+1, j} ≠ 0 $ and $ φ_{s+j+1, n} ≠ 0 $. This finishes the induction.

Finally, we conclude $ φ_{i+j+1, n}, …, φ_{n+j+1, n} ≠ 0 $. Note that the index $ n+j+1 $ is just $ j $. Together with our assumption $ φ_{j+1, n}, …, φ_{n, n} ≠ 0 $ this proves that $ α_{i+j} = 0 $. This finishes the case.
\end{itemize}
Finally, in all cases we have constructed a one-parameter subgroup $ g(t) ⊂ \GL $ which pairs positively with $ θ_2 $ and proven that $ \lim_{t → 0} g(t)^{-1} (φ, x) $ exists. This finishes the proof.
\end{proof}

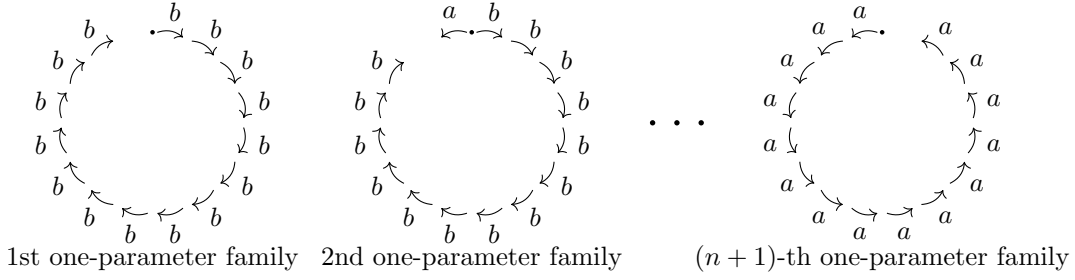
\begin{figure}
\centering
\begin{tikzpicture}[scale=0.6]
\begin{scope}
\path[fill] (90:2) circle[radius=0.05];
\foreach \i in {} {\path[draw, bend right, ->] ($ (22.5*\i + 3:2) $) to ($ (22.5*\i + 19.5:2) $); \path ($ (22.5*\i + 11.25:2.5) $) node {$ a $};};
\foreach \i in {3, 2, 1, 0, 15, 14, 13, 12, 11, 10, 9, 8, 7, 6, 5} {\path[draw, bend right, <-] ($ (22.5*\i + 3:2) $) to ($ (22.5*\i + 19.5:2) $); \path ($ (22.5*\i + 11.25:2.5) $) node {$ b $};};
\path (0, -3) node {1st one-parameter family};
\end{scope}
\begin{scope}[shift={(7, 0)}]
\path[fill] (90:2) circle[radius=0.05];
\foreach \i in {4} {\path[draw, bend right, ->] ($ (22.5*\i + 3:2) $) to ($ (22.5*\i + 19.5:2) $); \path ($ (22.5*\i + 11.25:2.5) $) node {$ a $};};
\foreach \i in {3, 2, 1, 0, 15, 14, 13, 12, 11, 10, 9, 8, 7, 6} {\path[draw, bend right, <-] ($ (22.5*\i + 3:2) $) to ($ (22.5*\i + 19.5:2) $); \path ($ (22.5*\i + 11.25:2.5) $) node {$ b $};};
\path (0, -3) node {2nd one-parameter family};
\end{scope}
\path (11.5, 0) node {\huge …};
\begin{scope}[shift={(16, 0)}]
\path[fill] (90:2) circle[radius=0.05];
\foreach \i in {4, 5, 6, 7, 8, 9, 10, 11, 12, 13, 14, 15, 0, 1, 2} {\path[draw, bend right, ->] ($ (22.5*\i + 3:2) $) to ($ (22.5*\i + 19.5:2) $); \path ($ (22.5*\i + 11.25:2.5) $) node {$ a $};};
\foreach \i in {} {\path[draw, bend right, <-] ($ (22.5*\i + 3:2) $) to ($ (22.5*\i + 19.5:2) $); \path ($ (22.5*\i + 11.25:2.5) $) node {$ b $};};
\path (0, -3) node {$ (n+1) $-th one-parameter family};
\end{scope}
\end{tikzpicture}
\caption{This figure depicts the $ (n+1) $-many one-parameter families of $ θ_2 $-representations that lie in the nullcone. We have expressed the representations $ ρ_{a, b} $ in these families in terms of two parameters $ a, b ∈ ℂ $ for notational convenience, even though one parameter would suffice due to gauging.}
\end{figure}

\begin{lemma}
\label{th:caseA-theta2surj}
Let $ ρ ∈ \Rep(Π_Q, α) $ be a $ θ_2 $-semistable element in the nullcone. Then $ ρ $ lies in the image of $ R $.
\end{lemma}

\begin{proof}
Regard a representation $ ρ_{a, b} $ in the $ k $-th one-parameter family. We define $ g(t) = (t^{α_1}, …, t^{α_n}) $ with
\begin{align*}
α_i &= -i(n-k+1), \quad 1 ≤ i ≤ k-1, \\
α_i &= -(n-i+1)(k-1), \quad k ≤ i ≤ n.
\end{align*}
We observe that $ α_i + α_j ≤ α_{i+j} $ so that $ φ = \lim_{t → 0} g(t) φ^{(0)} $ exists. We calculate
\begin{align*}
φ_{i, 1} &= \lim_{t → 0} t^{-(i+1)(n-k+1)} t^{i(n-k+1)} t^{(n-k+1)} = 1, \quad 1 ≤ i ≤ k-2, \\
φ_{i, 1} &= \lim_{t → 0} t^{-(n-i)(k-1)} t^{(n-i+1)(k-1)} t^{(n-k+1)} = 0, \quad k-1 ≤ i ≤ n, \\
φ_{i, n} &= \lim_{t → 0} t^{-(n-i+2)(k-1)} t^{(n-i+1)(k-1)} t^{(k-1)} = 1, \quad k+1 ≤ i ≤ n, \\
φ_{i, n} &= \lim_{t → 0} t^{-(i-1)(n-k+1)} t^{i(n-k+1)} t^{(k-1)} = 0, \quad 1 ≤ i ≤ k-1, \\
φ_{k, n} &= \lim_{t → 0} t^{-(k-1)(n-k+1)} t^{(n-k+1)(k-1)} t^{(k-1)} = 0, \quad i = k.
\end{align*}
We conclude that $ ρ_{a, b} = R(φ, (a, b)) $. This finishes the proof.
\end{proof}

\subsection{The $ θ_2 $-semiinvariants on $ \CG_Γ × ℂ^2 $}
In this section, we shall determine the $ θ_2 $-semiinvariants on $ \CG_Γ × ℂ^2 $. This investigation is not strictly necessary for achieving the goal of this paper. However, it allows us to visit an alternative way of proving stage 6. We formulate the goal of this investigation in the following question:

\begin{question}
Given $ k ≥ 1 $, is it true that all $ kθ_2 $-semiinvariants on $ \CG_Γ × ℂ^2 $ are pullbacks from $ kθ_2 $-semiinvariants on $ \Rep(Π_Q, α) $ along the map $ R: \CG_Γ × ℂ^2 → \Rep(Π_Q, α) $?
\end{question}

In the present section, we answer this question positively for $ k = 1 $. Slightly more effort would most likely establish the statement for any $ k ≥ 1 $. Once a positive answer for the question is established, then \autoref{th:caseA-theta2-semistable} follows automatically. Indeed, if $ (φ, x) $ is $ θ_2 $-semistable, then there exists a $ kθ_2 $-semiinvariant $ f $ for some $ k ≥ 1 $ such that $ f(φ, x) ≠ 0 $, and by the positive answer it is then of the form $ f = g ∘ R $ for some $ kθ_2 $-semiinvariant on $ \Rep(Π_Q, α) $, so that $ g(R(φ, x)) ≠ 0 $ and we conclude that $ R(φ, x) $ is $ θ_2 $-semistable. In other words, a positive answer to the question proves that pullback along $ R $ provides an isomorphism between the semiinvariant rings $ \bigoplus_{k ≥ 0} ℂ[\Rep(Π_Q, α)]_{kθ_2} \isoto \bigoplus_{k ≥ 0} ℂ[\CG_Γ × ℂ^2]_{kθ_2} $ and correspondingly also an isomorphism between their $ \Proj $ varieties. This shows how a positive answer to the question provides a more insightful proof of stage 6.

We are now ready to start our investigation of the $ θ_2 $-semiinvariants. Our first lemma concerns the approximate shape of the semiinvariants:

\begin{lemma}
\label{th:caseA-theta2si-restr}
Let $ f: \CG_Γ × ℂ^2 → ℂ $ be a $ θ_2 $-semiinvariant. In case $ n $ is odd, the restriction $ π(f) $ lies in the linear span of the monomials
\begin{align*}
X^{l + (n+1)/2 + m(n+1)} Y^l, &\quad l, m ≥ 0, \\
X^l Y^{l + (n+1)/2 + m(n+1)}, &\quad l, m ≥ 0.
\end{align*}
In case $ n $ is even, the restriction $ π(f) $ lies in $ ℂ[X, Y]^{Γ} $, in other words, in the linear span of the monomials
\begin{align*}
X^{l + m(n+1)} Y^l, &\quad l, m ≥ 0, \\
X^l Y^{l + m(n+1)}, &\quad l, m ≥ 0.
\end{align*}
\end{lemma}

\begin{proof}
The trick is to exploit the stabilizer $ Γ = \Stab_{\GL} (φ^{(0)}) $. For every $ g ∈ Γ $, we have
\begin{equation*}
π(f)(gx) = f(φ^{(0)}, gx) = θ_2 (g) f(φ^{(0)}, x) = θ_2 (g) π(f)(x).
\end{equation*}
Recall that $ Γ ⊂ \GL $ is the cyclic group generated by the element $ σ = (e^{2πi/(n+1)}, …, e^{2πin/(n+1)}) ∈ \GL $ and we have $ θ_2 (σ) = (-1)^{n} $. If $ n $ is even, we conclude that $ π(f) ∈ ℂ[X, Y]^Γ $ which finishes the even case.

Let us now assume $ n $ is odd. We shall find all functions $ h ∈ ℂ[X, Y] $ which satisfy the rule $ h(gx) = θ_2 (g) h(x) $. These are precisely the functions which lie in the image of the Reynolds-style operator $ R: ℂ[X, Y] → ℂ[X, Y] $ given by
\begin{equation*}
R(h) = \frac{1}{|Γ|} \sum_{g ∈ Γ} θ_2 (g) g.h.
\end{equation*}
It remains to determine the image of the operator $ R $.  Let us now simply feed the monomials $ X^p Y^q $ into $ R $:
\begin{align*}
R(X^p Y^q) &= \frac{1}{n+1} \sum_{j = 0}^n e^{-2πijp/(n+1) + 2πijq/(n+1) + πijn} X^p Y^q.
\end{align*}
If $ (p-q)/(n+1) + n/2 $ is not an integer, then the sum is zero since the geometric series summation exponent $ (p-q) + n(n+1)/2 $ is an integer. Therefore we only regard the case that $ (p-q)/(n+1) + n/2 $ is an integer. Since $ n $ is odd, the image of $ R $ is spanned by the monomials $ X^p Y^q $ with $ p-q ∈ ℤ $ being $ (n+1)/2 $ plus an integer multiple of $ n+1 $. This finishes the proof.
\end{proof}

We are aiming to enumerate the $ θ_2 $-semiinvariants on $ \CG_Γ × ℂ^2 $. As we have seen in \autoref{th:caseA-theta2si-restr}, the restriction of a $ θ_2 $-semiinvariants has a very specific shape. The question arises whether any function with this shape is also the restriction of a $ θ_2 $-semiinvariant. The following lemma is the technical step in this direction.

\begin{lemma}
\label{th:caseA-SI-pbSI}
The pullback $ θ_2 $-semiinvariants are precisely the functions $ f: \CG_Γ × ℂ^2 → ℂ $ whose restriction $ π(f) $ satisfies
\begin{equation}
\label{eq:caseA-SI-odd-module}
π(f) ∈ ℂ[X, Y]^{Γ} · \vspan\{X^{(n-q)(n-q+1)/2} Y^{q(q+1)/2}\}_{q = 0, …, n}.
\end{equation}
\end{lemma}

\begin{proof}
We start by recalling that we have the standard $ θ_2 $-semiinvariants of $ \Rep(Π_Q, α) $ given for $ q = 0, …, n $ by
\begin{equation*}
f_q = A_{n-q}^1 … A_1^{n-q} · A_{n-q+2}^1 … A_{n+1}^q.
\end{equation*}
It is known that every $ θ_2 $-semiinvariant is a $ ℂ[\Rep(Π_Q, α)]^{\GL} $-linear combination of these standard semiinvariants. Note that the pullback restriction $ π(R^* f_q) $ is precisely the monomial $ X^{(n-q)(n-q+1)/2} Y^{q(q+1)/2} $. We shall now prove both parts of the claim separately.

For the first part, let $ g ∈ ℂ[\Rep(Π_Q, α)]_{θ_2} $ be any $ θ_2 $-semiinvariant. Then we can write $ g = \sum_{q = 0}^n g_q f_q $ with $ g_q ∈ ℂ[\Rep(Π_Q, α)]^{\GL} $. Thus,
\begin{equation*}
π(R^* g) = \sum_{q = 0}^n π(R^* g_q) π(R^* f_q).
\end{equation*}
Recall that $ π(R^* g_q) ∈ ℂ[X, Y]^{Γ} $ and $ π(R^* f_q) = X^{(n-q)(n-q+1)/2} Y^{q(q+1)/2} $. This finishes the first part of the proof.

For the second part, let $ f: \CG_Γ × ℂ^2 → ℂ $ be any function such that
\begin{equation*}
π(f) = \sum_{q = 0}^n g_q X^{(n-q)(n-q+1)/2} Y^{q(q+1)/2}, \text{ with } g_q ∈ ℂ[X, Y]^{\GL}.
\end{equation*}
Since $ R^* $ is an isomorphism on invariants, there exist invariants $ h_q ∈ ℂ[\Rep(Π_Q, α)]^{\GL} $ such that $ R^* h_q = g_q $. Then we immediately see that $ π(f) = π(R^* (\sum_{q = 0}^n h_q f_q)) $. Since $ π $ is injective among $ θ_2 $-semiinvariants, we conclude that $ f $ is a pullback semiinvariant. This finishes the second part of the claim.
\end{proof}

\autoref{th:caseA-SI-pbSI} characterizes the pullback $ θ_2 $-semiinvariants very neatly by describing the shape of their restriction. As we shall see in \autoref{th:caseA-SI-odd-pbrestr}, all $ θ_2 $-semiinvariants are in fact of this shape. Therefore all $ θ_2 $-semiinvariants actually arise as pullbacks.

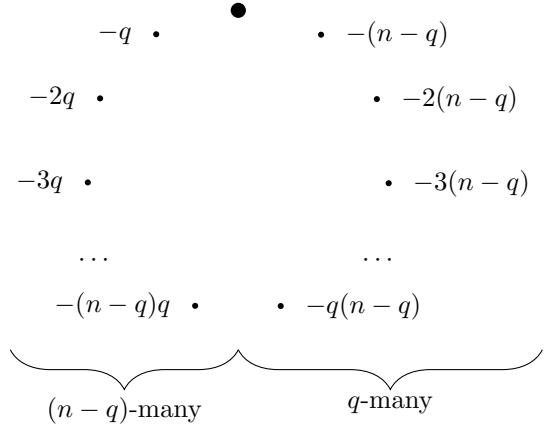
\begin{figure}
\centering
\begin{minipage}{0.49\linewidth}
\begin{tikzpicture}
\path[fill] (90:2) circle[radius=0.1];
\path[fill] ($ (90+1*32.7:2) $) circle[radius=0.04] node[left, shift={(left:0.2)}] {$ -q $};
\path[fill] ($ (90+2*32.7:2) $) circle[radius=0.04] node[left, shift={(left:0.2)}] {$ -2q $};
\path[fill] ($ (90+3*32.7:2) $) circle[radius=0.04] node[left, shift={(left:0.2)}] {$ -3q $};
\path[fill] ($ (90+4*32.7:2) $) node[left] {$ \cdots $};
\path[fill] ($ (90+5*32.7:2) $) circle[radius=0.04] node[left, shift={(left:0.2)}] {$ -(n-q)q $};
\path[fill] ($ (90+6*32.7:2) $) circle[radius=0.04] node[right, shift={(right:0.2)}] {$ -q(n-q) $};
\path[fill] ($ (90+7*32.7:2) $) node[right] {$ \cdots $};
\path[fill] ($ (90+8*32.7:2) $) circle[radius=0.04] node[right, shift={(right:0.2)}] {$ -3(n-q) $};
\path[fill] ($ (90+9*32.7:2) $) circle[radius=0.04] node[right, shift={(right:0.2)}] {$ -2(n-q) $};
\path[fill] ($ (90+10*32.7:2) $) circle[radius=0.04] node[right, shift={(right:0.2)}] {$ -(n-q) $};
\path[draw, decoration={brace, amplitude=1.5em, mirror}, decorate] (-3, -2.5) to node[midway, below, shift={(down:0.5)}] {$ (n-q) $-many} (0, -2.5);
\path[draw, decoration={brace, amplitude=1.5em, mirror}, decorate] (0, -2.5) to node[midway, below, shift={(down:0.5)}] {$ q $-many} (4, -2.5);
\end{tikzpicture}
\end{minipage}
\caption{This figure qualitatively depicts the exponents $ α_j $ used in the proof of \autoref{th:caseA-SI-odd-pbrestr}. The highlighted vertex in the top of the figure is the special vertex of the quiver. There are $ (n-q) $-many vertices whose exponents are increasing negative multiples of $ q $, and $ q $-many vertices whose exponents are increasing negative multiples of $ n-q $. The proof uses exponents which are slightly different from the ones depicted in this figure in order to ease the distinction of exponents that are responsible for divergence.}
\label{fig:caseA-SI-exponents}
\end{figure}

\begin{lemma}
\label{th:caseA-SI-odd-pbrestr}
Regard the $ A_n $ case with odd $ n $. Then every $ θ_2 $-semiinvariant on $ \CG_Γ × ℂ^2 $ is the pullback of a $ θ_2 $-semiinvariant on $ \Rep(Π_Q, α) $.
\end{lemma}

\begin{proof}
Let $ f ∈ \CG_Γ × ℂ^2 → ℂ $ be any $ θ_2 $-semiinvariant. We already know that $ π(f) $ is of the form
\begin{equation*}
π(f) = \sum_{l, m ≥ 0} α_{l, m} X^{l + (n+1)/2 + m(n+1)} Y^l + \sum_{l, m ≥ 0} β_{l, m} X^l Y^{l + (n+1)/2 + m(n+1)}.
\end{equation*}
Here $ α_{l, m} $ and $ β_{l, m} $ are scalar coefficients.

We divide the remainder of the proof into three parts. The first part is to show that we can assume that $ α_{l, m} = β_{l, m} = 0 $ whenever $ m ≥ (n-1)/2 $ or $ l ≥ q(q+1)/2 $ with $ q ≔ (n-1)/2 - m $. The second part is to show that for arbitrary but fixed $ l_0, m_0 ≥ 0 $ assuming $ q ≔ (n-1)/2 - m_0 ≥ 0 $ and $ l_0 < q (q+1)/2 $ we have $ α_{l_0, m_0} = 0 $. The final part draws the right conclusions and finishes the proof.

We now proceed to the first part of the proof. We show that the summands with $ m ≥ (n-1)/2 $ or $ l ≥ q(q+1)/2 $ with $ q ≔ (n-1)/2 - m $ are already pullback semiinvariants by themselves. To see this, let us look at the monomial $ X^{l + (n+1)/2 + m(n+1)} Y^l $. In case $ m ≥ (n-1)/2 $, we can write
\begin{equation*}
X^{l + (n+1)/2 + m(n+1)} Y^l = X^{(n+1)(m-(n-1)/2)} · (XY)^l · X^{n(n+1)/2}.
\end{equation*}
In case $ m ≤ (n-1)/2 $ and $ l ≥ q(q+1)/2 $ with $ q = (n-1)/2 - m $, we can write
\begin{equation*}
X^{l + (n+1)/2 + m(n+1)} Y^l = (X^{n+1})^{s/(n+1)} · (XY)^{l-q(q+1)/2} · X^{(n-q)(n-q+1)/2} Y^{q(q+1)/2}
\end{equation*}
To see that the monomial is indeed a $ ℂ[X, Y]^{\GL} $-linear multiple of $ f_q = X^{(n-q)(n-q+1)/2} Y^{q(q+1)/2} $, note again that
\begin{equation*}
s ≔ [(n-q)(n-q+1)/2 + l - q(q+1)/2] - [l + (n+1)/2 + m(n+1)] ≥ 0,
\end{equation*}
and that $ s $ is necessarily divisible by $ n+1 $, since the difference of exponents between $ X $ and $ Y $ on the left-hand side and on the remainder of the right-hand side is divisible by $ n+1 $. Analogous considerations hold for $ X^l Y^{l + (n+1)/2 + m(n+1)} $.

We have now shown that the terms of $ π(f) $ with $ m ≥ (n-1)/2 $ or $ l ≥ q(q+1)/2 $ with $ q ≔ (n-1)/2 - m $ already lie in the right-hand side of \eqref{eq:caseA-SI-odd-module}. By \autoref{th:caseA-SI-pbSI}, we can lift these terms to a $ θ_2 $-semiinvariant on $ \CG_Γ × ℂ^2 $. Subtracting this $ θ_2 $-semiinvariant from $ f $, the remainder is still a $ θ_2 $-semiinvariant. We can thus assume henceforth that $ α_{l, m} = β_{l, m} = 0 $ whenever $ m ≥ (n-1)/2 $ or $ l ≥ q(q+1)/2 $ with $ q ≔ (n-1)/2 - m $. This finishes the first part of the proof.

We now proceed to the second part of the proof. We show that for arbitrary but fixed $ l_0, m_0 ≥ 0 $ assuming $ q ≔ (n-1)/2 - m_0 ≥ 0 $ and $ l_0 < q (q+1)/2 $ we have $ α_{l_0, m_0} = 0 $. We start by defining the one-parameter subgroup $ g(t) = (t^{α_1}, …, t^{α_n}) ⊂ \GL $ given by the following exponents:
\begin{align*}
α_j &= -jqp, \quad \text{if } j ≤ n-q, \\
α_j &= -(n+1-j)(n-q)(p+1), \quad \text{if } n-q+1 ≤ j ≤ n.
\end{align*}
Here $ p \gg 0 $ is a natural number. The exponents are depicted graphically in \autoref{fig:caseA-SI-exponents}. It is an elementary check that for $ p $ large enough, we have $ α_{i+j} ≥ α_i + α_j $ for any $ i, j $ and therefore the limit $ \lim_{t → 0} g(t) φ^{(0)} $ exists. Put $ x = (1, 1) ∈ ℂ^2 $. For brevity let us convene that we write $ k = l + (n+1)/2 + m (n+1) $ in the below sum which runs over $ l, m ≥ 0 $. Since $ f $ is assumed to be a $ θ_2 $-semiinvariant, we obtain

\begin{align*}
f(g(t) φ^{(0)}, x) &= θ(g(t)) f(φ^{(0)}, g(t)^{-1} x) \\
&= \sum_{l, m ≥ 0} α_{l, m} t^{\sum_{j = 0}^n α_j + qpk + (n-q)(p+1)l} + \sum_{l, m ≥ 0} β_{l, m} t^{\sum_{j = 0}^n α_j + qpl + (n-q)(p+1)k}.
\end{align*}
We now prove two claims about this sum. The first claim entails that the exponent $ \sum_{j = 0}^n α_j + qp k_0 + (n-q)(p+1) l_0 $ is negative, where $ k_0 = l_0 + (n+1)/2 + m_0 (n+1) $. The second claim entails that for any pair $ (l, m) ≠ (l_0, m_0) $ this specific exponent is never attained again, that is,
\begin{align}
\sum_{j = 0}^n α_j + qpk + (n-q)(p+1)l &≠ \sum_{j = 0}^n α_j + qp k_0 + (n-q)(p+1) l_0,
\label{eq:caseA-SI-odd-exponentunique1} \\
\sum_{j = 0}^n α_j + qpl + (n-q)(p+1)k &≠ \sum_{j = 0}^n α_j + qp k_0 + (n-q)(p+1) l_0
\label{eq:caseA-SI-odd-exponentunique2}
\end{align}
For the first claim, we calculate the sum of exponents as
\begin{equation}
\label{eq:caseA-SI-odd-alphasum}
\sum_{j = 0}^n α_j = - \frac{(n-q)(n-q+1)qp}{2} - \frac{q(q+1)(n-q)(p+1)}{2}.
\end{equation}
Aiming to compare the first summand on the right-hand side of \eqref{eq:caseA-SI-odd-alphasum} with $ qp k_0 $, we observe
\begin{align*}
k_0 - (n-q)(n-q+1)/2 &= l_0 + (n+1)/2 + m_0 (n+1) - (n-q)(n-q+1)/2 \\
&< q(q+1)/2 + (n+1)/2 + m_0 (n+1) - q(q+1)/2 + (2q-1)n/2 - n^2 /2 \\
&= m_0 - \frac{n-1}{2} = -q ≤ 0.
\end{align*}
Aiming to compare the second summand on the right-hand side of \eqref{eq:caseA-SI-odd-alphasum} with $ (n-q)(p+1) l_0 $, we observe that by assumption we have
\begin{align*}
(n-q) l_0 - \frac{q(q+1)(n-q)}{2} &< 0.
\end{align*}
Finally, we conclude that the exponent $ \sum_{j = 0}^n α_j + qp k_0 + (n-q)(p+1) l_0 $ is negative, proving the first claim.

For the second claim, we shall check the two inequalities \eqref{eq:caseA-SI-odd-exponentunique1} and \eqref{eq:caseA-SI-odd-exponentunique2} separately. Towards \eqref{eq:caseA-SI-odd-exponentunique1}, assume that there is any pair $ (l, m) $ such that $ qpk + (n-q)(p+1) l = qp k_0 + (n-q)(p+1) l_0 $. Then we conclude that $ p $ divides $ (n-q)(p+1) (l_0 - l) $ and $ p+1 $ divides $ qp (k_0 - k) $. Since $ p $ and $ p+1 $ are coprime, we conclude that $ p $ divides $ (n-q)(l_0 - l) $ and $ p+1 $ divides $ q (k_0 - k) $. Note that the absolute values of the two numbers $ (n-q)(l_0 - l) $ and $ q(k_0 - k) $ are smaller than $ n^3 $. Therefore, if we pick $ p $ to be larger than $ n^3 $, we conclude $ l = l_0 $ and $ k = k_0 $, in particular $ m = m_0 $. This proves the first inequality.

Towards the second inequality \eqref{eq:caseA-SI-odd-exponentunique2}, assume that there is any pair $ (l, m) $ such that $ qpl + (n-q)(p+1)k = qp k_0 + (n-q)(p+1) l_0 $. We similarly conclude $ l = k_0 $ and $ k = l_0 $, which however implies $ k = l_0 ≤ k_0 = l $, a contradiction. This proves the second inequality.

Finally, we have proved both intermediate claims. We conclude that the coefficient $ α_{l_0, m_0} $ vanishes since otherwise $ f(g(t) φ^{(0)}, x) $ diverges. This finishes the second part of the proof.

Ultimately, let us draw the conclusions and finish the proof. In a fashion analogous to the second part of the proof, one shows that if $ q ≔ (n-1)/2 - m_0 ≥ 0 $ and $ l_0 < q (q+1)/2 $ then $ β_{l_0, m_0} = 0 $. Recalling from the first part of the proof that we have already subtracted certain terms from $ π(f) $ and therefore assumed that $ α_{l, m} = β_{l, m} = 0 $ whenever $ m ≥ (n-1)/2 $ or $ l ≥ q(q+1)/2 $ with $ q ≔ (n-1)/2 - m $, we conclude that $ π(f) $ vanishes. This finishes the proof.
\end{proof}

\begin{remark}
\label{th:caseA-SI-even-rem}
We have shown in \autoref{th:caseA-SI-odd-pbrestr} that for odd $ n $ all $ θ_2 $-semiinvariants on $ \CG_Γ × ℂ^2 $ are pullbacks of $ θ_2 $-semiinvariants on $ \Rep(Π_Q, α) $. The same statement holds for even $ n $ with a similar proof.
\end{remark}

\section{The $ D_n $ case}
\label{sec:caseD}
In this section, we treat the $ D_n $ case in detail. We construct the variety $ \CG_Γ $, the stability parameters $ θ_1 $, $ θ_2 $ and the map $ R: \CG_Γ × ℂ^2 → \Rep(Π_Q, α) $, where $ (Q, α) $ is the Kleinian $ D_n $ quiver. We prove the remaining technical parts of \autoref{stageplan}.

\subsection{Stage 1: Construction of the Clebsch-Gordan variety}
\label{sec:caseD-step1}
The Kleinian group is the binary dihedral group of order $ 4(n-2) $:
\begin{align*}
Γ &= \BD_{n-2} = ⟨a, x \running a^{2(n-2)} = 1, x^2 = a^{n-2}, x^{-1} a x = a^{-1}⟩.
\end{align*}
When $ n $ is even, the group $ Γ $ has four 1-dimensional representations $ E_1, …, E_4 $ and two families $ O_1, O_3, …, O_{n-3} $ and $ I_2, I_4, …, I_{n-4} $ of 2-dimensional representations. When $ n $ is odd, the group $ Γ $ has four 1-dimensional representations $ E_1, …, E_4 $ and two families $ O_1, O_3, …, O_{n-4} $ and $ I_2, …, I_{n-3} $ of 2-dimensional representations. The $ Γ $-representation $ ℂ^2 $ itself is isomorphic to $ O_1 $. The representations can be written explicitly in the following matrix form:
\begin{center}
\begin{tabular}{lcc}
Representation & Action of $ a $ & Action of $ x $ \\\hline
$ E_1 $ & 1 & $ 1 $ \\
$ E_2 $ & 1 & $ -1 $ \\
$ E_3 $ & $ -1 $ & $ 1 $ \\
$ E_4 $ & $ -1 $ & $ -1 $ \\
$ O_k $ & $ \pmat{e^{πik/(n-2)} & 0 \\ 0 & e^{-πik/(n-2)}} $ & $ \pmat{0 & -1 \\ 1 & 0} $ \\
$ I_k $ & $ \pmat{e^{πik/(n-2)} & 0 \\ 0 & e^{-πik/(n-2)}} $ & $ \pmat{0 & 1 \\ 1 & 0} $
\end{tabular}
\end{center}

\begin{remark}
For notational reasons, it is occasionally difficult to deal with all $ D_n $ at the same time. We may sometimes make the assumption that $ n $ is an even integer in order to write down the most accurate statements. Moreover, accurate treatment of the Kleinian $ D_4 $ case requires attention to the Clebsch-Gordan coefficients which deviate at first sight from the other $ D_n $ cases. For instance, for $ n ≥ 5 $ the representation $ O_1 ¤ O_1 $ decomposes into $ E_1 ⊕ E_2 ⊕ O_2 $, while for $ n = 4 $ it decomposes into $ E_1 ⊕ E_2 ⊕ E_3 ⊕ E_4 $. Nevertheless, a shared treatment of all $ D_n $ cases is possible as long as attention to this detail is paid and the reader reinterprets the statements in a specific way for the $ n = 4 $ case. In fact, the representation $ O_1 $ in the $ n = 4 $ case has more similarity with the representation $ I_{(n-4)/2} $ in the case of even $ n ≥ 6 $. Therefore, the most accurate statements for the $ n = 4 $ case are obtained when $ E_i $ is interpreted as $ E_i $, but $ O_1 $ is interpreted as $ I_{(n-4)/2} $.
\end{remark}

The gauge group consists of three general linear groups of rank one and $ (n-3) $-many general linear groups of rank two:
\begin{equation*}
\GL = \underset{E_1}{\GL_1 (ℂ)} × \underset{E_2}{\GL_1 (ℂ)} × \underset{E_3}{\GL_1 (ℂ)} × \underset{O_1}{\GL_1 (ℂ)} × … × \underset{O_{n-3}}{\GL_1 (ℂ)}, \\
\end{equation*}
%
%
An Clebsch-Gordan datum $ φ $ for $ Γ $ consists of a long list of bilinear maps. For us, the range of relevant entries of $ φ $ is rather limited. In case $ n $ is even, the most important are the following:
\begin{align*}
φ_{E_i, E_j}: E_i ¤ E_j &→ E_{σ(ij)}, \quad i, j ∈ \{2, 3, 4\}, \\
φ_{E_1, O_1}: E_1 ¤ O_1 &→ O_1, \\
φ_{E_2, O_1}: E_2 ¤ O_1 &→ O_1, \\
φ_{E_3, O_1}: E_3 ¤ O_1 &→ O_{n/2 - 1}, \\
φ_{E_4, O_1}: E_4 ¤ O_1 &→ O_{n/2 - 1}, \\
φ_{I_k, O_1}: I_k ¤ O_1 &→ O_{k-1} ⊕ O_{k+1}, \quad k = 2, 4, …, n-4, \\
φ_{O_k, O_1}: O_i ¤ O_1 &→ I_{k-1} ⊕ I_{k+1}, \quad k = 3, 5, …, n-5, \\
φ_{O_1, O_1}: O_1 ¤ O_1 &→ E_1 ⊕ E_2 ⊕ O_1, \\
φ_{O_{n-3}, O_1}: O_{n-3} ¤ O_1 &→ I_{n-4} ⊕ E_3 ⊕ E_4.
\end{align*}
Here $ σ(ij) $ stands for the missing index. If $ i ≠ j $, then it is defined by $ \{i, j, σ(ij)\} = \{2, 3, 4\} $. If $ i = j $, then $ σ(ij) = 1 $. In case $ n $ is odd, the range of most important entries is analogous. In case $ n ≥ 5 $, the specific Clebsch-Gordan datum $ φ^{(0)} $ consists of the following choices:
\begin{alignat*}{2}
φ^{(0)}_{E_i, E_j} &= 1, && \\
φ^{(0)}_{E_1, O_1} &= \Id, & 
φ^{(0)}_{E_2, O_1} &= \pmat{1 & 0 \\ 0 & -1}, \\
φ^{(0)}_{E_3, O_1} &= \pmat{0 & 1 \\ -1 & 0}, &
φ^{(0)}_{E_4, O_1} &= \pmat{0 & 1 \\ 1 & 0}, \\
φ^{(0)}_{I_k, O_1} &= \pmat{0 & 1 & 0 & 0 \\ 0 & 0 & -1 & 0 \\ 2 & 0 & 0 & 0 \\ 0 & 0 & 0 & 2}, &
φ^{(0)}_{O_k, O_1} &= \pmat{0 & 1 & 0 & 0 \\ 0 & 0 & -1 & 0 \\ 2 & 0 & 0 & 0 \\ 0 & 0 & 0 & 2}, \\
φ^{(0)}_{O_1, O_1} &= \pmat{0 & 1 & -1 & 0 \\ 0 & 1 & 1 & 0 \\ 2 & 0 & 0 & 0 \\ 0 & 0 & 0 & 2}, & \qquad
φ^{(0)}_{O_{n-3}, O_1} &= \pmat{0 & 1 & 0 & 0 \\ 0 & 0 & -1 & 0 \\ 1 & 0 & 0 & 1 \\ 1 & 0 & 0 & -1}.
\end{alignat*}
In case $ n = 4 $, we choose the entries $ φ^{(0)}_{E_i, E_j} $ and $ φ^{(0)}_{E_i, O_1} $ as above and
\begin{equation*}
φ^{(0)}_{O_1, O_1} = \pmat{0 & 1 & -1 & 0 \\ 0 & 1 & 1 & 0 \\ 1 & 0 & 0 & 1 \\ 1 & 0 & 0 & -1}.
\end{equation*}
%
%
%
%
%
We define the varieties $ \CG_Γ $ and the stability parameters $ θ_1 $ and $ θ_2 $ as follows:
\begin{equation*}
\CG_Γ = \closure{\GL φ^{(0)}}, \quad θ_1 = (\underset{E_2}{-1}, \underset{E_3}{-1}, \underset{E_4}{-1}, \underset{O_1}{-2}, …, \underset{O_{n-3}}{-2}), \quad θ_1 = (+1, …, +1).
\end{equation*}

\subsection{Stage 2: Construction of the map $ R $}
\label{sec:caseD-2}
We define the map $ R(φ, x) $ following our general recipe $ R(φ, x) = φ(- ¤ x) $:
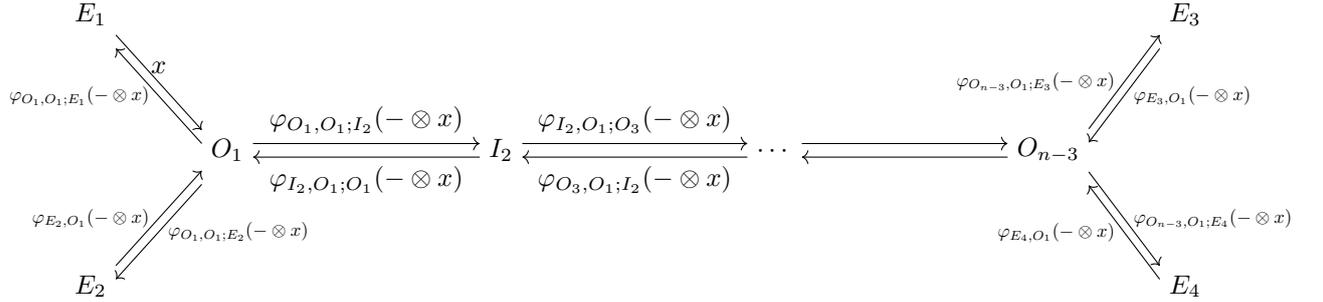
\begin{center}
\begin{tikzpicture}[scale=0.9]
\path (-2, 2) node (E1) {$ E_1 $};
\path (-2, -2) node (E2) {$ E_2 $};
\path (0, 0) node (O1) {$ O_1 $};
\path (4, 0) node (I2) {$ I_2 $};
\path (8, 0) node (dots) {$ \cdots \vphantom{A_2} $};
\path (12, 0) node (On-3) {$ O_{n-3} $};
\path (14, 2) node (E3) {$ E_3 $};
\path (14, -2) node (E4) {$ E_4 $};
\path[draw, ->] (E1.south east) to node[midway, above] {$ x $} (O1.north west);
\path[draw, ->] ($ (O1.north west) + (down:0.2) $) to node[midway, left] {\tinymath{$ φ_{O_1, O_1; E_1} (- ¤ x) $}} ($ (E1.south east) + (down:0.2) $);
\path[draw, ->] (E2.north east) to node[midway, left] {\tinymath{$ φ_{E_2, O_1} (- ¤ x) $}} (O1.south west);
\path[draw, ->] ($ (O1.south west) + (down:0.2) $) to node[midway, right] {\tinymath{$ φ_{O_1, O_1; E_2} (- ¤ x) $}} ($ (E2.north east) + (down:0.2) $);
\path[draw, ->] ($ (O1.east) + (up:0.1) $) to node[midway, above] {$ φ_{O_1, O_1; I_2} (- ¤ x) $} ($ (I2.west) + (up:0.1) $);
\path[draw, ->] ($ (I2.west) + (down:0.1) $) to node[midway, below] {$ φ_{I_2, O_1; O_1} (- ¤ x) $} ($ (O1.east) + (down:0.1) $);
\path[draw, ->] ($ (I2.east) + (up:0.1) $) to node[midway, above] {$ φ_{I_2, O_1; O_3} (- ¤ x) $} ($ (dots.west) + (up:0.1) $);
\path[draw, ->] ($ (dots.west) + (down:0.1) $) to node[midway, below] {$ φ_{O_3, O_1; I_2} (- ¤ x) $} ($ (I2.east) + (down:0.1) $);
\path[draw, ->] ($ (dots.east) + (up:0.1) $) to ($ (On-3.west) + (up:0.1) $);
\path[draw, ->] ($ (On-3.west) + (down:0.1) $) to ($ (dots.east) + (down:0.1) $);
\path[draw, ->] ($ (On-3.north east) $) to node[midway, left] {\tinymath{$ φ_{O_{n-3}, O_1; E_3} (- ¤ x) $}} (E3.south west);
\path[draw, ->] ($ (E3.south west) + (down:0.2) $) to node[midway, right] {\tinymath{$ φ_{E_3, O_1} (- ¤ x) $}} ($ (On-3.north east) + (down:0.2) $);
\path[draw, ->] (On-3.south east) to node[midway, right] {\tinymath{$ φ_{O_{n-3}, O_1; E_4} (- ¤ x) $}} (E4.north west);
\path[draw, ->] ($ (E4.north west) + (down:0.2) $) to node[midway, left] {\tinymath{$ φ_{E_4, O_1} (- ¤ x) $}} ($ (On-3.south east) + (down:0.2) $);
\end{tikzpicture}
\end{center}
In particular, the specific representation $ R(φ^{(0)}, x) $ takes the following shape:
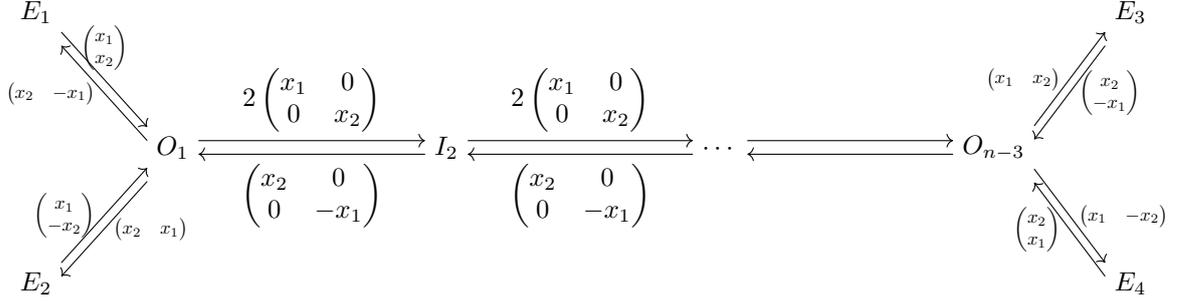
\begin{center}
\begin{tikzpicture}[scale=0.9]
\path (-2, 2) node (E1) {$ E_1 $};
\path (-2, -2) node (E2) {$ E_2 $};
\path (0, 0) node (O1) {$ O_1 $};
\path (4, 0) node (I2) {$ I_2 $};
\path (8, 0) node (dots) {$ \cdots \vphantom{A_2} $};
\path (12, 0) node (On-3) {$ O_{n-3} $};
\path (14, 2) node (E3) {$ E_3 $};
\path (14, -2) node (E4) {$ E_4 $};
\path[draw, ->] (E1.south east) to node[midway, above] {\tinymath{$ \pmat{x_1 \\ x_2} $}} (O1.north west);
\path[draw, ->] ($ (O1.north west) + (down:0.2) $) to node[midway, left] {\tinymath{$ \pmat{x_2 & -x_1} $}} ($ (E1.south east) + (down:0.2) $);
\path[draw, ->] (E2.north east) to node[midway, left] {\tinymath{$ \pmat{x_1 \\ -x_2} $}} (O1.south west);
\path[draw, ->] ($ (O1.south west) + (down:0.2) $) to node[midway, right] {\tinymath{$ \pmat{x_2 & x_1} $}} ($ (E2.north east) + (down:0.2) $);
\path[draw, ->] ($ (O1.east) + (up:0.1) $) to node[midway, above] {$ 2 \pmat{x_1 & 0 \\ 0 & x_2} $} ($ (I2.west) + (up:0.1) $);
\path[draw, ->] ($ (I2.west) + (down:0.1) $) to node[midway, below] {$ \pmat{x_2 & 0 \\ 0 & -x_1} $} ($ (O1.east) + (down:0.1) $);
\path[draw, ->] ($ (I2.east) + (up:0.1) $) to node[midway, above] {$ 2 \pmat{x_1 & 0 \\ 0 & x_2} $} ($ (dots.west) + (up:0.1) $);
\path[draw, ->] ($ (dots.west) + (down:0.1) $) to node[midway, below] {$ \pmat{x_2 & 0 \\ 0 & -x_1} $} ($ (I2.east) + (down:0.1) $);
\path[draw, ->] ($ (dots.east) + (up:0.1) $) to ($ (On-3.west) + (up:0.1) $);
\path[draw, ->] ($ (On-3.west) + (down:0.1) $) to ($ (dots.east) + (down:0.1) $);
\path[draw, ->] ($ (On-3.north east) $) to node[midway, left] {\tinymath{$ \pmat{x_1 & x_2} $}} (E3.south west);
\path[draw, ->] ($ (E3.south west) + (down:0.2) $) to node[midway, right] {\tinymath{$ \pmat{x_2 \\ -x_1} $}} ($ (On-3.north east) + (down:0.2) $);
\path[draw, ->] (On-3.south east) to node[midway, right] {\tinymath{$ \pmat{x_1 & -x_2} $}} (E4.north west);
\path[draw, ->] ($ (E4.north west) + (down:0.2) $) to node[midway, left] {\tinymath{$ \pmat{x_2 \\ x_1} $}} ($ (On-3.south east) + (down:0.2) $);
\end{tikzpicture}
\end{center}

\begin{lemma}
\label{th:caseD-2-preproj}
For any $ (φ, x) ∈ \CG_Γ × ℂ^2 $, the representation $ R(φ, x) $ satisfies the preprojective conditions.
\end{lemma}

\begin{proof}
We start with the observation that $ R(φ, x) $ is in every case a quiver representation and $ R $ defines a $ \GL $-equivariant map $ R: \CG_Γ × ℂ^2 → \Rep(\Qbar, α) $ to the representations of the Kleinian double quiver. Next, we observe that the representation $ R(φ^{(0)}, x) $, which is depicted above, satisfies the preprojective relations. This implies that $ R(gφ^{(0)}, g(g^{-1} x)) = g R(φ^{(0)}, g^{-1} x) $ also satisfies the preprojective relations for any $ g ∈ \GL $ and $ x ∈ ℂ^2 $. By a standard limit argument, the preprojective conditions then also hold for any $ (φ, x) ∈ \CG_Γ × ℂ^2 $. We conclude that $ R $ becomes a $ \GL $-equivariant map $ R: \CG_Γ × ℂ^2 → \Rep(Π_Q, α) $. This finishes the proof.
\end{proof}

\subsection{Stage 3: Verification that $ R^* $ is an isomorphism of invariants}
\label{sec:caseD-3}
It is our task to prove the following lemma.

\begin{lemma}
\label{th:caseD-Rstarinv}
The map $ π ∘ R^*: ℂ[\Rep(Π_Q, α)]^{\GL} → ℂ[\CG_Γ × ℂ^2]^{\GL} → ℂ[X, Y]^{\Stab_{\GL} (φ^{(0)})} $ is an isomorphism.
\end{lemma}

\begin{proof}
The trick is to regard the three generating elements of the domain and codomain. In the codomain we pick the three elements
\begin{equation*}
A' = X^2 Y^2, \quad B' = X^{2(n-2)} + Y^{2(n-2)}, \quad C' = X^{2(n-2) + 1} Y - X Y^{2(n-2)+1}.
\end{equation*}
In the domain we pick the three elements 
\begin{align*}
A &= 2^{-2} \trace(A_1^* A_2 A_2^* A_1), \\
B &= \trace(A_1^* A_3^* … A_{n-1}^* A_{n-1} …  A_3 A_1) - 2^{n-3} A^{(n-2)/2}, \\
C &= - 2^{3-n} \trace(A_1 A_1^* A_2 A_2^*  A_3^* … A_{n-1}^* A_{n-1} … A_1).
\end{align*}
Evidently, we have $ π(R^* (A)) = A' $ and $ π(R^* (B)) = B' $ and $ π(R^* (C)) = C' $. In both rings, the only relation satisfied by the three generators is the Kleinian $ D_n $ relation. This finishes the proof.
\end{proof}

\subsection{Stage 4: Identification of the stabilizer of $ φ^{(0)} $}
\label{sec:caseD-4}
We shall examine the stabilizer of $ φ^{(0)} $ under the $ \GL $-action. The stabilizer group of $ φ^{(0)} $ under the $ \GL $-action is generated by two elements $ σ_a $ and $ σ_x $. Mapping these group elements to $ a, x ∈ Γ = \BD_{n-2} $ verifies that the stabilizer group is isomorphic to $ Γ = \BD_{n-2} $.
\begin{align*}
& \Stab_{\GL} (φ^{(0)}) = ⟨σ_a, σ_x⟩ ⊂ \GL ,\\
& σ_a = \left(\underset{E_2}{+1}, \underset{E_3}{-1}, \underset{E_4}{-1}, \underset{O_j}{\pmat{e^{\frac{πi}{n-2} · j} & 0 \\ 0 & e^{-\frac{πi}{n-2} · j}}}, \underset{I_j}{\pmat{e^{\frac{πi}{n-2} · j} & 0 \\ 0 & e^{-\frac{πi}{n-2} · j}}}\right), \\
& σ_x = \left(\underset{E_2}{-1}, \underset{E_3}{+1}, \underset{E_4}{-1}, \underset{O_j}{\pmat{0 & 1 \\ -1 & 0}}, \underset{I_j}{\pmat{0 & 1 \\ 1 & 0}}\right).
\end{align*}
The elements $ σ_a $ and $ σ_x $ act on $ ℂ^2 = O_1 $ simply by left-multiplication with their $ O_1 $-entries. We immediately observe that this action is isomorphic to the action of $ Γ $ on $ ℂ^2 $. This finishes stage four.

\subsection{Stage 5: Identification of the $ θ_1 $-semistable locus}
\label{sec:caseD-5}
It is our task to prove the following lemma. Recall that $ f_0: \CG_Γ → ℂ $ denotes the $ |Γ|θ_1 $-semiinvariant given by the weighted product over $ i, j $ of the determinant of $ φ_{i, j} $.

\begin{lemma}
\label{th:caseD-cgtheta1}
Let $ φ ∈ \CG_Γ^{θ_1} $. If $ f_0 (φ) ≠ 0 $, then $ φ ∈ \GL φ^{(0)} $.
\end{lemma} 

\begin{proof}
Since $ φ ∈ \CG_Γ = \overline{\GL φ^{(0)}} $, we can write $ φ = \lim_{k → ∞} g_k φ^{(0)} $ for some sequence $ (g_k) ⊂ \GL $. We shall now prove that the determinants of the individual components of $ g_k $ converge, up to choice of a subsequence of $ (g_k) $. First of all, note that $ (g_k φ^{(0)})_{E_i} = φ^{(0)}_{ii} (g_k)_{E_i}^{-2} $. Meanwhile, $ (g_k φ^{(0)})_{E_i} $ converges to $ φ_{E_i} $, which is a nonzero number. After passing to a subsequence of $ (g_k) $, we can thus assume that $ (g_k)_{E_i} $ converges, with nonzero limit value. Next, we observe
\begin{align*}
\det (g_k φ^{(0)})_{O_1, O_1} &= (g_k)_{E_2} \det(g_k)_{O_2} \det(g_k)_{O_1, O_1}^{-2}, \\
\det (g_k φ^{(0)})_{O_i, O_1} &= \det(g_k)_{O_{i-1}} \det(g_k)_{O_{i+1}} \det(g_k)_{O_i}^{-2} \det(g_k)_{O_1}^{-2}, \quad 2 ≤ i ≤ n-4, \\
\det (g_k φ^{(0)})_{O_{n-3}, O_1} &= (g_k)_{E_3} (g_k)_{E_4} \det(g_k)_{O_{n-4}} \det(g_k φ^{(0)})_{O_1, O_1}^{-2}.
\end{align*}
By assumption, all left-hand sides converge to a nonzero value. Multiplying up the equations in the appropriate way, we conclude that the ratio $ \det(g_k)_{O_i} \det(g_k)_{O_1}^{-i^2} $ converges to a nonzero value for $ 1 ≤ i ≤ n-3 $. Combining with the third equation, we conclude that $ \det(g_k)_{O_1}^{(n-2)^2} $ converges to a nonzero value. After passing to a subsequence, we conclude that $ \det(g_k)_{O_1} $ converges and thanks to the convergence of the ratios also $ \det(g_k)_{O_i} $ converges for all $ 2 ≤ i ≤ n-3 $. This proves that the determinants of the individual components of $ (g_k) $ converge.

We shall now prove that $ (g_k) $ converges. We regard the map $ Φ_k $ given as the composition
\begin{align*}
\underbrace{(…(((O_1 ¤ O_1) ¤ O_1) ¤ O_1) ¤  …) ¤ O_1}_{(n-2) × O_1} &→  (…((O_2 ¤ O_1) ¤ O_1) ¤ …) ¤ O_1 \\
&→  (…(O_3 ¤ O_1) ¤ …) ¤ O_1 \\
&→ … \\
&→ E_3 ⊕ E_4
\end{align*}
In this composition, the arrows are given by $ π_{O_{i+1}} ∘ (g_k φ^{(0)})_{O_i, O_1} $ for $ 1 ≤ i ≤ n-4 $ and $ π_{E_3 ⊕ E_4} ∘ (g_k φ^{(0)})_{O_{n-3}, O_1} $. The map $ Φ_k $ is a map from the large space $ O_1^{¤ (n-2)} $ to the two-dimensional space $ E_3 ⊕ E_4 $, in other words a very wide matrix. By assumption, we know that $ (Φ_k) $ converges.

In what follows, we shall determine four entries of $ (Φ_k) $ explicitly and draw the conclusion that $ (g_k)_{O_1} $ converges. As a simplification, we calculate these specific entries by pretending that the inverse matrix of $ (g_k)_{O_i} $ agrees with its adjugate, pretending that $ \det(g_k)_{O_i} = 1 $. Since we have already proved that the determinants converge to nonzero values, this simplification is legitimate. We write
\begin{equation*}
(g_k)_{O_1} = \pmat{A_k & B_k \\ C_k & D_k}.
\end{equation*}
Let us now state our calculations as follows:
\begin{align*}
π_{E_3} (Φ_k (e_1^{¤ (n-2)})) &= -C_k^{n-2} + D_k^{n-2}, \\
π_{E_3} (Φ_k (e_2^{¤ (n-2)})) &= A_k^{n-2} - B_k^{n-2}, \\
π_{E_4} (Φ_k (e_1^{¤ (n-2)})) &= C_k^{n-2} + D_k^{n-2}, \\
π_{E_4} (Φ_k (e_2^{¤ (n-2)})) &= - A_k^{n-2} - B_k^{n-2}.
\end{align*}
Since the left-hand side converges as $ k → ∞ $, we conclude that $ (g_k)_{O_1} $ converges. We note that the limit is an invertible matrix. Finally, regard the maps
\begin{equation*}
π_{O_{i+1}} ∘ (g_k φ^{(0)})_{O_i, O_1} = (g_k)_{O_{i+1}} ∘ π_{O_{i+1}} ∘ φ^{(0)}_{O_i, O_1} ∘ ((g_k)_{O_i}^{-1} ¤ (g_k)_{O_1}^{-1}).
\end{equation*}
Since the left-hand side converges and $ (g_k)_{O_1} $ converges to an invertible matrix and $ π_{O_{i+1}} ∘ φ^{(0)}_{O_i, O_1} $ is surjective, we conclude inductively that $ (g_k)_{O_i} $ converges to an invertible matrix for all $ 2 ≤ i ≤ n-3 $. This proves that the sequence $ (g_k) ⊂ \GL $ converges. Consequently, we have $ φ = \lim_{k → ∞} g_k φ^{(0)} = (\lim_{k → ∞} g_k) φ^{(0)} ∈ \GL φ^{(0)} $. This shows that $ φ ∈ \GL φ^{(0)} $, and finishes the proof.
\end{proof}

\subsection{Stage 6: Identification of the $ θ_2 $-semistable locus}
\label{sec:caseD-6}
In this section, we prove the remaining parts of the sixth stage. One part is the claim that if $ (φ, x) ∈ \CG_Γ × ℂ^2 $ is $ θ_2 $-semistable, then $ R(φ, x) ∈ \Rep(Π_Q, α) $ is $ θ_2 $-semistable. The other part is the claim that the map $ R: (\CG_Γ × ℂ^2)^{θ_2} → \Rep(Π_Q, α)^{θ_2} $ is surjective.

We start by showing that if $ (φ, x) ∈ \CG_Γ × ℂ^2 $ is $ θ_2 $-semistable, then $ R(φ, x) ∈ \Rep(Π_Q, α) $ is a $ θ_2 $-semistable representation. The strategy is to prove the contraposite. We shall provide a list of non-$ θ_2 $-semistable representations in \autoref{fig:stability-cases} and then prove that none of them lie in the image of $ (\CG_Γ × ℂ^2)^{θ_2} $ under $ R $.

\begin{figure}
\centering
\begin{subfigure}{0.3\linewidth}
\begin{tikzpicture}[scale=1]
\path (0, 0) node (C) {$ E_1 $};
\path (2, -2) node (U2) {$ E_2 $};
\path (4, 0) node (U3) {$ E_3 $};
\path (2, 2) node (U4) {$ E_4 $};
\path (2, 0) node (U5) {$ O_1 $};
\path[draw, ->] ($ (C.east) + (up:0.1) $) to node[midway, above] {$ 0 $} ($ (U5.west) + (up:0.1) $);
\path[draw, <-] ($ (C.east) + (down:0.1) $) to node[midway, below] {$ * $} ($ (U5.west) + (down:0.1) $);
\path[draw, ->] ($ (U2.north) + (left:0.1) $) to node[midway, left] {$ * $} ($ (U5.south) + (left:0.1) $);
\path[draw, <-] ($ (U2.north) + (right:0.1) $) to node[midway, right] {$ * $} ($ (U5.south) + (right:0.1) $);
\path[draw, ->] ($ (U3.west) + (down:0.1) $) to node[midway, below] {$ * $} ($ (U5.east) + (down:0.1) $);
\path[draw, <-] ($ (U3.west) + (up:0.1) $) to node[midway, above] {$ * $} ($ (U5.east) + (up:0.1) $);
\path[draw, ->] ($ (U4.south) + (right:0.1) $) to node[midway, right] {$ * $} ($ (U5.north) + (right:0.1) $);
\path[draw, <-] ($ (U4.south) + (left:0.1) $) to node[midway, left] {$ * $} ($ (U5.north) + (left:0.1) $);
\end{tikzpicture}
\caption{}
\label{fig:stability-case-0}
\end{subfigure}
\begin{subfigure}{0.3\linewidth}
\begin{tikzpicture}[scale=1]
\path (0, 0) node (C) {$ E_1 $};
\path (2, -2) node (U2) {$ E_2 $};
\path (4, 0) node (U3) {$ E_3 $};
\path (2, 2) node (U4) {$ E_4 $};
\path (2, 0) node (U5) {$ O_1 $};
\path[draw, ->] ($ (C.east) + (up:0.1) $) to node[midway, above] {$ e_1 $} ($ (U5.west) + (up:0.1) $);
\path[draw, <-] ($ (C.east) + (down:0.1) $) to node[midway, below] {$ π_2 $} ($ (U5.west) + (down:0.1) $);
\path[draw, ->] ($ (U2.north) + (left:0.1) $) to node[midway, left] {$ e_1 $} ($ (U5.south) + (left:0.1) $);
\path[draw, <-] ($ (U2.north) + (right:0.1) $) to node[midway, right] {$ π_2 $} ($ (U5.south) + (right:0.1) $);
\path[draw, ->] ($ (U3.west) + (down:0.1) $) to node[midway, below] {$ e_1 $} ($ (U5.east) + (down:0.1) $);
\path[draw, <-] ($ (U3.west) + (up:0.1) $) to node[midway, above] {$ π_2 $} ($ (U5.east) + (up:0.1) $);
\path[draw, ->] ($ (U4.south) + (right:0.1) $) to node[midway, right] {$ e_1 $} ($ (U5.north) + (right:0.1) $);
\path[draw, <-] ($ (U4.south) + (left:0.1) $) to node[midway, left] {$ π_2 $} ($ (U5.north) + (left:0.1) $);
\end{tikzpicture}
\caption{}
\label{fig:stability-case-A}
\end{subfigure}
\begin{subfigure}{0.3\linewidth}
\begin{tikzpicture}[scale=1]
\path (0, 0) node (C) {$ E_1 $};
\path (2, -2) node (U2) {$ E_2 $};
\path (4, 0) node (U3) {$ E_3 $};
\path (2, 2) node (U4) {$ E_4 $};
\path (2, 0) node (U5) {$ O_1 $};
\path[draw, ->] ($ (C.east) + (up:0.1) $) to node[midway, above] {$ e_1 $} ($ (U5.west) + (up:0.1) $);
\path[draw, <-] ($ (C.east) + (down:0.1) $) to node[midway, below] {$ π_2 $} ($ (U5.west) + (down:0.1) $);
\path[draw, ->] ($ (U2.north) + (left:0.1) $) to node[midway, left] {$ e_1 $} ($ (U5.south) + (left:0.1) $);
\path[draw, <-] ($ (U2.north) + (right:0.1) $) to node[midway, right] {$ π_2 $} ($ (U5.south) + (right:0.1) $);
\path[draw, ->] ($ (U3.west) + (down:0.1) $) to node[midway, below] {$ e_1 $} ($ (U5.east) + (down:0.1) $);
\path[draw, <-] ($ (U3.west) + (up:0.1) $) to node[midway, above] {$ π_2 $} ($ (U5.east) + (up:0.1) $);
\path[draw, ->] ($ (U4.south) + (right:0.1) $) to node[midway, right] {$ e_1 $} ($ (U5.north) + (right:0.1) $);
\path[draw, <-] ($ (U4.south) + (left:0.1) $) to node[midway, left] {$ 0 $} ($ (U5.north) + (left:0.1) $);
\end{tikzpicture}
\caption{}
\label{fig:stability-case-B1}
\end{subfigure}
\begin{subfigure}{0.3\linewidth}
\begin{tikzpicture}[scale=1]
\path (0, 0) node (C) {$ E_1 $};
\path (2, -2) node (U2) {$ E_2 $};
\path (4, 0) node (U3) {$ E_3 $};
\path (2, 2) node (U4) {$ E_4 $};
\path (2, 0) node (U5) {$ O_1 $};
\path[draw, ->] ($ (C.east) + (up:0.1) $) to node[midway, above] {$ e_1 $} ($ (U5.west) + (up:0.1) $);
\path[draw, <-] ($ (C.east) + (down:0.1) $) to node[midway, below] {$ π_2 $} ($ (U5.west) + (down:0.1) $);
\path[draw, ->] ($ (U2.north) + (left:0.1) $) to node[midway, left] {$ e_1 $} ($ (U5.south) + (left:0.1) $);
\path[draw, <-] ($ (U2.north) + (right:0.1) $) to node[midway, right] {$ π_2 $} ($ (U5.south) + (right:0.1) $);
\path[draw, ->] ($ (U3.west) + (down:0.1) $) to node[midway, below] {$ e_1 $} ($ (U5.east) + (down:0.1) $);
\path[draw, <-] ($ (U3.west) + (up:0.1) $) to node[midway, above] {$ π_2 $} ($ (U5.east) + (up:0.1) $);
\path[draw, ->] ($ (U4.south) + (right:0.1) $) to node[midway, right] {$ e_2 $} ($ (U5.north) + (right:0.1) $);
\path[draw, <-] ($ (U4.south) + (left:0.1) $) to node[midway, left] {$ 0 $} ($ (U5.north) + (left:0.1) $);
\end{tikzpicture}
\caption{}
\label{fig:stability-case-B2}
\end{subfigure}
\begin{subfigure}{0.3\linewidth}
\begin{tikzpicture}[scale=1]
\path (0, 0) node (C) {$ E_1 $};
\path (2, -2) node (U2) {$ E_2 $};
\path (4, 0) node (U3) {$ E_3 $};
\path (2, 2) node (U4) {$ E_4 $};
\path (2, 0) node (U5) {$ O_1 $};
\path[draw, ->] ($ (C.east) + (up:0.1) $) to node[midway, above] {$ e_1 $} ($ (U5.west) + (up:0.1) $);
\path[draw, <-] ($ (C.east) + (down:0.1) $) to node[midway, below] {$ π_2 $} ($ (U5.west) + (down:0.1) $);
\path[draw, ->] ($ (U2.north) + (left:0.1) $) to node[midway, left] {$ e_1 $} ($ (U5.south) + (left:0.1) $);
\path[draw, <-] ($ (U2.north) + (right:0.1) $) to node[midway, right] {$ π_2 $} ($ (U5.south) + (right:0.1) $);
\path[draw, ->] ($ (U3.west) + (down:0.1) $) to node[midway, below] {$ e_1 $} ($ (U5.east) + (down:0.1) $);
\path[draw, <-] ($ (U3.west) + (up:0.1) $) to node[midway, above] {$ π_2 $} ($ (U5.east) + (up:0.1) $);
\path[draw, ->] ($ (U4.south) + (right:0.1) $) to node[midway, right] {$ 0 $} ($ (U5.north) + (right:0.1) $);
\path[draw, <-] ($ (U4.south) + (left:0.1) $) to node[midway, left] {$ * $} ($ (U5.north) + (left:0.1) $);
\end{tikzpicture}
\caption{}
\label{fig:stability-case-B3}
\end{subfigure}
%
%
\begin{subfigure}{0.3\linewidth}
\begin{tikzpicture}[scale=1]
\path (0, 0) node (C) {$ E_1 $};
\path (2, -2) node (U2) {$ E_2 $};
\path (4, 0) node (U3) {$ E_3 $};
\path (2, 2) node (U4) {$ E_4 $};
\path (2, 0) node (U5) {$ O_1 $};
\path[draw, ->] ($ (C.east) + (up:0.1) $) to node[midway, above] {$ e_1 $} ($ (U5.west) + (up:0.1) $);
\path[draw, <-] ($ (C.east) + (down:0.1) $) to node[midway, below] {$ 0 $} ($ (U5.west) + (down:0.1) $);
\path[draw, ->] ($ (U2.north) + (left:0.1) $) to node[midway, left] {$ e_1 $} ($ (U5.south) + (left:0.1) $);
\path[draw, <-] ($ (U2.north) + (right:0.1) $) to node[midway, right] {$ π_2 $} ($ (U5.south) + (right:0.1) $);
\path[draw, ->] ($ (U3.west) + (down:0.1) $) to node[midway, below] {$ e_1 $} ($ (U5.east) + (down:0.1) $);
\path[draw, <-] ($ (U3.west) + (up:0.1) $) to node[midway, above] {$ π_2 $} ($ (U5.east) + (up:0.1) $);
\path[draw, ->] ($ (U4.south) + (right:0.1) $) to node[midway, right] {$ e_1 $} ($ (U5.north) + (right:0.1) $);
\path[draw, <-] ($ (U4.south) + (left:0.1) $) to node[midway, left] {$ π_2 $} ($ (U5.north) + (left:0.1) $);
\end{tikzpicture}
\caption{}
\label{fig:stability-case-B5}
\end{subfigure}
\begin{subfigure}{0.3\linewidth}
\begin{tikzpicture}[scale=1]
\path (0, 0) node (C) {$ E_1 $};
\path (2, -2) node (U2) {$ E_2 $};
\path (4, 0) node (U3) {$ E_3 $};
\path (2, 2) node (U4) {$ E_4 $};
\path (2, 0) node (U5) {$ O_1 $};
\path[draw, ->] ($ (C.east) + (up:0.1) $) to node[midway, above] {$ e_1 $} ($ (U5.west) + (up:0.1) $);
\path[draw, <-] ($ (C.east) + (down:0.1) $) to node[midway, below] {$ π_2 $} ($ (U5.west) + (down:0.1) $);
\path[draw, ->] ($ (U2.north) + (left:0.1) $) to node[midway, left] {$ * $} ($ (U5.south) + (left:0.1) $);
\path[draw, <-] ($ (U2.north) + (right:0.1) $) to node[midway, right] {$ * $} ($ (U5.south) + (right:0.1) $);
\path[draw, ->] ($ (U3.west) + (down:0.1) $) to node[midway, below] {$ * $} ($ (U5.east) + (down:0.1) $);
\path[draw, <-] ($ (U3.west) + (up:0.1) $) to node[midway, above] {$ * $} ($ (U5.east) + (up:0.1) $);
\path[draw, ->] ($ (U4.south) + (right:0.1) $) to node[midway, right] {$ * $} ($ (U5.north) + (right:0.1) $);
\path[draw, <-] ($ (U4.south) + (left:0.1) $) to node[midway, left] {$ 0 $} ($ (U5.north) + (left:0.1) $);
\end{tikzpicture}
\caption{}
\label{fig:stability-case-C1}
\end{subfigure}
\begin{subfigure}{0.3\linewidth}
\begin{tikzpicture}[scale=1]
\path (0, 0) node (C) {$ E_1 $};
\path (2, -2) node (U2) {$ E_2 $};
\path (4, 0) node (U3) {$ E_3 $};
\path (2, 2) node (U4) {$ E_4 $};
\path (2, 0) node (U5) {$ O_1 $};
\path[draw, ->] ($ (C.east) + (up:0.1) $) to node[midway, above] {$ e_1 $} ($ (U5.west) + (up:0.1) $);
\path[draw, <-] ($ (C.east) + (down:0.1) $) to node[midway, below] {$ π_2 $} ($ (U5.west) + (down:0.1) $);
\path[draw, ->] ($ (U2.north) + (left:0.1) $) to node[midway, left] {$ e_1 $} ($ (U5.south) + (left:0.1) $);
\path[draw, <-] ($ (U2.north) + (right:0.1) $) to node[midway, right] {$ π_2 $} ($ (U5.south) + (right:0.1) $);
\path[draw, ->] ($ (U3.west) + (down:0.1) $) to node[midway, below] {$ * $} ($ (U5.east) + (down:0.1) $);
\path[draw, <-] ($ (U3.west) + (up:0.1) $) to node[midway, above] {$ 0 $} ($ (U5.east) + (up:0.1) $);
\path[draw, ->] ($ (U4.south) + (right:0.1) $) to node[midway, right] {$ * $} ($ (U5.north) + (right:0.1) $);
\path[draw, <-] ($ (U4.south) + (left:0.1) $) to node[midway, left] {$ 0 $} ($ (U5.north) + (left:0.1) $);
\end{tikzpicture}
\caption{}
\label{fig:stability-case-C2}
\end{subfigure}
\begin{subfigure}{0.3\linewidth}
\begin{tikzpicture}[scale=1]
\path (0, 0) node (C) {$ E_1 $};
\path (2, -2) node (U2) {$ E_2 $};
\path (4, 0) node (U3) {$ E_3 $};
\path (2, 2) node (U4) {$ E_4 $};
\path (2, 0) node (U5) {$ O_1 $};
\path[draw, ->] ($ (C.east) + (up:0.1) $) to node[midway, above] {$ e_1 $} ($ (U5.west) + (up:0.1) $);
\path[draw, <-] ($ (C.east) + (down:0.1) $) to node[midway, below] {$ 0 $} ($ (U5.west) + (down:0.1) $);
\path[draw, ->] ($ (U2.north) + (left:0.1) $) to node[midway, left] {$ e_1 $} ($ (U5.south) + (left:0.1) $);
\path[draw, <-] ($ (U2.north) + (right:0.1) $) to node[midway, right] {$ * $} ($ (U5.south) + (right:0.1) $);
\path[draw, ->] ($ (U3.west) + (down:0.1) $) to node[midway, below] {$ e_1 $} ($ (U5.east) + (down:0.1) $);
\path[draw, <-] ($ (U3.west) + (up:0.1) $) to node[midway, above] {$ * $} ($ (U5.east) + (up:0.1) $);
\path[draw, ->] ($ (U4.south) + (right:0.1) $) to node[midway, right] {$ e_1 $} ($ (U5.north) + (right:0.1) $);
\path[draw, <-] ($ (U4.south) + (left:0.1) $) to node[midway, left] {$ * $} ($ (U5.north) + (left:0.1) $);
\end{tikzpicture}
\caption{}
\label{fig:stability-case-C3}
\end{subfigure}
\caption{This figures lists all representations of the Kleinian $ D_4 $ quiver which are not $ θ_2 $-semistable. The star symbol $ * $ stands for arbitrary values. \autoref{fig:stability-case-C1} more generally concerns the case where $ φ_{55, 1} (- ¤ e_1) = π_2 $ and for at least one $ i ∈ \{2, 3, 4\} $ we have $ φ_{i5} (- ¤ e_1) = 0 $.}
\label{fig:stability-cases}
\end{figure}
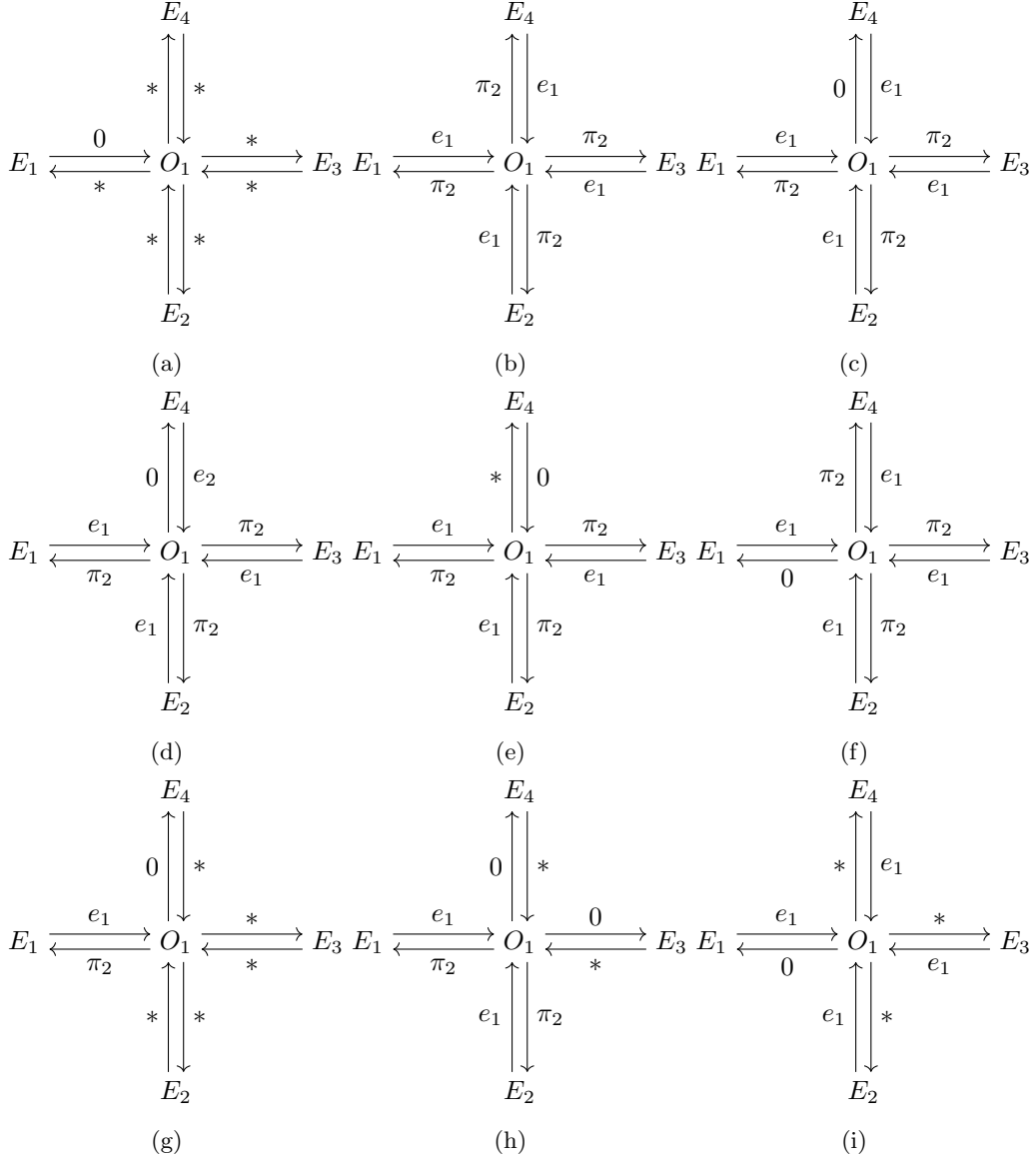
\begin{figure}
\ContinuedFloat
\centering
\begin{subfigure}{0.3\linewidth}
\begin{tikzpicture}[scale=1]
\path (0, 0) node (C) {$ E_1 $};
\path (2, -2) node (U2) {$ E_2 $};
\path (4, 0) node (U3) {$ E_3 $};
\path (2, 2) node (U4) {$ E_4 $};
\path (2, 0) node (U5) {$ O_1 $};
\path[draw, ->] ($ (C.east) + (up:0.1) $) to node[midway, above] {$ e_1 $} ($ (U5.west) + (up:0.1) $);
\path[draw, <-] ($ (C.east) + (down:0.1) $) to node[midway, below] {$ 0 $} ($ (U5.west) + (down:0.1) $);
\path[draw, ->] ($ (U2.north) + (left:0.1) $) to node[midway, left] {$ e_1 $} ($ (U5.south) + (left:0.1) $);
\path[draw, <-] ($ (U2.north) + (right:0.1) $) to node[midway, right] {$ π_2 $} ($ (U5.south) + (right:0.1) $);
\path[draw, ->] ($ (U3.west) + (down:0.1) $) to node[midway, below] {$ e_1 $} ($ (U5.east) + (down:0.1) $);
\path[draw, <-] ($ (U3.west) + (up:0.1) $) to node[midway, above] {$ π_2 $} ($ (U5.east) + (up:0.1) $);
\path[draw, ->] ($ (U4.south) + (right:0.1) $) to node[midway, right] {$ e_2 $} ($ (U5.north) + (right:0.1) $);
\path[draw, <-] ($ (U4.south) + (left:0.1) $) to node[midway, left] {$ 0 $} ($ (U5.north) + (left:0.1) $);
\end{tikzpicture}
\caption{}
\label{fig:stability-case-C4}
\end{subfigure}
\begin{subfigure}{0.3\linewidth}
\begin{tikzpicture}[scale=1]
\path (0, 0) node (C) {$ E_1 $};
\path (2, -2) node (U2) {$ E_2 $};
\path (4, 0) node (U3) {$ E_3 $};
\path (2, 2) node (U4) {$ E_4 $};
\path (2, 0) node (U5) {$ O_1 $};
\path[draw, ->] ($ (C.east) + (up:0.1) $) to node[midway, above] {$ e_1 $} ($ (U5.west) + (up:0.1) $);
\path[draw, <-] ($ (C.east) + (down:0.1) $) to node[midway, below] {$ 0 $} ($ (U5.west) + (down:0.1) $);
\path[draw, ->] ($ (U2.north) + (left:0.1) $) to node[midway, left] {$ e_1 $} ($ (U5.south) + (left:0.1) $);
\path[draw, <-] ($ (U2.north) + (right:0.1) $) to node[midway, right] {$ π_2 $} ($ (U5.south) + (right:0.1) $);
\path[draw, ->] ($ (U3.west) + (down:0.1) $) to node[midway, below] {$ e_1 $} ($ (U5.east) + (down:0.1) $);
\path[draw, <-] ($ (U3.west) + (up:0.1) $) to node[midway, above] {$ π_2 $} ($ (U5.east) + (up:0.1) $);
\path[draw, ->] ($ (U4.south) + (right:0.1) $) to node[midway, right] {$ 0 $} ($ (U5.north) + (right:0.1) $);
\path[draw, <-] ($ (U4.south) + (left:0.1) $) to node[midway, left] {$ π_1 $} ($ (U5.north) + (left:0.1) $);
\end{tikzpicture}
\caption{}
\label{fig:stability-case-C5}
\end{subfigure}
\begin{subfigure}{0.3\linewidth}
\begin{tikzpicture}[scale=1]
\path (0, 0) node (C) {$ E_1 $};
\path (2, -2) node (U2) {$ E_2 $};
\path (4, 0) node (U3) {$ E_3 $};
\path (2, 2) node (U4) {$ E_4 $};
\path (2, 0) node (U5) {$ O_1 $};
\path[draw, ->] ($ (C.east) + (up:0.1) $) to node[midway, above] {$ e_1 $} ($ (U5.west) + (up:0.1) $);
\path[draw, <-] ($ (C.east) + (down:0.1) $) to node[midway, below] {$ 0 $} ($ (U5.west) + (down:0.1) $);
\path[draw, ->] ($ (U2.north) + (left:0.1) $) to node[midway, left] {$ e_2 $} ($ (U5.south) + (left:0.1) $);
\path[draw, <-] ($ (U2.north) + (right:0.1) $) to node[midway, right] {$ π_1 $} ($ (U5.south) + (right:0.1) $);
\path[draw, ->] ($ (U3.west) + (down:0.1) $) to node[midway, below] {$ e_2 $} ($ (U5.east) + (down:0.1) $);
\path[draw, <-] ($ (U3.west) + (up:0.1) $) to node[midway, above] {$ π_1 $} ($ (U5.east) + (up:0.1) $);
\path[draw, ->] ($ (U4.south) + (right:0.1) $) to node[midway, right] {$ * $} ($ (U5.north) + (right:0.1) $);
\path[draw, <-] ($ (U4.south) + (left:0.1) $) to node[midway, left] {$ 0 $} ($ (U5.north) + (left:0.1) $);
\end{tikzpicture}
\caption{}
\label{fig:stability-case-C6}
\end{subfigure}
\begin{subfigure}{0.3\linewidth}
\begin{tikzpicture}[scale=1]
\path (0, 0) node (C) {$ E_1 $};
\path (2, -2) node (U2) {$ E_2 $};
\path (4, 0) node (U3) {$ E_3 $};
\path (2, 2) node (U4) {$ E_4 $};
\path (2, 0) node (U5) {$ O_1 $};
\path[draw, ->] ($ (C.east) + (up:0.1) $) to node[midway, above] {$ e_1 $} ($ (U5.west) + (up:0.1) $);
\path[draw, <-] ($ (C.east) + (down:0.1) $) to node[midway, below] {$ 0 $} ($ (U5.west) + (down:0.1) $);
\path[draw, ->] ($ (U2.north) + (left:0.1) $) to node[midway, left] {$ 0 $} ($ (U5.south) + (left:0.1) $);
\path[draw, <-] ($ (U2.north) + (right:0.1) $) to node[midway, right] {$ * $} ($ (U5.south) + (right:0.1) $);
\path[draw, ->] ($ (U3.west) + (down:0.1) $) to node[midway, below] {$ 0 $} ($ (U5.east) + (down:0.1) $);
\path[draw, <-] ($ (U3.west) + (up:0.1) $) to node[midway, above] {$ * $} ($ (U5.east) + (up:0.1) $);
\path[draw, ->] ($ (U4.south) + (right:0.1) $) to node[midway, right] {$ 0 $} ($ (U5.north) + (right:0.1) $);
\path[draw, <-] ($ (U4.south) + (left:0.1) $) to node[midway, left] {$ * $} ($ (U5.north) + (left:0.1) $);
\end{tikzpicture}
\caption{}
\label{fig:stability-case-D1}
\end{subfigure}
\begin{subfigure}{0.3\linewidth}
\begin{tikzpicture}[scale=1]
\path (0, 0) node (C) {$ E_1 $};
\path (2, -2) node (U2) {$ E_2 $};
\path (4, 0) node (U3) {$ E_3 $};
\path (2, 2) node (U4) {$ E_4 $};
\path (2, 0) node (U5) {$ O_1 $};
\path[draw, ->] ($ (C.east) + (up:0.1) $) to node[midway, above] {$ e_1 $} ($ (U5.west) + (up:0.1) $);
\path[draw, <-] ($ (C.east) + (down:0.1) $) to node[midway, below] {$ 0 $} ($ (U5.west) + (down:0.1) $);
\path[draw, ->] ($ (U2.north) + (left:0.1) $) to node[midway, left] {$ 0 $} ($ (U5.south) + (left:0.1) $);
\path[draw, <-] ($ (U2.north) + (right:0.1) $) to node[midway, right] {$ * $} ($ (U5.south) + (right:0.1) $);
\path[draw, ->] ($ (U3.west) + (down:0.1) $) to node[midway, below] {$ 0 $} ($ (U5.east) + (down:0.1) $);
\path[draw, <-] ($ (U3.west) + (up:0.1) $) to node[midway, above] {$ * $} ($ (U5.east) + (up:0.1) $);
\path[draw, ->] ($ (U4.south) + (right:0.1) $) to node[midway, right] {$ e_2 $} ($ (U5.north) + (right:0.1) $);
\path[draw, <-] ($ (U4.south) + (left:0.1) $) to node[midway, left] {$ 0 $} ($ (U5.north) + (left:0.1) $);
\end{tikzpicture}
\caption{}
\label{fig:stability-case-D2}
\end{subfigure}
\begin{subfigure}{0.3\linewidth}
\begin{tikzpicture}[scale=1]
\path (0, 0) node (C) {$ E_1 $};
\path (2, -2) node (U2) {$ E_2 $};
\path (4, 0) node (U3) {$ E_3 $};
\path (2, 2) node (U4) {$ E_4 $};
\path (2, 0) node (U5) {$ O_1 $};
\path[draw, ->] ($ (C.east) + (up:0.1) $) to node[midway, above] {$ e_1 $} ($ (U5.west) + (up:0.1) $);
\path[draw, <-] ($ (C.east) + (down:0.1) $) to node[midway, below] {$ 0 $} ($ (U5.west) + (down:0.1) $);
\path[draw, ->] ($ (U2.north) + (left:0.1) $) to node[midway, left] {$ 0 $} ($ (U5.south) + (left:0.1) $);
\path[draw, <-] ($ (U2.north) + (right:0.1) $) to node[midway, right] {$ 0 $} ($ (U5.south) + (right:0.1) $);
\path[draw, ->] ($ (U3.west) + (down:0.1) $) to node[midway, below] {$ 0 $} ($ (U5.east) + (down:0.1) $);
\path[draw, <-] ($ (U3.west) + (up:0.1) $) to node[midway, above] {$ 0 $} ($ (U5.east) + (up:0.1) $);
\path[draw, ->] ($ (U4.south) + (right:0.1) $) to node[midway, right] {$ e_2 $} ($ (U5.north) + (right:0.1) $);
\path[draw, <-] ($ (U4.south) + (left:0.1) $) to node[midway, left] {$ 0 $} ($ (U5.north) + (left:0.1) $);
\end{tikzpicture}
\caption{}
\label{fig:stability-case-D2s}
\end{subfigure}
\begin{subfigure}{0.3\linewidth}
\begin{tikzpicture}[scale=1]
\path (0, 0) node (C) {$ E_1 $};
\path (2, -2) node (U2) {$ E_2 $};
\path (4, 0) node (U3) {$ E_3 $};
\path (2, 2) node (U4) {$ E_4 $};
\path (2, 0) node (U5) {$ O_1 $};
\path[draw, ->] ($ (C.east) + (up:0.1) $) to node[midway, above] {$ e_1 $} ($ (U5.west) + (up:0.1) $);
\path[draw, <-] ($ (C.east) + (down:0.1) $) to node[midway, below] {$ 0 $} ($ (U5.west) + (down:0.1) $);
\path[draw, ->] ($ (U2.north) + (left:0.1) $) to node[midway, left] {$ 0 $} ($ (U5.south) + (left:0.1) $);
\path[draw, <-] ($ (U2.north) + (right:0.1) $) to node[midway, right] {$ * $} ($ (U5.south) + (right:0.1) $);
\path[draw, ->] ($ (U3.west) + (down:0.1) $) to node[midway, below] {$ * $} ($ (U5.east) + (down:0.1) $);
\path[draw, <-] ($ (U3.west) + (up:0.1) $) to node[midway, above] {$ 0 $} ($ (U5.east) + (up:0.1) $);
\path[draw, ->] ($ (U4.south) + (right:0.1) $) to node[midway, right] {$ * $} ($ (U5.north) + (right:0.1) $);
\path[draw, <-] ($ (U4.south) + (left:0.1) $) to node[midway, left] {$ 0 $} ($ (U5.north) + (left:0.1) $);
\end{tikzpicture}
\caption{}
\label{fig:stability-case-D3}
\end{subfigure}
\begin{subfigure}{0.3\linewidth}
\begin{tikzpicture}[scale=1]
\path (0, 0) node (C) {$ E_1 $};
\path (2, -2) node (U2) {$ E_2 $};
\path (4, 0) node (U3) {$ E_3 $};
\path (2, 2) node (U4) {$ E_4 $};
\path (2, 0) node (U5) {$ O_1 $};
\path[draw, ->] ($ (C.east) + (up:0.1) $) to node[midway, above] {$ e_1 $} ($ (U5.west) + (up:0.1) $);
\path[draw, <-] ($ (C.east) + (down:0.1) $) to node[midway, below] {$ 0 $} ($ (U5.west) + (down:0.1) $);
\path[draw, ->] ($ (U2.north) + (left:0.1) $) to node[midway, left] {$ * $} ($ (U5.south) + (left:0.1) $);
\path[draw, <-] ($ (U2.north) + (right:0.1) $) to node[midway, right] {$ 0 $} ($ (U5.south) + (right:0.1) $);
\path[draw, ->] ($ (U3.west) + (down:0.1) $) to node[midway, below] {$ * $} ($ (U5.east) + (down:0.1) $);
\path[draw, <-] ($ (U3.west) + (up:0.1) $) to node[midway, above] {$ 0 $} ($ (U5.east) + (up:0.1) $);
\path[draw, ->] ($ (U4.south) + (right:0.1) $) to node[midway, right] {$ * $} ($ (U5.north) + (right:0.1) $);
\path[draw, <-] ($ (U4.south) + (left:0.1) $) to node[midway, left] {$ 0 $} ($ (U5.north) + (left:0.1) $);
\end{tikzpicture}
\caption{}
\label{fig:stability-case-D4}
\end{subfigure}
\caption{Continued.}
\label{fig:stability-cases}
\end{figure}

In \autoref{fig:stability-cases}, we list all representations of the Kleinian $ D_4 $ quiver which are not $ θ_2 $-stable. We have depicted all of them only up to gauging, and in fact depicted them in a sloppy way. More precisely, we have ignored all scalar values. With respect to the sign convention that the arrows pointing towards the central vertex are the starred arrows of the double quiver $ \bar Q $, the preprojective conditions are not exactly satisfied and the values on the arrows should for instance rather be read $ ± e_1 $ instead of $ e_1 $. All these details do not play a role for our investigations of stability. Let us state the fact that we have listed all non-semistable representations as follows:

\begin{lemma}
Let $ (Q, α) $ be the Kleinian $ D_4 $ quiver and $ ρ ∈ \Rep(Π_Q, α) $ be a non-$ θ_2 $-semistable representation. Then $ ρ $ is one of the representations listed in \autoref{fig:stability-cases}, up to gauging.
\end{lemma}

We are now ready to prove that if $ (φ, x) ∈ \CG_Γ × ℂ^2 $ is $ θ_2 $-semistable, then $ R(φ, x) ∈ \Rep(Π_Q, α) $ is $ θ_2 $-semistable. The strategy is to prove the contraposite and use the list of non-$ θ_2 $-semistable representations from \autoref{fig:stability-cases}.

\begin{lemma}
\label{th:caseD-theta2-semistable}
Let $ (Q, α) $ be the Kleinian $ D_4 $ quiver and $ (φ, x) ∈ \CG_Γ × ℂ^2 $ be $ θ_2 $-semistable. Then $ R(φ, x) ∈ \Rep(Π_Q, α) $ is $ θ_2 $-semistable.
\end{lemma}

\begin{proof}
We achieve the statement by applying the Hilbert-Mumford criterion. The criterion states that $ (φ, x) $ is not $ θ_2 $-semistable if and only if there exists a one-parameter subgroup $ g(t) ⊂ G $ which pairs positively with $ θ_2 $ and for which $ \lim_{t → 0} g(t)^{-1} (φ, x) $ exists within $ \CG_Γ × ℂ^2 $. Since $ \CG_Γ $ is a closed subset of an affine space, it suffices to simply show that $ g(t)^{-1} (φ, x) $ converges within the affine space.

We prove the statement by contraposition: Assume that $ R(φ, x) $ is a non-$ θ_2 $-semistable representation. It is our goal to deduce that $ (φ, x) $ is not $ θ_2 $-semistable. The property of $ θ_2 $-semistability is invariant under the $ G $-action, therefore it suffices to check it for all non-$ θ_2 $-semistable representations up to gauging. We achieve this by going through all the cases listed in \autoref{fig:stability-cases}. Apart from the representation in \autoref{fig:stability-case-0}, where $ x = 0 $, the listed representations all have $ x = e_1 $.

\begin{itemize}
\item Regard the representation in \autoref{fig:stability-case-0}. We read off that $ x = 0 $. Pick the one-parameter subgroup $ g(t) = (t, t, t, t I_2) $. Then $ \lim_{t → 0} g(t)^{-1} (φ, x) $ exists and in fact vanishes. The critical ingredient is $ x = 0 $.
\item Regard the representation in \autoref{fig:stability-case-A}. Pick the one-parameter subgroup $ g(t) = (1, 1, 1, \pmat{1 & 0 \\ 0 & t}) $. Then $ \lim_{t → 0} g(t)^{-1} (φ, x) $ exists. Indeed, the critical ingredients are $ φ_{E_i, O_1} (e_1) = e_1 $ and $ φ_{O_1, O_1; E_i} (e_1 ¤ e_1) = 0 $.
\item Regard the representation in \autoref{fig:stability-case-B1}. Pick the one-parameter subgroup $ g(t) = (1, 1, 1, \pmat{1 & 0 \\ 0 & t}) $. Then $ \lim_{t → 0} g(t)^{-1} (φ, x) $ exists. Indeed, the critical ingredients are $ φ_{E_i, O_1} (e_1) = e_1 $ and $ φ_{O_1, O_1; E_i} (e_1 ¤ e_1) = 0 $.
\item Regard the representation in \autoref{fig:stability-case-B2}. Regard the diagram for $ E_4 ¤ O_1 ¤ O_1 $ with entry vector $ 1 ¤ e_2 ¤ e_1 $. Since $ φ_{O_1, O_1; E_1} (e_2 ¤ e_1) = 1 $ and $ φ_{O_1, O_1; E_4} (- ¤ e_1) = 0 $, we have a contradiction and conclude the representation is not in the image of $ R $.
\item Regard the representation in \autoref{fig:stability-case-B3}. Regard the one-parameter subgroup $ g(t) = (t, t, 1, \pmat{1 & 0 \\ 0 & t}) $. Then $ \lim_{t → 0} g(t)^{-1} (φ, x) $ exists. Indeed, the critical ingredients are $ φ_{E_i, O_1} (e_1) ∈ \vspan(e_1) $ and $ φ_{O_1, O_1; E_2} (e_1 ¤ e_1) = φ_{O_1, O_1; E_3} (e_1 ¤ e_1) = 0 $.
\item Regard the representation in \autoref{fig:stability-case-B5}. Regard the one-parameter subgroup $ g(t) = (1, 1, 1, \pmat{1 & 0 \\ 0 & t}) $. Then $ \lim_{t → 0} g(t)^{-1} (φ, x) $ exists. Indeed, the critical ingredients are $ φ_{E_i, O_1} (e_1) = e_1 $ and $ φ_{O_1, O_1; E_i} (e_1 ¤ e_1) = 0 $.
\item Regard the representation in \autoref{fig:stability-case-C1}. Regard the diagram for $ O_4 ¤ O_1 ¤ O_1 $ with entry vector $ 1 ¤ e_2 ¤ e_1 $. Since $ φ_{O_1, O_1; E_1} (e_2 ¤ e_1) = 1 $ and $ φ_{O_1, O_1; E_4} (- ¤ e_1) = 0 $, we have a contradiction.
\item Regard the representation in \autoref{fig:stability-case-C2}. Regard the one-parameter subgroup $ g(t) = (1, 1, 1, \pmat{1 & 0 \\ 0 & t}) $. Then $ \lim_{t → 0} g(t)^{-1} (φ, x) $ exists. Indeed, the critical ingredient is $ φ_{E_i, O_1} (e_1) ∈ \vspan(e_1) $.
\item Regard the representation in \autoref{fig:stability-case-C3}. Regard the one-parameter subgroup $ g(t) = (1, 1, 1, \pmat{1 & 0 \\ 0 & t}) $. Then $ \lim_{t → 0} g(t)^{-1} (φ, x) $ exists. Indeed, the critical ingredients are $ φ_{E_i, O_1} (e_1) = e_1 $.
\item Regard the representation in \autoref{fig:stability-case-C4}. Regard the one-parameter subgroup $ g(t) = (1, 1, t, \pmat{1 & 0 \\ 0 & t}) $. We claim that $ \lim_{t → 0} g(t)^{-1} (φ, x) $ exists. Indeed, regard the diagram for $ E_2 ¤ O_1 ¤ O_1 $ with entry vector $ 1 ¤ e_2 ¤ e_1 $. Since $ φ_{O_1, O_1; E_3} (e_2 ¤ e_1) = 1 $ and $ φ_{O_1, O_1; E_4} (- ¤ e_1) = 0 $, we read off that $ φ_{E_2, E_3} = 0 $. Now $ \lim_{t → 0} g(t)^{-1} (φ, x) $ exists. Indeed, the critical ingredients are $ φ_{E_2, E_3} = 0 $, $ φ_{E_2, O_1} (e_1) = e_1 $, $ φ_{E_3, O_1} (e_1) = e_1 $, $ φ_{E_4, O_1} (e_1) = e_2 $ and $ φ_{O_1, O_1; E_4} (e_1 ¤ e_1) = 0 $.
\item Regard the representation in \autoref{fig:stability-case-C5}. Regard the one-parameter subgroup $ g(t) = (1, 1, 1, \pmat{1 & 0 \\ 0 & t}) $. Then $ \lim_{t → 0} g(t)^{-1} (φ, x) $ exists. Indeed, the critical ingredients are $ φ_{E_i, O_1} (e_1) ∈ \vspan(e_1) $.
\item Regard the representation in \autoref{fig:stability-case-C6}. The first step is to prove that $ φ_{O_1, O_1; E_4} = 0 $. We start with the observation that $ φ_{E_2, O_1} (e_1) = e_2 $ and $ φ_{O_1, O_1; E_4} (e_1 ¤ e_1) = 0 $ and $ φ_{O_1, O_1; E_4} (e_2 ¤ e_1) = 0 $. We use the diagram for $ O_1 ¤ E_2 ¤ O_1 $ with input vector $ e_1 ¤ 1 ¤ e_1 $ to conclude that $ φ_{O_1, O_1; E_4} (e_1 ¤ e_2) = 0 $. Next, we use the diagram for $ E_2 ¤ O_1 ¤ O_1 $ with input vector $ 1 ¤ e_1 ¤ e_1 $ to conclude that $ φ_{E_2, E_2} = 0 $. Finally, we use the diagram for $ E_3 ¤ O_1 ¤ O_1 $ with input vector $ 1 ¤ e_1 ¤ e_1 $ to conclude that $ φ_{E_2, E_3} = 0 $. We use the diagram for $ E_2 ¤ O_1 ¤ O_1 $ with input vector $ 1 ¤ e_1 ¤ e_2 $ to conclude that $ φ_{O_1, O_1; E_4} (e_2 ¤ e_2) = 0 $. This shows $ φ_{O_1, O_1; E_4} = 0 $. Now pick the one-parameter subgroup $ g(t) = (1, 1, t, I_2) $. We shall readily check that $ \lim_{t → 0} g(t)^{-1} (φ, x) $ exists. The critical ingredients are that $ φ_{O_1, O_1; E_4} = 0 $ and $ φ_{E_i, E_4} $ converges to zero for $ i = 2, 3, 4 $.
\item Regard the representation in \autoref{fig:stability-case-D1}. Regard the one-parameter subgroup $ g(t) = (1, 1, 1, \pmat{1 & 0 \\ 0 & t}) $. Then $ \lim_{t → 0} g(t)^{-1} (φ, x) $ exists. Indeed, the critical ingredients are $ φ_{E_i, O_1} (e_1) = 0 $.
\item Regard the representation in \autoref{fig:stability-case-D2}. We can assume that $ φ_{O_1, O_1; E_2} (- ¤ e_1) ≠ 0 $ or $ φ_{O_1, O_1; E_3} (- ¤ e_1) ≠ 0 $, otherwise we are covered by \autoref{fig:stability-case-D2s}. Without loss of generality assume that $ φ_{O_1, O_1; E_3} (- ¤ e_1) ≠ 0 $. Let $ y $ denote an element of $ O_1 $ which does not lie in the kernel. Regard the one-parameter subgroup $ g(t) = (1, 1, t, \pmat{1 & 0 \\ 0 & t}) $. Now regard the diagram for $ E_2 ¤ O_1 ¤ O_1 $ with entry vector $ 1 ¤ y ¤ e_1 $. Since $ φ_{O_1, O_1; E_3} (y ¤ e_1) ≠ 0 $ and $ φ_{O_1, O_1; E_4} (- ¤ e_1) = 0 $, we get $ φ_{E_2, E_3} = 0 $. We now see that $ \lim_{t → 0} g(t)^{-1} (φ, x) $ exists. The critical ingredients are $ φ_{E_2, E_3} = 0 $, $ φ_{O_1, O_1; E_4} (e_1 ¤ e_1) = 0 $, $ φ_{E_2, O_1} (e_1) = 0 $ and $ φ_{E_3, O_1} (e_1) = 0 $.
\item Regard the representation in \autoref{fig:stability-case-D2s}. Regard the one-parameter subgroup $ g(t) = (t, t, t, \pmat{1 & 0 \\ 0 & t}) $. Then $ \lim_{t → 0} g(t)^{-1} (φ, x) $ exists. Indeed, the critical ingredient is $ φ_{O_1, O_1} (- ¤ e_1) = 0 $.
\item Regard the representation in \autoref{fig:stability-case-D3}. Regard the one-parameter subgroup $ g(t) = (1, t, t, \pmat{1 & 0 \\ 0 & t}) $. Then $ \lim_{t → 0} g(t)^{-1} (φ, x) $ exists. Indeed, the critical ingredients are $ φ_{E_2, O_1} (e_1) = 0 $, $ φ_{O_1, O_1; E_3} (e_1 ¤ e_1) = 0 $ and $ φ_{O_1, O_1; E_4} (e_1 ¤ e_1) = 0 $.
\item Regard the representation in \autoref{fig:stability-case-D4}. Regard the one-parameter subgroup $ g(t) = (t, t, t, \pmat{1 & 0 \\ 0 & t}) $. Then $ \lim_{t → 0} g(t)^{-1} (φ, x) $ exists. Indeed, the critical ingredient is $ φ_{O_1, O_1} (e_1 ¤ e_1) = 0 $.
\end{itemize}

In every case, we have verified that $ (φ, x) $ is not $ θ_2 $-semistable. This finishes the proof.
\end{proof}

\begin{figure}
\centering
\begin{tikzpicture}
\begin{scope}[shift={(8, 5)}]
\path (0, 1) node (A) {$ E_1 $};
\path (0, -1) node (B) {$ E_2 $};
\path (1, 0) node (C) {$ O_1 $};
\path (2, 0) node (D) {$ I_2 $};
\path (3, 0) node (E) {$ … $};
\path (4, 0) node (F) {\tinymath{$ I_{n-4} $}};
\path (5, 0) node (G) {\tinymath{$ O_{n-3} $}};
\path (6, 1) node (H) {$ E_3 $};
\path (6, -1) node (I) {$ E_4 $};
\path[draw, <-] ($ (A) + (335:0.3) $) to node[pos=0.5, right] {$ 0 $} ($ (C) + (115:0.3) $);
\path[draw, ->] ($ (A) + (295:0.3) $) to node[pos=0.6, left] {$ e_1 $} ($ (C) + (155:0.3) $);
\path[draw, ->] ($ (B) + (65:0.3) $) to node[pos=0.4, left] {$ e_2 $} ($ (C) + (205:0.3) $);
\path[draw, <-] ($ (B) + (25:0.3) $) to node[pos=0.3, right] {$ π_1 $} ($ (C) + (245:0.3) $);
\path[draw, ->] ($ (C) + (20:0.3) $) to node[midway, above] {$ \tiny \pmat{1 & 0 \\ 0 & 1} $} ($ (D) + (160:0.3) $);
\path[draw, <-] ($ (C) + (340:0.3) $) to node[midway, below] {$ \tiny \begin{pmatrix} 0 & 0 \\ 1 & 0 \end{pmatrix} $} ($ (D) + (200:0.3) $);
\path[draw, ->] ($ (F) + (20:0.3) $) to node[midway, above] {$ \tiny \pmat{1 & 0 \\ 0 & 1} $} ($ (G) + (160:0.3) $);
\path[draw, <-] ($ (F) + (340:0.3) $) to node[midway, below] {$ \tiny \begin{pmatrix} 0 & 0 \\ 1 & 0 \end{pmatrix} $} ($ (G) + (200:0.3) $);
\path[draw, ->] ($ (G) + (65:0.3) $) to node[pos=0.8, left] {$ (1 ~ a) $} ($ (H) + (205:0.3) $);
\path[draw, <-] ($ (G) + (25:0.3) $) to node[pos=0.6, right] {$ 0 $} ($ (H) + (245:0.3) $);
\path[draw, <-] ($ (G) + (335:0.3) $) to node[pos=0.6, right] {$ e_2 $} ($ (I) + (115:0.3) $);
\path[draw, ->] ($ (G) + (295:0.3) $) to node[pos=0.6, left] {$ π_1 $} ($ (I) + (155:0.3) $);
\end{scope}
\begin{scope}[shift={(8, -5)}]
\path (0, 1) node (A) {$ E_1 $};
\path (0, -1) node (B) {$ E_2 $};
\path (1, 0) node (C) {$ O_1 $};
\path (2, 0) node (D) {$ I_2 $};
\path (3, 0) node (E) {$ … $};
\path (4, 0) node (F) {\tinymath{$ I_{n-4} $}};
\path (5, 0) node (G) {\tinymath{$ O_{n-3} $}};
\path (6, 1) node (H) {$ E_3 $};
\path (6, -1) node (I) {$ E_4 $};
\path[draw, <-] ($ (A) + (335:0.3) $) to node[pos=0.3, right] {$ 0 $} ($ (C) + (115:0.3) $);
\path[draw, ->] ($ (A) + (295:0.3) $) to node[pos=0.6, left] {$ e_1 $} ($ (C) + (155:0.3) $);
\path[draw, ->] ($ (B) + (65:0.3) $) to node[pos=0.4, left] {$ e_2 $} ($ (C) + (205:0.3) $);
\path[draw, <-] ($ (B) + (25:0.3) $) to node[pos=0.3, right] {$ π_1 $} ($ (C) + (245:0.3) $);
\path[draw, ->] ($ (C) + (20:0.3) $) to node[midway, above] {$ \tiny \pmat{1 & 0 \\ 0 & 1} $} ($ (D) + (160:0.3) $);
\path[draw, <-] ($ (C) + (340:0.3) $) to node[midway, below] {$ \tiny \begin{pmatrix} 0 & 0 \\ 1 & 0 \end{pmatrix} $} ($ (D) + (200:0.3) $);
\path[draw, ->] ($ (F) + (20:0.3) $) to node[midway, above] {$ \tiny \pmat{1 & 0 \\ 0 & 1} $} ($ (G) + (160:0.3) $);
\path[draw, <-] ($ (F) + (340:0.3) $) to node[midway, below] {$ \tiny \begin{pmatrix} 0 & 0 \\ 1 & 0 \end{pmatrix} $} ($ (G) + (200:0.3) $);
\path[draw, ->] ($ (G) + (65:0.3) $) to node[pos=0.8, left] {$ π_1 $} ($ (H) + (205:0.3) $);
\path[draw, <-] ($ (G) + (25:0.3) $) to node[pos=0.6, right] {$ e_2 $} ($ (H) + (245:0.3) $);
\path[draw, <-] ($ (G) + (335:0.3) $) to node[pos=0.6, right] {$ 0 $} ($ (I) + (115:0.3) $);
\path[draw, ->] ($ (G) + (295:0.3) $) to node[pos=0.8, left] {$ (1 ~ a) $} ($ (I) + (155:0.3) $);
\end{scope}
\begin{scope}[shift={(0, 2)}, transform canvas={scale=0.7}]
\path (0, 1) node (A) {$ E_1 $};
\path (0, -1) node (B) {$ E_2 $};
\path (1, 0) node (C) {$ O_1 $};
\path (2, 0) node (D) {$ I_2 $};
\path (3, 0) node (E) {$ … $};
\path (4, 0) node (F) {\tinymath{$ I_{n-4} $}};
\path (5, 0) node (G) {\tinymath{$ O_{n-3} $}};
\path (6, 1) node (H) {$ E_3 $};
\path (6, -1) node (I) {$ E_4 $};
\path[draw, <-] ($ (A) + (335:0.3) $) to node[pos=0.3, right] {$ 0 $} ($ (C) + (115:0.3) $);
\path[draw, ->] ($ (A) + (295:0.3) $) to node[pos=0.6, left] {$ e_1 $} ($ (C) + (155:0.3) $);
\path[draw, ->] ($ (B) + (65:0.3) $) to node[pos=0.4, left] {$ e_2 $} ($ (C) + (205:0.3) $);
\path[draw, <-] ($ (B) + (25:0.3) $) to node[pos=0.3, right] {$ π_1 $} ($ (C) + (245:0.3) $);
\path[draw, ->] ($ (C) + (20:0.3) $) to node[midway, above] {$ \tiny \begin{pmatrix} 0 & 0 \\ 1 & 0 \end{pmatrix} $} ($ (D) + (160:0.3) $);
\path[draw, <-] ($ (C) + (340:0.3) $) to node[midway, below] {$ \tiny \begin{pmatrix} 1 & 0 \\ 0 & 1 \end{pmatrix} $} ($ (D) + (200:0.3) $);
\path[draw, ->] ($ (F) + (20:0.3) $) to node[midway, above] {$ \tiny \begin{pmatrix} 0 & 0 \\ 1 & 0 \end{pmatrix} $} ($ (G) + (160:0.3) $);
\path[draw, <-] ($ (F) + (340:0.3) $) to node[midway, below] {$ \tiny \begin{pmatrix} 1 & 0 \\ 0 & 1 \end{pmatrix} $} ($ (G) + (200:0.3) $);
\path[draw, ->] ($ (G) + (65:0.3) $) to node[pos=0.6, left] {$ π_1 $} ($ (H) + (205:0.3) $);
\path[draw, <-] ($ (G) + (25:0.3) $) to node[pos=0.6, right] {$ (a+1)e_2 $} ($ (H) + (245:0.3) $);
\path[draw, <-] ($ (G) + (335:0.3) $) to node[pos=0.6, right] {$ -ae_2 $} ($ (I) + (115:0.3) $);
\path[draw, ->] ($ (G) + (295:0.3) $) to node[pos=0.6, left] {$ π_1 $} ($ (I) + (155:0.3) $);
\end{scope}
\begin{scope}[shift={(0, -5)}]
\path (0, 1) node (A) {$ E_1 $};
\path (0, -1) node (B) {$ E_2 $};
\path (1, 0) node (C) {$ O_1 $};
\path (2, 0) node (D) {$ I_2 $};
\path (3, 0) node (E) {$ … $};
\path (4, 0) node (F) {\tinymath{$ I_{n-4} $}};
\path (5, 0) node (G) {\tinymath{$ O_{n-3} $}};
\path (6, 1) node (H) {$ E_3 $};
\path (6, -1) node (I) {$ E_4 $};
\path[draw, <-] ($ (A) + (335:0.3) $) to node[pos=0.3, right] {$ 0 $} ($ (C) + (115:0.3) $);
\path[draw, ->] ($ (A) + (295:0.3) $) to node[pos=0.6, left] {$ e_1 $} ($ (C) + (155:0.3) $);
\path[draw, ->] ($ (B) + (65:0.3) $) to node[pos=0.4, left] {$ 0 $} ($ (C) + (205:0.3) $);
\path[draw, <-] ($ (B) + (25:0.3) $) to node[pos=0.3, right] {$ (1 ~ a) $} ($ (C) + (245:0.3) $);
\path[draw, ->] ($ (C) + (20:0.3) $) to node[midway, above] {$ \tiny \begin{pmatrix} 1 & 0 \\ 0 & 0 \end{pmatrix} $} ($ (D) + (160:0.3) $);
\path[draw, <-] ($ (C) + (340:0.3) $) to node[midway, below] {$ \tiny \begin{pmatrix} 0 & 0 \\ 0 & 1 \end{pmatrix} $} ($ (D) + (200:0.3) $);
\path[draw, ->] ($ (F) + (20:0.3) $) to node[midway, above] {$ \tiny \begin{pmatrix} 1 & 0 \\ 0 & 0 \end{pmatrix} $} ($ (G) + (160:0.3) $);
\path[draw, <-] ($ (F) + (340:0.3) $) to node[midway, below] {$ \tiny \begin{pmatrix} 0 & 0 \\ 0 & 1 \end{pmatrix} $} ($ (G) + (200:0.3) $);
\path[draw, ->] ($ (G) + (65:0.3) $) to node[pos=0.8, left] {$ π_1 $} ($ (H) + (205:0.3) $);
\path[draw, <-] ($ (G) + (25:0.3) $) to node[pos=0.6, right] {$ e_2 $} ($ (H) + (245:0.3) $);
\path[draw, <-] ($ (G) + (335:0.3) $) to node[pos=0.6, right] {$ -e_2 $} ($ (I) + (115:0.3) $);
\path[draw, ->] ($ (G) + (295:0.3) $) to node[pos=0.6, left] {$ π_1 $} ($ (I) + (155:0.3) $);
\end{scope}
\path[draw, ultra thick] (3, 0) -- (5, 0);
\path[draw] (5.5, 0) node {\Large $ \cdots $};
\path[draw, ultra thick] (6, 0) -- (8, 0);
\path[draw, ultra thick] (8, 0) -- (10, 3);
\path[draw, ultra thick] (8, 0) -- (10, -3);
\path[draw, ultra thick] (1, -3) -- (3, 0);
\path[fill] (1, -3) circle[radius=0.1] node[left] {$ E_2 $};
\path[fill] (3, 0) circle[radius=0.1] node[below right] {$ O_1 $};
\path[fill] (8, 0) circle[radius=0.1] node[below left] {$ O_{n-3} $};
\path[fill] (10, 3) circle[radius=0.1] node[right] {$ E_3 $};
\path[fill] (10, -3) circle[radius=0.1] node[right] {$ E_4 $};
\end{tikzpicture}
\caption{This figure depicts the $ n $-many one-parameter families of representations which make up the exceptional fiber of the Kleinian $ D_n $ singularity with the specific choice $ θ_2 = (+1, …, +1) $. The thick lines in the figure express the $ D_n $ Dynkin diagram, and each of its $ n $ vertices corresponds to a one-parameter family. The representation treated in \autoref{th:caseD-theta2surj} is depicted in the bottom-left corner. The nodes $ I_{n-4} $ and $ O_{n-3} $ should read $ O_{n-4} $ and $ I_{n-3} $, respectively, if $ n $ is odd instead of even.}
\label{fig:excfiber-smooth-Dn}
\end{figure}
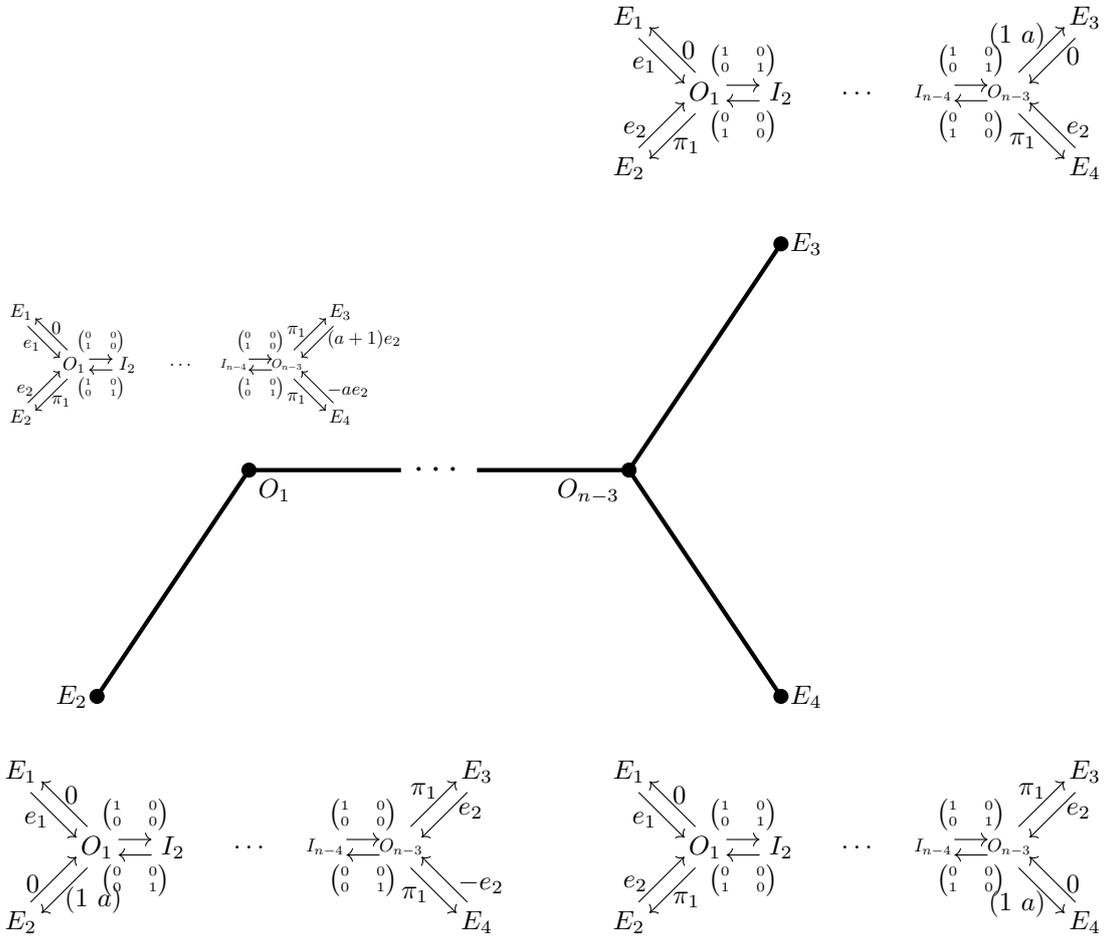

\begin{proposition}
\label{th:caseD-theta2surj}
Let $ ρ ∈ \Rep(Π_Q, α) $ be a $ θ_2 $-semistable element in the nullcone. Then $ ρ $ lies in the image of $ R $.
\end{proposition}

\begin{proof}
Recall that Crawley-Boevey's analysis \cite{cb-kleinian} shows that there are $ n $-many 1-parameter families of representations in the nullcone, each distinguished by their socle. These 1-parameter families are depicted in \autoref{fig:excfiber-smooth-Dn}, following \cite{vdKreeke-Namikawa}. We focus here on the 1-parameter family attached to the $ E_2 $ vertex. In other words, this is the 1-parameter family of representations whose socle equals the simple representation $ S_{E_2} $. The parameter is denoted $ a ∈ ℂ^* $ and we shall denote the representation by $ ρ_a $.

We shall now show that $ ρ_a ∈ \Im(R) $ by constructing $ (φ, x) ∈ \CG_Γ × ℂ^2 $ such that $ R(φ, x) = ρ_a $. Let us start by explaining our ansatz. The ansatz entails choosing $ φ $ as a limit $ φ = \lim_{t → 0} g(t) R φ^{(0)} $ and setting $ x = e_1 $. Here $ g(t) ⊂ \GL $ shall be a 1-parameter subgroup and $ R ∈ \GL $ is an unknown element. Note that every 1-parameter subgroup of $ \GL $ is a conjugate of a diagonal matrix group, whose diagonal elements consist purely of powers $ t^α $ with $ α ∈ ℤ $. Since $ g(t) R φ^{(0)} $ is supposed to converge, it is a reasonable assumption that the exponents $ α $ are all negative. We note that the entire 1-parameter subgroup $ g(t) $ fixes $ φ $:
\begin{equation*}
g(t) φ = \lim_{s → 0} g(t) g(s) R φ^{(0)} = \lim_{s → 0} g(ts) R φ^{(0)} = φ.
\end{equation*}
We conclude that in the limit $ g(t)^{-1} $ takes $ ρ_a $ to the zero representation:
\begin{equation*}
g(t)^{-1} R(φ, x) = R(g(t)^{-1} φ, g(t)^{-1} x) \underset{t → 0}{\longrightarrow} R(φ, 0) = 0.
\end{equation*}
This finishes the description of our ansatz. The remaining task is to identify a 1-parameter subgroup whose inverse in the limit takes $ ρ_a $ to the zero representation. We are allowed a certain flexibility by means of admitting conjugation and finding a suitable element $ R ∈ \GL $.

For the first step, let us find a 1-parameter subgroup $ g(t) ⊂ \GL $ with negative exponents whose inverse in the limit takes $ ρ_a $ to zero. It is easy to find the following candidate:
\begin{equation*}
g(t)^{-1} = \big(\underset{E_1}{t^{2n-2}}, \underset{E_3}{t^{n-2}}, \underset{E_4}{t^{n-2}}, \underset{O_1}{\pmat{t^1 & 0 \\ 0 & t^{2n-3}}}, …, \underset{O_{n-3}}{\pmat{t^{n-3} & 0 \\ 0 & t^{n-1}}}\big).
\end{equation*}
The remainder of the verification entails finding an element $ Q ∈ \GL $ and $ R ∈ \GL $ such that $ φ = \lim_{t →0} Q g(t) Q^{-1} R φ^{(0)} $ exists and $ R(φ, e_1) = ρ_a $. For the case $ n = 4 $, this can be done by hand, departing from the better-suited gauge-equivalent presentation of $ ρ_a $ provided in \cite{vdKreeke-Namikawa}. For the case $ n ≥ 5 $, the author was forced to invoke computer aid. With this aid, we successfully verified the solvability for $ Q $ and $ R $ for low $ n $. In what follows, we describe the result of the calculations for $ n = 4, 5, 6 $. We shall write out the elements $ Q, R ∈ \GL $ as
\begin{align*}
Q &= \big(\underset{E_2}{1}, \underset{E_3}{1}, \underset{E_4}{1}, \underset{O_1}{Q_1}, …, \underset{O_{n-3}}{Q_{n-3}}\big), \\
R &= \big(\underset{E_2}{r_2}, \underset{E_3}{r_3}, \underset{E_4}{r_4}, \underset{O_1}{R_1}, …, \underset{O_{n-3}}{R_{n-3}}\big).
\end{align*}
We are now ready to describe the specific matrices $ Q_i ∈ \GL_2 (ℂ) $ and $ R_i ∈ \GL_2 (ℂ) $ for $ 1 ≤ i ≤ n-3 $ and elements $ r_2, r_3, r_4 ∈ ℂ^* $. For $ n = 4 $ we construct the following matrices:
\begin{align*}
Q_1 &= \pmat{1 & 0 \\ -\frac{1}{2a} & 1}, \\
R_1 &= \pmat{1 & -e^{2πi/8} \\ -i e^{2πi/8} - \frac{1}{2a} & \frac{2^{-1} e^{2πi/8}}{a}}, \\
r_2 &= -a \\
r_3 &= i, \\
r_4 &= -i.
\end{align*}
For $ n = 5 $ we construct the following matrices:
\begin{align*}
Q_1 &= \pmat{1 & 0 \\ -\frac{1}{2a} & 1}, \\
Q_2 &= \pmat{1 & 0 \\ 0 & 1}, \\
R_1 &= \pmat{1 & -i 2^{1/3} \\ -2^{2/3} i - \frac{1}{2a} & \frac{2^{-2/3} i}{a}}, \\
R_2 &= \pmat{1 & -2^{-1/3} \\ 2^{1/3} & 0}, \\
r_2 &= -2a, \\
r_3 &= i, \\
r_4 &= -i.
\end{align*}
For $ n = 6 $ we construct the following matrices:
\begin{align*}
Q_1 &= \pmat{1 & 0 \\ -\frac{1}{2a} & 1}, \\
Q_2 &= \pmat{1 & 0 \\ 0 & 1}, \\
Q_3 &= \pmat{1 & 0 \\ 0 & 1}, \\
R_1 &= \pmat{1 & i e^{2πi/16} 2^{1/2} \\ 2^{3/2} i e^{-2πi/16} - \frac{1}{2a} & -\frac{1}{2a} i e^{2πi/16} 2^{1/2}}, \\
R_2 &= \pmat{1 & -e^{2πi/8} \\ -2i e^{2πi/8} & 0}, \\
R_3 &= \pmat{1 & 2^{-1/2} e^{-2πi/16} \\ - 2^{1/2} e^{2πi/16} & 0}, \\
r_2 &= -4a, \\
r_3 &= i, \\
r_4 &= -i.
\end{align*}
We have checked with computer aid that $ R(Q g(t) Q^{-1} R φ^{(0)}, e_1) → ρ_a $. In the case $ n ≥ 5 $, we have omitted the checks that $ Q g(t) Q^{-1} R φ^{(0)} $ converges and omitted calculations for the other $ (n-1) $-many 1-parameter families of representations. In the case $ n = 4 $, we have explicitly performed the checks that $ φ ≔ \lim_{t → 0} Q g(t) Q^{-1} R φ^{(0)} $ exists and $ R(φ, e_1) = ρ_a $, and also treated the other three 1-parameter families of representations, based on the alternative description of $ ρ_a $ from \cite{vdKreeke-Namikawa}. We finish the proof here.
\end{proof}

\subsection{Symmetry and coherence relations}
In this section, we describe explicitly the symmetry and coherence relations for the Clebsch-Gordan data $ φ ∈ \CG_Γ $ in the $ D_4 $ case. We first deal with the symmetry relations, which are rather easy to state, and after that deal with the coherence relations. This section is intended for illustratory purposes.

We start by explaining the symmetry relations. The data contained in a Clebsch-Gordan datum $ φ ∈ \CG_Γ $ are
\begin{equation}
\label{eq:caseD-coherence-essentials}
\begin{aligned}
& φ_{E_2, E_2}, φ_{E_2, E_3}, φ_{E_2, E_4}, φ_{E_3, E_3}, φ_{E_3, E_4}, φ_{E_4, E_4}, \\
& φ_{E_2, O_1}, φ_{E_3, O_1}, φ_{E_4, O_1}, φ_{O_1, O_1}.
\end{aligned}
\end{equation}
By construction, all other data can be expressed in terms of these. For instance, the maps $ φ_{E_1, •} $ and $ φ_{•, E_1} $ are simply the identity maps for any choice $ • ∈ \{E_1, E_2, E_3, E_4, O_1\} $. Moreover, we have the symmetry relations
\begin{equation}
\label{eq:caseD-coherence-relations}
\begin{aligned}
φ_{E_i, E_j} &= φ_{E_j, E_i} ∘ σ, \quad 1 ≤ i, j ≤ 4, \\
φ_{O_1, E_i} &= φ_{E_i, O_1} ∘ σ, \quad 1 ≤ i ≤ 4, \\
φ_{O_1, O_1} &= \pmat{-1 & 0 & 0 & 0 \\ 0 & 1 & 0 & 0 \\ 0 & 0 & 1 & 0 \\ 0 & 0 & 0 & 1} ∘ φ_{O_1, O_1} ∘ σ.
\end{aligned}
\end{equation}
In all three rows, the letter $ σ $ generically denotes the tensor flip. Let us comment on these symmetry relations. The first two rows hold for $ φ^{(0)} $ by definition and therefore also for any $ φ ∈ \CG_Γ $. Simply speaking, our construction of the Clebsch-Gordan variety has gone the easy way and defined that only the data in the list \eqref{eq:caseD-coherence-essentials} shall matter and all data not contained in the list is obtained from enforcing trivial symmetry. The symmetry relation for $ φ_{O_1, O_1} $ comes down to a self-symmetry statement and is slightly more complicated. Indeed, we observe that $ π_{E_1} φ^{(0)}: O_1 ¤ O_1 → E_1 $ is a skew-symmetric bilinear form and thus remains so after gauging with any $ g ∈ \GL $. Similarly, $ π_{E_i} φ^{(0)}: O_1 ¤ O_1 → E_i $ is a symmetric bilinear form and thus remains so after gauging. This finishes the explanation of all symmetry relations \eqref{eq:caseD-coherence-relations}.

Next we shall explicitly state the coherence relations. We use the language established in \autoref{th:construction-CGvariety-coherence}. For brevity we may write $ φ_{ij} $ for $ φ_{E_i, E_j} $ and similarly $ φ_{i, O_1} $ for $ φ_{E_i, O_1} $. We also write $ σ(ij) $ for the missing index, that is, the index such that $ \{i, j, σ(ij)\} = \{2, 3, 4\} $. We calculate the coherence factors $ γ $ by simply using the datum $ φ = φ^{(0)} $ and forcing the coherence diagram to commute.

The first coherence diagram is the one whose three input representations are one-dimensional. It reads as follows:
\begin{center}
\begin{tikzcd}
E_i ¤ E_j ¤ E_k \arrow[dd, "φ_{ij} ¤ \id"] \arrow[rr, "\id ¤ φ_{jk}"] && E_i ¤ E_{σ(jk)} \arrow[d, "φ_{i, σ(jk)}"] \\
&& E_{σ(i, σ(jk))} \\
E_{σ(ij)} ¤ E_k \arrow[r, "φ_{σ(ij), k}"] & E_{σ(σ(ij), k)} \arrow[ur, "1", "\sim"'] &
\end{tikzcd}
\end{center}
In other words, among the one-dimensional representations the compositions are plainly coherent.
\begin{equation*}
φ_{σ(σ(ij), k)} ∘ (φ_{ij} ¤ \id) = φ_{i, σ(jk)} ∘ (\id ¤ φ_{jk}).
\end{equation*}
Let us now regard the coherence diagrams with $ E_i, E_j, O_1 $:
\begin{center}
\begin{tikzcd}
E_i ¤ E_j ¤ O_1 \arrow[dd, "φ_{ij} ¤ \id"] \arrow[rr, "\id ¤ φ_{j, O_1}"] && E_i ¤ O_1 \arrow[d, "φ_{i, O_1}"] \\
&& E_{O_1} \\
E_{σ(ij)} ¤ O_1 \arrow[r, "φ_{σ(ij), O_1}"] & O_1 \arrow[ur, "γ_{i, j, O_1}", "\sim"'] &
\end{tikzcd}
\end{center}
Evaluating this diagram for all $ i, j ∈ \{2, 3, 4\} $, we obtain the following coherence factors:
\begin{align*}
& γ_{2, 2, O_1} = -\Id, γ_{2, 3, O_1} = +\Id, γ_{2, 4, O_1} = +\Id, \\
& γ_{3, 2, O_1} = -\Id, γ_{3, 3, O_1} = -\Id, γ_{3, 4, O_1} = +\Id, \\
& γ_{4, 2, O_1} = -\Id, γ_{4, 3, O_1} = -\Id, γ_{4, 4, O_1} = +\Id. \\
\end{align*}
Let us now regard the coherence diagrams with $ E_i, O_1, E_j $:
\begin{center}
\begin{tikzcd}
E_i ¤ O_1 ¤ E_j \arrow[dd, "φ_{i, O_1} ¤ \id"] \arrow[rr, "\id ¤ φ_{O_1, j}"] && E_i ¤ O_1 \arrow[d, "φ_{i, O_1}"] \\
&& O_1 \\
O_1 ¤ E_j \arrow[r, "φ_{O_1, j}"] & O_1 \arrow[ur, "γ_{i, O_1, j}", "\sim"'] &
\end{tikzcd}
\end{center}
Evaluating this diagram for all $ i, j ∈ \{2, 3, 4\} $, we obtain the following coherence factors:
\begin{align*}
& γ_{2, 2, O_1} = +\Id, γ_{2, 3, O_1} = -\Id, γ_{2, 4, O_1} = -\Id, \\
& γ_{3, 2, O_1} = -\Id, γ_{3, 3, O_1} = +\Id, γ_{3, 4, O_1} = -\Id, \\
& γ_{4, 2, O_1} = -\Id, γ_{4, 3, O_1} = -\Id, γ_{4, 4, O_1} = +\Id.
\end{align*}
Let us now regard the coherence diagrams with $ E_i, O_1, E_j $:
\begin{center}
\begin{tikzcd}
O_1 ¤ E_i ¤ E_j \arrow[dd, "φ_{O_1, i} ¤ \id"] \arrow[rr, "\id ¤ φ_{ij}"] && O_1 ¤ E_{σ(ij)} \arrow[d, "φ_{O_1, σ(ij)}"] \\
&& O_1 \\
O_1 ¤ E_j \arrow[r, "φ_{O_1, j}"] & O_1 \arrow[ur, "γ_{O_1, i, j}", "\sim"'] &
\end{tikzcd}
\end{center}
Evaluating this diagram for all $ i, j ∈ \{2, 3, 4\} $, we obtain the following coherence factors:
\begin{align*}
& γ_{O_1, 2, 2} = +\Id, γ_{O_1, 2, 3} = -\Id, γ_{O_1, 2, 4} = -\Id, \\
& γ_{O_1, 3, 2} = +\Id, γ_{O_1, 3, 3} = -\Id, γ_{O_1, 3, 4} = -\Id, \\
& γ_{O_1, 4, 2} = +\Id, γ_{O_1, 4, 3} = +\Id, γ_{O_1, 4, 4} = +\Id.
\end{align*}
Let us now regard the coherence diagrams with $ O_1, O_1, E_i $:
\begin{center}
\begin{tikzcd}
O_1 ¤ O_1 ¤ E_i \arrow[dd, "φ_{O_1, O_1} ¤ \id"] \arrow[rrrr, "\id ¤ φ_{O_1, i}"] &&&& O_1 ¤ O_1 \arrow[d, "φ_{O_1, O_1}"] \\
&&&& E_1 ⊕ E_2 ⊕ E_3 ⊕ E_4 \\
(E_1 ⊕ E_2 ⊕ E_3 ⊕ E_4) ¤ E_i \arrow[rrr, "\id ⊕ φ_{2i} ⊕ φ_{3i} ⊕ φ_{4i}"] &&& E_i ⊕ E_{σ(2i)} ⊕ E_{σ(3i)} ⊕ E_{σ(4i)} \arrow[ur, "γ_{O_1, O_1, i}", "\sim"'] &
\end{tikzcd}
\end{center}
Evaluating this diagram for all $ i ∈ \{2, 3, 4\} $, we obtain the following coherence factors:
\begin{align*}
γ_{O_1, O_1, E_2} = \pmat{0 & 1 & 0 & 0 \\ 1 & 0 & 0 & 0 \\ 0 & 0 & 0 & 1 \\ 0 & 0 & 1 & 0},
γ_{O_1, O_1, E_3} = \pmat{0 & 0 & 1 & 0 \\ 0 & 0 & 0 & 1 \\ 1 & 0 & 0 & 0 \\ 0 & 1 & 0 & 0},
γ_{O_1, O_1, E_4} = \pmat{0 & 0 & 0 & 1 \\ 0 & 0 & 1 & 0 \\ 0 & 1 & 0 & 0 \\ 1 & 0 & 0 & 0}.
\end{align*}
Let us now regard the coherence diagrams with $ O_1, E_i, O_1 $:
\begin{center}
\begin{tikzcd}
O_1 ¤ E_i ¤ O_1 \arrow[dd, "φ_{O_1, E_i} ¤ \id"] \arrow[rrrr, "\id ¤ φ_{i, O_1}"] &&&& O_1 ¤ O_1 \arrow[d, "φ_{O_1, O_1}"] \\
&&&& E_1 ⊕ E_2 ⊕ E_3 ⊕ E_4 \\
O_1 ¤ O_1 \arrow[rrr, "φ_{O_1, O_1}"] &&& E_1 ⊕ E_2 ⊕ E_3 ⊕ E_4 \arrow[ur, "γ_{O_1, E_i, O_1}", "\sim"'] &
\end{tikzcd}
\end{center}
Evaluating this diagram for all $ i ∈ \{2, 3, 4\} $, we obtain the following coherence factors:
\begin{align*}
γ_{O_1, E_2, O_1} = \pmat{-1 & 0 & 0 & 0 \\ 0 & -1 & 0 & 0 \\ 0 & 0 & 1 & 0 \\ 0 & 0 & 0 & 1},
γ_{O_1, E_3, O_1} = \pmat{-1 & 0 & 0 & 0 \\ 0 & 1 & 0 & 0 \\ 0 & 0 & -1 & 0 \\ 0 & 0 & 0 & 1},
γ_{O_1, E_4, O_1} = \pmat{-1 & 0 & 0 & 0 \\ 0 & 1 & 0 & 0 \\ 0 & 0 & 1 & 0 \\ 0 & 0 & 0 & -1}.
\end{align*}
Let us now regard the coherence diagrams with $ E_i, O_1, O_1 $:
\begin{center}
\begin{tikzcd}
E_i ¤ O_1 ¤ O_1 \arrow[dd, "φ_{i, O_1} ¤ \id"] \arrow[rrrr, "\id ¤ φ_{O_1, O_1}"] &&&& E_i ¤ (E_1 ⊕ E_2 ⊕ E_3 ⊕ E_4) \arrow[d, "\id ⊕ φ_{σ(i2)} ⊕ φ_{σ(i3)} ⊕ φ_{σ(i4)}"] \\
&&&& E_i ⊕ E_{σ(i2)} ⊕ E_{σ(i3)} ⊕ E_{σ(i4)} \\
O_1 ¤ O_1 \arrow[rrr, "φ_{O_1, O_1}"] &&& E_1 ⊕ E_2 ⊕ E_3 ⊕ E_4 \arrow[ur, "γ_{E_i, O_1, O_1}", "\sim"'] &
\end{tikzcd}
\end{center}
Evaluating this diagram for all $ i ∈ \{2, 3, 4\} $, we obtain the following coherence factors:
\begin{align*}
γ_{E_2, O_1, O_1} = \pmat{0 & 1 & 0 & 0 \\ 1 & 0 & 0 & 0 \\ 0 & 0 & 0 & 1 \\ 0 & 0 & 1 & 0},
γ_{E_3, O_1, O_1} = \pmat{1 & 0 & 0 & 0 \\ 0 & -1 & 0 & 0 \\ 0 & 0 & -1 & 0 \\ 0 & 0 & 0 & 1},
γ_{E_4, O_1, O_1} = \pmat{-1 & 0 & 0 & 0 \\ 0 & 1 & 0 & 0 \\ 0 & 0 & 1 & 0 \\ 0 & 0 & 0 & -1}.
\end{align*}

Let us now regard the coherence diagram with $ O_1, O_1, O_1 $:
\begin{center}
\begin{tikzcd}
O_1 ¤ O_1 ¤ O_1 \arrow[dd, "φ_{O_1, O_1} ¤ \id"] \arrow[rrrr, "\id ¤ φ_{O_1, O_1}"] &&&& O_1 ¤ (E_1 ⊕ E_2 ⊕ E_3 ⊕ E_4) \arrow[d, "\id ⊕ φ_{O_1, 2} ⊕ φ_{O_1, 3} ⊕ φ_{O_1, 4}"] \\
&&&& O_1 ⊕ O_1 ⊕ O_1 ⊕ O_1 \\
O_1 ¤ O_1 \arrow[rrr, "φ_{O_1, O_1}"] &&& O_1 ⊕ O_1 ⊕ O_1 ⊕ O_1 \arrow[ur, "γ_{O_1, O_1, O_1}", "\sim"'] &
\end{tikzcd}
\end{center}
%
%
%
%
%

Evaluating this diagram for all $ i ∈ \{2, 3, 4\} $, we obtain the following coherence factors:
\begin{align*}
γ_{O_1, O_1, O_1} = \pmat{
-I_2/2 & -I_2/2 & +I_2/2 & +I_2/2 \\
+I_2/2 & +I_2/2 & +I_2/2 & +I_2/2 \\
-I_2/2 & +I_2/2 & +I_2/2 & -I_2/2 \\
-I_2/2 & +I_2/2 & -I_2/2 & +I_2/2}.
\end{align*}

%
%
%
%

\subsection{Comparison with the variety $ Z $ of Abdelgadir and Segal}
\label{sec:caseD-comparison}
In this section, we examine the difference between our Clebsch-Gordan variety $ \CG_Γ $ associated with the Kleinian $ D_4 $ case and the variety $ Z $ of Abdelgadir-Segal. In line with our expectations, the open loci $ \CG_Γ^{\open} $ and $ Z^{\open} $ are isomorphic, but against our expectations it turns out that the entire varieties $ \CG_Γ $ and $ Z $ are not isomorphic in a natural way. When dealing with the variety $ Z $, we use our notation $ E_1 $, $ E_2 $, $ E_3 $, $ E_4 $ and $ O_1 $ from \autoref{sec:caseD-step1} instead of $ ℂ $, $ L_1 $, $ L_2 $, $ L_3 $ and $ V $ used by Abdelgadir and Segal.

The section is structured as follows. The first part is to choose bases for the vector spaces $ E_1 $, $ E_2 $, $ E_3 $, $ E_4 $ and $ O_1 $ and their tensor, wedge and symmetric products in order to facilitate explicit calculations. The second part is to reformulate the definition of the variety $ Z $ and its $ \GL $-action with this choice of basis. The third part is to reformulate the definition of the map $ R: Z × ℂ^2 → \Rep(Π_Q, α) $ with this choice of basis. The fourth part is to construct an isomorphism $ Ψ^{\open}: \CG_Γ^{\open} \isoto Z^{\open} $ which makes $ R: \CG_Γ^{\open} × ℂ^2 → \Rep(Π_Q, α) $ equal to $ R: Z^{\open} × ℂ^2 → \Rep(Π_Q, α) $. Finally, we prove that the isomorphism $ Ψ^{\open} $ does not extend to a map $ \CG_Γ → Z $.


As the first part of this section, we choose bases for all the spaces involved in the construction of $ Z $. The four vector spaces $ E_1 $, $ E_2 $, $ E_3 $ and $ E_4 $ are one-dimensional and we simply use an identification with $ ℂ $. In other words, they have a single basis vector which we may write as $ 1 $ or drop from expressions whenever possible. The space $ O_1 $ is two-dimensional and we identify it with $ ℂ^2 $ and basis vectors $ e_1 $ and $ e_2 $. There are several important spaces in the construction of $ Z $, and we shall choose their bases as follows:
\begin{align*}
O_1 ∧ O_1 &= \vspan\{e_1 ∧ e_2\}, \\
\Sym^2 O_1 &= \vspan\{e_1 ¤ e_1, e_1 ¤ e_2, e_2 ¤ e_2\}, \\
∧^2 (\Sym^2 O_1) &= \vspan\{(e_1 ¤ e_1) ∧ (e_1 ¤ e_2), (e_1 ¤ e_1) ∧ (e_2 ¤ e_2), (e_1 ¤ e_2) ∧ (e_2 ¤ e_2)\}.
\end{align*}
We also make the obvious choices of basis for any tensor products between all these spaces. There is a canonical isomorphism $ F: ∧^2 \Sym^2 O_1 → \Sym^2 O_1 ¤ (O_1 ∧ O_1) $ given by
\begin{align*}
F((u_1 ¤ v_1) ∧ (u_2 ¤ v_2)) & = (u_1 ¤ u_2) ¤ (v_1 ∧ v_2) + (v_1 ¤ u_2) ¤ (u_1 ∧ v_2) \\
& ~~ + (u_1 ¤ v_2) ¤ (v_1 ∧ u_2) + (v_1 ¤ v_2) ¤ (u_1 ∧ u_2).
\end{align*}
In terms of our basis, $ F $ takes the following matrix form:
\begin{equation*}
F = \pmat{2 & 0 & 0 \\ 0 & 4 & 0 \\ 0 & 0 & 2}, \quad F^{-1} = \pmat{1/2 & 0 & 0 \\ 0 & 1/4 & 0 \\ 0 & 0 & 1/2}.
\end{equation*}
While our basis choices are lexically canonical, there are apparently still slight inconsistencies. We overcome these by means of a few constant conversion matrices. We shall define these matrices as follows:
\begin{equation*}
S = \pmat{-1 & 0 & 0 \\ 0 & 1 & 0 \\ 0 & 0 & -1}, \quad σ = \pmat{0 & 0 & 1 \\ 0 & 1 & 0 \\ 1 & 0 & 0}.
\end{equation*}

As the second part of this section, we reformulate the variety $ Z $ with our choice of bases. Recall that in the description of Abdelgadir and Segal, a point $ (β, A, B) ∈ Z $ consists of the following data:
\begin{align*}
β &∈ \Hom((O_1 ∧ O_1) ¤ (O_1 ∧ O_1)) → E_2 ¤ E_3 ¤ E_4, \\
α_i &∈ \Hom(E_i^{¤2}, O_1 ∧ O_1), \quad i = 2, 3, 4, \\
B &∈ \Hom(\Sym^2 O_1, E_2 ⊕ E_3 ⊕ E_4).
\end{align*}
These data are subject to three algebraic compatibility conditions \cite[(E1–3)]{Abdelgadir-Segal}. The group $ \GL $ acts naturally on the variety $ Z $ by means of left-multiplication on the codomain and left-inverse-multiplication on the domain of each map $ β $, $ α_i $ or $ B $. The letter $ A $ is a compact way of denoting the matrix
\begin{equation*}
A = \pmat{α_1 & 0 & 0 \\ 0 & α_2 & 0 \\ 0 & 0 & α_3}.
\end{equation*}
We caution the reader about the index discrepancy given by the fact that $ α_i $ is the $ i $-th diagonal entry of the matrix $ A $ but describes a map $ E_{i+1}^{¤2} → O_1 ∧ O_1 $. For the description of $ Z $ we need to translate the morphism $ ∧^2 B ∈ \Hom(∧^2 (\Sym^2 O_1), (E_3 ¤ E_4) ⊕ (E_2 ¤ E_4) ⊕ (E_2 ¤ E_3)) $ in terms of the bases. Start by writing $ B = (B_{ij}) $ as 3×3-matrix in terms of the bases. Then in terms of bases, we have
\begin{equation*}
∧^2 B = \pmat{
B_{21} B_{32} - B_{31} B_{22} & B_{21} B_{33} - B_{31} B_{23} & B_{22} B_{33} - B_{32} B_{23} \\
B_{11} B_{32} - B_{31} B_{12} & B_{11} B_{33} - B_{31} B_{13} & B_{12} B_{33} - B_{32} B_{13} \\
B_{11} B_{22} - B_{21} B_{12} & B_{11} B_{23} - B_{21} B_{13} & B_{12} B_{23} - B_{22} B_{13}}.
\end{equation*}

\begin{remark}
\label{th:caseD-comparison-W2B}
We observe by direct computation that
\begin{equation}
∧^2 B = S \det(B) B^{-T} σ S.
\end{equation}
\end{remark}

\begin{lemma}
\label{th:caseD-comparison-Z}
The variety $ Z ⊂ ℂ × ℂ^3 × ℂ^{3, 3} $ consists of points $ (β, A, B) $ given by a scalar $ β ∈ ℂ $, a diagonal 3×3-matrix $ A $, and a 3×3-matrix $ B $, satisfying the following three conditions:
\begin{itemize}
\item[(E1)] $ B^T A (S B σ S F) = -16 α_1 α_2 α_3 β^2 I_3 $,
\item[(E2)] $ (S B σ S F) B^T A = -16 α_1 α_2 α_3 β^2 I_3 $,
\item[(E3)] $ (∧^2 B) F^{-1} = βAB $,
\end{itemize}
The condition (E3) alone implies $ \det(B) = - 16 β^3 α_1 α_2 α_3 $. Moreover, the $ \GL $-action on $ Z $ is
given by
\begin{align*}
g(β, A, B) &= (gβ, gA, gB), \\
gβ &= g_{E_2} g_{E_3} g_{E_4} \det(g_{O_1})^{-2} β, \\
gA &= \det(g_{O_1}) A \pmat{g_{E_2}^{-2} & 0 & 0 \\ 0 & g_{E_3}^{-2} & 0 \\ 0 & 0 & g_{E_4}^{-2}}, \\
gB &= \pmat{g_{E_2} & 0 & 0 \\ 0 & g_{E_3} & 0 \\ 0 & 0 & g_{E_4}} B G^{¤2}.
\end{align*}
Here $ G = g_{O_1}^{-1} $ and $ G^{¤2} $ is the Gram matrix explicitly given by
\begin{equation*}
G^{¤2} = \pmat{G_{11}^2 & G_{11} G_{12} & G_{12}^2 \\
2 G_{11} G_{21} & G_{11} G_{22} + G_{21} G_{12} & 2 G_{12} G_{22} \\
G_{21}^2 & G_{21} G_{22} & G_{22}^2}.
\end{equation*}
\end{lemma}

\begin{proof}
We divide the proof into three parts. In the first part, we translate the equation \cite[(E3)]{Abdelgadir-Segal} into our choice of bases. In the second part, we derive the equation for $ \det(B) $. In the third part, we explain the translation of the conditions (E1) and (E2). In the fourth part, we comment on the description of the $ \GL $-action.

For the first part, we start from the observation that (E3) version $ ∧^2 B = β AB J^{-1} $ of Abdelgadir and Segal should be correctly read as $ (∧^2 B) K^{-1} = β AB J^{-1} $ where $ K $ is the natural isomorphism $ K: ∧^2 (\Sym^2 O_1) → \Sym^2 O_1^* ¤ (O_1 ∧ O_1)^{¤3} $. As it turns out, in bases $ K^{-1} = F^{-1} J^{-1} $. Therefore the translation of (E3) of Abdelgadir and Segal into a basis version is
\begin{equation*}
(∧^2 B) F^{-1} = βAB.
\end{equation*}
For the second part of the proof, we use \autoref{th:caseD-comparison-W2B} and deduce the following two equations. Using (E3) to compare both right-hand sides gives $ \det(B) = -16 α_1 α_2 α_3 β^3 $, as claimed.
\begin{align*}
\det((∧^2 B) F^{-1}) &= - \det(B)^3 \det(B)^{-1} / 16, \\
\det(βAB) &= β^3 α_1 α_2 α_3 \det(B).
\end{align*}
For the third part of the proof, we derive (E1) and (E2) by rearranging (E3) under the assumption that $ \det(B) ≠ 0 $. Indeed, with the help of \autoref{th:caseD-comparison-W2B} we read
\begin{equation*}
β B^T A = F^{-1} (∧^2 B)^T = F^{-1} S σ \det(B) B^{-1} S = - F^{-1} · 16 β^3 α_1 α_2 α_3 · S σ B^{-1} S.
\end{equation*}
Canceling $ β $ on both sides is possible under the assumption that $ \det(B) ≠ 0 $. Bringing $ F^{-1} S σ B^{-1} S $ to the other side, either on the left or on the right, yields (E1) and (E2) in analogy with the version of Abdelgadir and Segal.

For the fourth part of the proof, we observe that the $ \GL $-action descends in a straightforward manner to our description of $ Z $ by means of bases. Most importantly, the scalar value of $ g_{O_1} ∧ g_{O_1}: O_1 ∧ O_1 → O_1 ∧ O_1 $ is simply $ \det(g_{O_1}) $ and the matrix expression of $ \Sym^2 g_{O_1}^{-1}: \Sym^2 O_1 → \Sym^2 O_1 $ is the mentioned Gram matrix $ G^{¤2} $. This finishes the proof.
\end{proof}

\begin{example}
\label{th:caseD-comparison-Zexp}
We claim that the following data $ (β, A, B) $ defines a point in $ Z $:
\begin{align*}
β &= -1/2, \\
A &= I_3, \\
B &= \pmat{0 & 1 & 0 \\ 1 & 0 & 1 \\ 1 & 0 & -1}, \\
\end{align*}
Indeed, we calculate that
\begin{equation*}
∧^2 B = \pmat{0 & -2 & 0 \\
-1 & 0 & -1 \\
-1 & 0 & 1}.
\end{equation*}
Thus we immediately have the (E3) condition $ (∧^2 B) F^{-1} = βAB $. The conditions (E1) and (E2) follow automatically since $ \det(B) ≠ 0 $. Therefore $ (β, A, B) $ defines a point in $ Z $.
\end{example}

For the third part of this section, we reformulate the definition of the map $ R: Z × ℂ^2 → \Rep(Π_Q, α) $ in terms of the reformulated version of $ Z $ given by \autoref{th:caseD-comparison-Z}. Simply speaking, the description of the map $ R $ by Abdelgadir and Segal is in abstract terms and we shall therefore translate it into matrix terms. The work of Abdelgadir and Segal also uses a different selection of starred/non-starred arrows in the double quiver $ \Qbar $ than in our description of $ \Qbar $ in the general $ D_n $ case. We shall therefore reformulate their map $ R $ into our selection. Recall that our selection and assignment of arrow names is depicted in \autoref{fig:prelim-kleinian}.

\begin{lemma}
\label{th:caseD-comparison-R}
The map $ R: Z × ℂ^2 → \Rep(Π_Q, α) $ sends a point $ ((β, A, B), x) $ to the representation $ ρ ∈ \Rep(Π_Q, α) $ given by
\begin{align*}
ρ(A_1) &= x, \\
ρ(A_1^*) &= 4 α_1 α_2 α_3 β^2 \pmat{x_2 & -x_1}, \\
ρ(A_2) &= - α_1 H ρ(A_2^*)^T, \\
ρ(A_2^*) &= \pmat{x_1 B_{11} + x_2 B_{12} & x_1 B_{12} + x_2 B_{13}}, \\
ρ(A_3^*) &= - α_2 H ρ(A_3)^T, \\
ρ(A_3) &= \pmat{x_1 B_{21} + x_2 B_{22} & x_1 B_{22} + x_2 B_{23}}, \\
ρ(A_4^*) &= α_3 H ρ(A_4)^T, \\
ρ(A_4) &= \pmat{x_1 B_{31} + x_2 B_{32} & x_1 B_{32} + x_2 B_{33}}.
%
\end{align*}
Here $ H = \pmat{0 & -1 \\ 1 & 0} $.
\end{lemma}

\begin{proof}
It is our task to explain why the map $ R $ of Abdelgadir and Segal translates into the given shape upon choice of bases. In fact, the translation is more or less straightforward from the description of Abdelgadir and Segal \cite[section 3]{Abdelgadir-Segal}. The matrix $ H $ is the basis version of the isomorphism $ H: O_1^* ¤ (O_1 ∧ O_1) \isoto O_1 $ given by $ H(f ¤ (u ∧ v)) = f(u) v - f(v) u $. Indeed, in basis elements it reads
\begin{align*}
e_1^* ¤ (e_1 ∧ e_2) &↦ e_2, \\
e_2^* ¤ (e_1 ∧ e_2) &↦ -e_1.
\end{align*}
The scalar multiplicative constants involved in our description are somewhat sketchy details, so we shall provide some comments here on why our translation is acceptable. First, it is checked by hand that $ ρ $ satisfies the preprojective relations
\begin{align*}
&ρ(A_1^*) ρ(A_1) = ρ(A_2^*) ρ(A_2) = 0, \quad ρ(A_3) ρ(A_3^*) = ρ(A_4) ρ(A_4)^* = 0, \\
&ρ(A_1) ρ(A_1^*) + ρ(A_2) ρ(A_2^*) - ρ(A_3^*) ρ(A_3) - ρ(A_4^*) ρ(A_4) = 0.
\end{align*}
The first row is checked easily by noticing that $ v H v^T = 0 $ for any row vector $ v $. The second row relies on the fact that $ (β, A, B) ∈ Z $ satisfy conditions (E1–3). Second, it is easy to see that the map $ R $ defined in the statement is $ \GL $-equivariant. We finish the proof here.
\end{proof}

For the fourth part of the present section, we construct an isomorphism $ Ψ^{\open}: \CG_Γ^{\open} \isoto Z^{\open} $ which makes $ R: \CG_Γ^{\open} × ℂ^2 → \Rep(Π_Q, α) $ equal to $ R: Z^{\open} × ℂ^2 → \Rep(Π_Q, α) $. We start with a brief discussion about the differences and similarities between $ \CG_Γ $ and $ Z $. Indeed, an element $ φ ∈ \CG_Γ $ by construction contains tensor maps such as $ φ_{E_2, O_1}: E_2 ¤ O_1 → O_1 $ and $ φ_{O_1, O_1}: O_1 ¤ O_1 → E_1 ⊕ E_2 ⊕ E_3 ⊕ E_4 $. It looks natural to suggest that from this amount of data we can read off maps $ β, A, B $. For instance, we immediately have a natural candidate $ B: \Sym^2 O_1 → E_2 ⊕ E_3 ⊕ E_4 $ given by
\begin{equation*}
B(v ¤ w) ≔ π_{E_2 ⊕ E_3 ⊕ E_4} (φ_{O_1, O_1} (v ¤ w)).
\end{equation*}
This definition of $ B $ is even well-defined since $ φ_{O_1, O_1} $ satisfies a strict symmetry relation. However, the choices for $ β $ and $ A $ are less obvious. The author has found somewhat natural choices for them, however they are only well-defined on the open locus $ \CG_Γ^{\open} = \GL φ^{(0)} ⊂ \CG_Γ $. We record these choices explicitly as follows:

\begin{lemma}
\label{th:caseD-comparison-candidate}
The map $ Ψ^{\open}: \CG_Γ^{\open} → Z^{\open} $, given by sending an element $ φ $ to the element $ (β, A, B) $ defined below, is well-defined. The map is in fact a $ \GL $-equivariant isomorphism of algebraic varieties.
\begin{align*}
β &= -\frac{1}{2} φ_{E_2, E_3}^{-1} φ_{E_4, E_4}^{-1} (φ_{O_1, O_1 → E_1} (e_1 ¤ e_2))^2  ∈ ℂ, \\
A &= (φ_{O_1, O_1 → E_1} (e_1 ¤ e_2))^{-1} \pmat{φ_{E_2, E_2} & 0 & 0 \\ 0 & φ_{E_3, E_3} & 0 \\ 0 & 0 & φ_{E_4, E_4}}, \\
B &= \pmat{
φ_{O_1, O_1 → E_2} (e_1 ¤ e_1) & φ_{O_1, O_1 → E_2} (e_1 ¤ e_2) & φ_{O_1, O_1 → E_2} (e_2 ¤ e_2) \\
φ_{O_1, O_1 → E_3} (e_1 ¤ e_1) & φ_{O_1, O_1 → E_3} (e_1 ¤ e_2) & φ_{O_1, O_1 → E_3} (e_2 ¤ e_2) \\
φ_{O_1, O_1 → E_4} (e_1 ¤ e_1) & φ_{O_1, O_1 → E_4} (e_1 ¤ e_2) & φ_{O_1, O_1 → E_4} (e_2 ¤ e_2)}.
\end{align*}
In this definition we have used the shorthand notation $ φ_{O_1, O_1 → E_i} ≔ π_{E_i} φ_{O_1, O_1} $. Finally, the map $ Ψ^{\open} $ renders the following diagram commutative:
\begin{equation*}
\begin{tikzcd}
\CG_Γ^{\open} × ℂ^2 \arrow[rd, "R"] \arrow[r, "Ψ^{\open} × \Id_{ℂ^2}", "\sim"'] & Z^{\open} × ℂ^2 \arrow[d, "R"] \\
& \Rep(Π_Q, α)
\end{tikzcd}
\end{equation*}
\end{lemma}

\begin{proof}
We divide the proof into three parts. In the first part, we explain well-definedness of $ Ψ^{\open} $. In the second part, we explain why $ Ψ^{\open} $ is a $ \GL $-equivariant isomorphism. In the third part, we explain why $ Ψ^{\open} $ renders the diagram commutative.

For the first part, let $ g ∈ \GL $ and regard the orbit element $ φ = g φ^{(0)} $. Checking well-definedness of $ Ψ^{\open} (φ) $ comes down to checking that the relevant entries of $ φ $ are nonzero, so that all inverses in the definition of $ β $, $ A $ and $ B $ exist. For instance, the bilinear form $ φ_{O_1, O_1 → E_1} $ is skew-symmetric and therefore descends to a map $ O_1 ∧ O_1 → E_1 $ which scales by $ \det(g_{O_1})^{-1} $ under gauging by $ g ∈ \GL $. More precisely, we have
\begin{equation*}
φ_{O_1, O_1 → E_1} (e_1 ¤ e_2) = \det(g_{O_1})^{-1} φ_{O_1, O_1 → E_1} (e_1 ¤ e_2).
\end{equation*}
In particular, its value is nonzero. Even simpler considerations show that $ φ_{E_2, E_3} $ and $ φ_{E_4, E_4} $ are both nonzero. This proves well-definedness.

For the second part of the proof, let us mention that after digesting the first part it will be obvious to the reader that $ Ψ^{\open} $ is in fact $ \GL $-equivariant. Next, recall that by construction $ \CG_Γ^{\open} = \GL φ^{(0)} $ consists of a single orbit. Moreover, the open set $ Z^{\open} $ also consists of a single orbit \cite[Lemma 3.5]{Abdelgadir-Segal}. This proves the second part.

For the third part of the proof, thanks to the second part, it suffices to check $ R(Ψ(φ^{(0)}), x) = R(φ^{(0)}, x) $. This can be easily done by hand. Both sides are equal to the representation depicted in \autoref{sec:caseD-step1}. The author in fact also confirmed the commutativity for any $ φ = g φ^{(0)} $ with computer aid. This finishes the proof.
\end{proof}

\begin{example}
Let us determine the element $ Ψ^{\open} (φ^{(0)}) $. Its $ B $-value is simply given by extracting the corresponding 3×3-submatrix from $ φ^{(0)}_{O_1, O_1} $:
\begin{equation*}
B = \pmat{0 & 1 & 0 \\ 1 & 0 & 1 \\ 1 & 0 & -1}.
\end{equation*}
Moreover, we immediately calculate that $ β = -1/2 $ and $ A = \Id $. We conclude that $ Ψ^{\open} (φ^{(0)}) $ is simply the element $ (β, A, B) $ described in \autoref{th:caseD-comparison-Zexp}.
\end{example}

Thanks to \autoref{th:caseD-comparison-candidate}, we have a candidate approach for an isomorphism $ Ψ: \CG_Γ → Z $. However, there are many elements $ φ ∈ \CG_Γ $ with vanishing values $ φ_{E_2, E_3} $ or $ φ_{E_4, E_4} $. The description of $ Ψ $ provided in \autoref{th:caseD-comparison-candidate} is therefore not suited to define an isomorphism $ Ψ: \CG_Γ → Z $. It is an intriguing question whether there exist alternative expressions which do not make use of inverses. The reader may wish to draw the analogy to the trivial example that a map on the hyperbola $ \{xy = 1\} ⊂ ℂ^2 $ may be defined by $ f(x, y) = 1/x $ but alternatively as $ f(x, y) = y $. Despite the high hopes, we shall now prove that the answer to the isomorphism question is negative.

\begin{proposition}
There does not exist a morphism of algebraic varieties $ Ψ: \CG_Γ → Z $ which renders the following diagram commutative:
\begin{equation*}
\begin{tikzcd}
\CG_Γ × ℂ^2 \arrow[rd, "R"] \arrow[r, "Ψ × \Id_{ℂ^2}", "\sim"'] & Z × ℂ^2 \arrow[d, "R"] \\
& \Rep(Π_Q, α)
\end{tikzcd}
\end{equation*}
\end{proposition}

\begin{proof}
In the first part of the proof, we show that any such morphism $ Ψ $ would necessarily agree with $ Ψ^0 $ on the subset $ \CG_Γ^{\open} ⊂ \CG_Γ $. In the second part of the proof, we show that there is however no extension of $ Ψ^{\open} $ to a morphism of varieties $ \CG_Γ → Z $. Jointly, these arguments close the proof.

For the first part, assume there exists a morphism $ Ψ $ which renders the diagram commutative. We shall prove that on $ \CG_Γ^{\open} $ it agrees with $ Ψ^{\open} $. Pick an arbitrary $ φ ∈ \CG_Γ^{\open} $. Now by commutativity and \autoref{th:caseD-comparison-candidate}, for every $ x ∈ ℂ^2 $ we have
\begin{equation}
\label{eq:caseD-comparison-repcompare}
R(Ψ(φ), x) = R(φ, x) = R(Ψ^{\open} (φ), x).
\end{equation}
We shall now prove that $ Ψ(φ) = Ψ^{\open} (φ) $, essentially by explaining how to read off the $ β $, $ A $ and $ B $ values from the quiver representation on the left-hand side and the right-hand side of \eqref{eq:caseD-comparison-repcompare}. Indeed, by comparing the outward-pointing arrows to $ E_2 $, $ E_3 $ and $ E_4 $ for various $ x ∈ ℂ^2 $ we immediately deduce that the $ B $-value of $ Ψ(φ) $ and $ Ψ^{\open} (φ) $ agrees. Since $ Ψ^{\open} (φ) ∈ Z^{\open} $, we conclude $ \det(B) ≠ 0 $ and therefore also $ Ψ(φ) ∈ Z^{\open} $. By comparing the inward-pointing arrows from $ E_2 $, $ E_3 $ and $ E_4 $ and using $ \det(B) ≠ 0 $, we deduce that all $ α_i $-values agree. Since $ Ψ(φ) ∈ Z^{\open} $, they are necessarily nonzero. Finally, by comparing the outward-pointing arrow to $ E_1 $ and using that $ α_1, α_2, α_3 ≠ 0 $, we deduce that the $ β $-value agrees. This shows $ Ψ(φ) = Ψ^{\open} (φ) $.

For the second part, let us assume that $ Ψ $ is an extension of $ Ψ^{(0)} $ to a morphism of algebraic varieties $ \CG_Γ → Z $. Let us define the following algebraic two-parameter family of Clebsch-Gordan elements $ φ^{•, •}: ℂ^2 → \CG_Γ $:
\begin{align*}
φ^{a, b}_{E_i, E_j} &= 0, \quad (i, j) ∈ \{(2, 2), (2, 3), (2, 4), (3, 3), (4, 4)\}, \\
φ^{a, b}_{E_3, E_4} &= a, \\
φ^{a, b}_{E_i, O_1} &= 0, \\
φ^{a, b}_{O_1, O_1} &= \pmat{0 & 0 & 0 & 0 \\ 0 & b & b & 0 \\ 0 & 0 & 0 & 0 \\ 0 & 0 & 0 & 0}.
\end{align*}
Let us explain that the element $ φ^{a, b} $ indeed lies in $ \CG_Γ $. Indeed, for $ a, b ≠ 0 $ we have
\begin{align*}
φ^{a, b} &= \lim_{t → 0} g(t) φ^{(0)}, \\
g(t) &= (a t^{-2}, t^{-1}, t^{-1}, \pmat{ab^{-1} t^{-1} & 0 \\ 0 & t^{-1}}) ∈ \GL.
\end{align*}
Since $ φ^{a, b} ∈ \CG_Γ $ for $ a, b ≠ 0 $ and $ \CG_Γ $ is closed, we deduce that $ φ^{a, b} ∈ \CG_Γ $ for any $ (a, b) ∈ ℂ^2 $. This establishes the algebraic morphism $ φ^{•, •}: ℂ^2 → \CG_Γ $. Now assuming that $ Ψ: \CG_Γ → Z $ is an extension of $ Ψ^{\open} $, we regard the composition
\begin{equation*}
π_β ∘ Ψ ∘ φ^{•, •}: ℂ^2 → \CG_Γ → Z → ℂ.
\end{equation*}
This is a morphism of algebraic varieties as well. We calculate for $ a, b ≠ 0 $ that
\begin{align*}
π_β (Ψ(φ^{a, b})) &= π_β (Ψ^{\open} (φ^{a, b})) \\
&= \lim_{t → 0} \big(-\frac{1}{2} (a^{-1} t^2)^{-1} t^{-2} (a^{-1} b t^2)^2\big) \\
&= -\frac{b^2}{2a}.
\end{align*}
Evidently, this cannot converge as $ a → 0 $. We have produced a contradiction, effectively proving that $ Ψ^{\open} $ has no extension to an algebraic morphism $ \CG_Γ → Z $. This finishes the proof.
\end{proof}

\begin{remark}
From the proof we conclude that the existence of a natural isomorphism $ Ψ: \CG_Γ → Z $ fails grossly. The failure is not simply due to different signs or conventions, but due to different “scaling” behavior of the varieties as the Clebsch-Gordan data $ φ $, or $ β, A, B $, approach zero. Given this difference, it is surprising that $ \CG_Γ × ℂ^2 $ and $ Z × ℂ^2 $ have at least two identical GIT quotients.
\end{remark}

\printbibliography

@misc{bellamy-schedler,
author = {{Bellamy}, G. and {Schedler}, T.},
title = "{Symplectic resolutions of Quiver varieties and character varieties}",
note  = "\url{math.AG/1602.00164}",
year = 2016,
}

@article{king-quivers,
author={{King}, A.~D.},
title="Moduli of representations of finite dimensional algebras",
journal="Quart.\ J.\ Math.\ Oxford",
year=1994,
volume=45,
number=2,
pages="515–530"
}

@article{cb-kleinian,
author={{Crawley-Boevey}, W.},
title="On the exceptional fibres of Kleinian singularities",
journal="Amer.\ J.\ Math.",
volume=122,
number=5,
year=2000,
pages="1027–1037"
}

@article{kac,
author={{Kac}, V.},
title="Infinite root systems, representations of graphs and invariant theory",
journal="J.\ Algebra",
volume=78,
year=1982,
pages="141–162"
}

@misc{Abdelgadir-Segal,
Author = {Tarig Abdelgadir and Ed Segal},
Title = {The McKay correspondence in type $D_4$ via VGIT},
Year = {2024},
Eprint = {arXiv:2402.05763},
}

@misc{vdKreeke-Namikawa,
	author={Jasper van de Kreeke},
	title={Calabi-Yau techniques for Namikawa-Weyl groups},
	year={2025},
	eprint={XXXX.YYYYY},
	archivePrefix={arXiv}
}

@article{Lawrence-Nekrasov-Vafa,
author = {Lawrence, Albion and Nekrasov, Nikita and Vafa, Cumrun},
copyright = {1998},
issn = {0550-3213},
journal = {Nuclear Physics B},
language = {eng ; jpn},
number = {1},
pages = {199-209},
publisher = {Elsevier B.V},
title = {On conformal field theories in four dimensions},
volume = {533},
year = {1998},
}

@misc{Lurie-tannaka,
      title={Tannaka Duality for Geometric Stacks}, 
      author={Jacob Lurie},
      year={2005},
      eprint={math/0412266},
      archivePrefix={arXiv},
      primaryClass={math.AG},
      url={https://arxiv.org/abs/math/0412266}, 
}

@article{Kuznetsov,
    AUTHOR = {Kuznetsov, Alexander},
     TITLE = {Lefschetz decompositions and categorical resolutions of
              singularities},
   JOURNAL = {Selecta Math. (N.S.)},
  FJOURNAL = {Selecta Mathematica. New Series},
    VOLUME = {13},
      YEAR = {2008},
    NUMBER = {4},
     PAGES = {661--696},
}

@misc{vandenBergh-NCresolutions,
      title={Non-commutative crepant resolutions, an overview}, 
      author={Michel Van den Bergh},
      year={2022},
      eprint={2207.09703},
      archivePrefix={arXiv},
      primaryClass={math.AG},
}

@article {vandenBergh-flops,
    AUTHOR = {Van den Bergh, Michel},
     TITLE = {Three-dimensional flops and noncommutative rings},
   JOURNAL = {Duke Math. J.},
  FJOURNAL = {Duke Mathematical Journal},
    VOLUME = {122},
      YEAR = {2004},
    NUMBER = {3},
     PAGES = {423--455},
}

@book {Fulton-toric,
    AUTHOR = {Fulton, William},
     TITLE = {Introduction to toric varieties},
    SERIES = {Annals of Mathematics Studies},
    VOLUME = {131},
      NOTE = {The William H. Roever Lectures in Geometry},
 PUBLISHER = {Princeton University Press, Princeton, NJ},
      YEAR = {1993},
     PAGES = {xii+157},
}

@article {Bridgeland-King-Reid,
    AUTHOR = {Bridgeland, Tom and King, Alastair and Reid, Miles},
     TITLE = {The {M}c{K}ay correspondence as an equivalence of derived
              categories},
   JOURNAL = {J. Amer. Math. Soc.},
  FJOURNAL = {Journal of the American Mathematical Society},
    VOLUME = {14},
      YEAR = {2001},
    NUMBER = {3},
     PAGES = {535--554},
}

@article {Crawley-Boevey-Holland-deformed-preprojective,
    AUTHOR = {Crawley-Boevey, William and Holland, Martin P.},
     TITLE = {Noncommutative deformations of {K}leinian singularities},
   JOURNAL = {Duke Math. J.},
  FJOURNAL = {Duke Mathematical Journal},
    VOLUME = {92},
      YEAR = {1998},
    NUMBER = {3},
     PAGES = {605--635},
      ISSN = {0012-7094,1547-7398},
   MRCLASS = {14B07 (16G10)},
  MRNUMBER = {1620538},
MRREVIEWER = {Michel\ Van den Bergh},
       DOI = {10.1215/S0012-7094-98-09218-3},
       URL = {https://doi.org/10.1215/S0012-7094-98-09218-3},
}

\bigskip

\begin{tabular}{@{}l@{}}%
    \textsc{Department of mathematics, University of California, Berkeley, USA}\\
    \textit{jasper.kreeke@berkeley.edu}
  \end{tabular}

\end{document}